\documentclass[11pt]{amsproc}
\usepackage{thesis,graphicx}

\begin{document}
%***************************************************************

%***************************************************************
\begin{center}

    {\Huge Closest Spacing of Eigenvalues}

    \vspace{1.85in}

    {\Large Jade P. Vinson}

    \vspace{1.45in}

    {\normalsize A DISSERTATION \\
      PRESENTED TO THE FACULTY \\
      OF PRINCETON UNIVERSITY \\
      IN CANDIDACY FOR THE DEGREE \\
      OF DOCTOR OF PHILOSOPHY\\}

    \vspace{1.1in}

    {\normalsize RECOMMENDED FOR ACCEPTANCE \\
      BY THE DEPARTMENT OF \\
      MATHEMATICS\\}

    \vspace{0.5in}

    {\normalsize JUNE, 2001}

\end{center}

%***************************************************************
\newpage

\begin{center} {\bf Abstract} \end{center}

We consider several variants of the following problem: pick an
\nbyn matrix from some unitary ensemble of random matrices.  In an
interval containing many eigenvalues, what is the closest spacing
between two eigenvalues? We are interested in the
correct scaling for this random variable as \nlim, and the
limiting distribution of the rescaled random variable.

One can predict the distribution of the minimum spacing
heuristically by assuming that the consecutive spacings are
independent random variables, chosen from the consecutive spacing
distribution for unitary ensembles.  The consecutive spacings are
{\it not} independent.  However, in all cases studied, the
heuristics predict asymptotically the correct scaling and
distribution of the minimum spacing.

Using the method of moments, we show that the number of eigenvalue
pairs in a given interval and closer than a small distance
$\gamma$, is approximately a Poisson random variable with mean as
predicted by heuristics.  Varying $\gamma$, we obtain the result
for minimum spacing.

In the most concrete special case, our Main Theorem is this:

\begin{thm}[Main Theorem]
  Choose a random $n\times n$ unitary matrix.
  Let $Z_n$ be the closest spacing between any two eigenvalues.
  Fix $\beta>0$.  Then as $n\to\infty$,
  \begin{equation*}
    \Pr\left( Z_n \left(
          \frac{n^4}{72\pi}
        \right)^{\frac13}
      > \beta \right)
    \to e^{-\beta^3}.
  \end{equation*}
\end{thm}

The correct scaling for the closest eigenvalue
spacing is $n^{-1} (n|I_n|)^{-\frac13}$.  In this case, $n^{-1}$
is comparable to the mean spacing and $n|I_n|$ is comparable to
the expected number of eigenvalues in the interval being
considered. On the other hand, if the eigenvalues were independent
and distributed according to a Poisson process, then the closest
spacing would scale like $n^{-1}(n |I_n|)^{-1}$, which is much
closer.  Thus we have another confirmation of what is often said:
``The eigenvalues of a unitary matrix repel each other''.

Our results also apply to the Gaussian unitary ensemble (GUE) and,
with restrictions, to a universal class of unitary ensembles (UUE)
studied by Deift, Kreicherbauer, McLaughlin, Venakides, and Zhou.
In each of these cases, the expected number of eigenvalues in the
interval must be large, and the interval must be contained inside
the bulk distribution of the eigenvalues.

%***************************************************************
\newpage

\begin{center} {\bf Acknowledgements} \end{center}

I would like to thank my advisor, Peter Sarnak, for suggesting
the problem of minimum spacing and an approach using
the method of moments.  Peter Sarnak always shared his
advice on mathematical problems, and exercised good judgement
with regard to his role in my studies.  At first he played an
active role, suggesting topics to think
about and articles to read, but during
the last year his involvement tapered off rapidly as soon
as the minimum spacing problem looked promising.
I think that this is the correct balance.

Financial support was provided by the National Defense
Science and Engineering Graduate Fellowship Program.

I find that writing a sincere Acknowledgement is
an extremely difficult task, and it would take days of
reflection to do it well.  Therefore instead of attempting
to thank the many people who have influenced my
mathematical life, I will limit the scope of this Acknowledgement
to the dissertation itself.

I would like to thank my family for encouraging me to complete the Ph.D.
in time for graduation in June.  Assuming the final public oral
in a week is successful, I will be able to relax for the first summer of
my adult life.  I would also like to thank my colleagues
at the Institute for Defense Analyses for encouraging me to finish my
Ph.D. now rather than later.

The dissertation was completed with little time to spare.  I would
like to thank Peter Sarnak and Yakov Sinai for reading it
and making suggestions.  I would also like to thank Scott Kenney,
the department manager, for walking me through the formalities
required to schedule a final public oral.

%***************************************************************
\newpage

%***************************************************************
\section{Introduction}
%***************************************************************

We consider several variants of the following problem: pick an
\nbyn matrix from some unitary ensemble of random matrices.  In an
interval containing many eigenvalues, what is the closest
spacing between two eigenvalues? We
are interested in the correct scaling for this random variable
as \nlim, and the limiting distribution of the rescaled
random variable.

One can predict the distribution of the minimum spacing
heuristically by assuming that the consecutive spacings are
independent random variables, chosen from the consecutive spacing
distribution for unitary ensembles.  The consecutive spacings are
{\it not} independent.  However, in all cases studied, the
heuristics predict asymptotically the correct scaling and
distribution of the minimum spacing.

Using the method of moments, we show that the number of eigenvalue
pairs in a given interval and closer than a small distance
$\gamma$, is approximately a Poisson random variable with mean as
predicted by heuristics.  Varying $\gamma$, we obtain the result
for minimum spacing.

Our results apply to the circular unitary ensemble (CUE), the Gaussian
unitary ensemble (GUE), and with restrictions to a universal class of
unitary ensembles (UUE) studied by
Deift, Kreicherbauer, McLaughlin, Venakides, and Zhou.
In each of these cases, the expected number of eigenvalues in the interval
must be large, and the interval must be contained inside the bulk distribution
of the eigenvalues.

%***************************************************************
\subsection{Statement of the Main Theorem}

The Main Theorem is easiest to state in the case of
the circular unitary ensemble (CUE),
which is the compact group $U_n$ with Haar measure.

\begin{thm}
  For each $n$, let $I_n \subset [0,2\pi)$ be an interval
  so that $n|I_n|\to\infty$ as $n\to\infty$.

  Choose a matrix randomly from $CUE_n$ and let $Z_n$ be the
  closest spacing between any two eigenvalues whose average is
  in $I_n$.  Fix $\beta>0$.  Then as $n\to\infty$,
  \begin{equation*}
    \Pr\left( Z_n \left(
          \frac{n^4|I_n|}{144\pi^2}
        \right)^{\frac13}
      > \beta \right)
    \to e^{-\beta^3}.
  \end{equation*}
\end{thm}

Of particular interest is the case when $I_n=[0,2\pi)$ for all $n$:

\begin{thm}[Main Theorem, CUE Version] \label{mainCUE}
  Choose a matrix randomly from $CUE_n$ and let $Z_n$ be the
  closest spacing between any two eigenvalues.
  Fix $\beta>0$.  Then as $n\to\infty$,
  \begin{equation*}
    \Pr\left( Z_n \left(
          \frac{n^4}{72\pi}
        \right)^{\frac13}
      > \beta \right)
    \to e^{-\beta^3}.
  \end{equation*}
\end{thm}

Given Theorem~\ref{mainCUE} and its Corollary, observe that the
correct scaling for the closest eigenvalue spacing is $n^{-1}
(n|I_n|)^{-\frac13}$.  In this case, $n^{-1}$ is comparable to the
mean spacing and $n|I_n|$ is comparable to the expected number of
eigenvalues in the interval being considered.  On the other hand,
if the eigenvalues were independent and distributed according to a
Poisson process, then the closest spacing would scale like
$n^{-1}(n |I_n|)^{-1}$, which is much closer.  Thus we have
another confirmation of what is often said: ``The eigenvalues of a
unitary matrix repel each other''.

The universal unitary ensembles (UUE) are not completely standard.
Given a real analytic potential $V(x)$ with sufficient growth at infinity,
$UUE_n$ is the ensemble of Hermitian matrices with
the following joint probability density function (j.p.d.f.)
for the matrix entries:

\begin{equation*}
  (\mbox{const})\ \  e^{-n \sum V(\lambda_j)} dM
\end{equation*}

Using Weyl integration (Appendix~\ref{WeylAppendix}), the j.p.d.f. for the eigenvalues is
\begin{equation*}
  (\mbox{const})\ \
  \prod_{i<j} (\lambda_i - \lambda_j)^2
  e^{-n\sum V(\lambda_i)} d\Lambda
\end{equation*}

One could recover the Gaussian Unitary Ensemble (GUE)
as a special case of UUE by choosing the potential $V(x)=x^2$
and rescaling properly.  The word ``unitary'' in this context confuses
many newcomers to the field of random matrices.  The GUE is a
probability distribution on the set of {\it Hermitian} matrices
which is invariant under conjugation by any {\it unitary} matrix.

Our main result, in the case of universal ensembles, is the following:

\begin{thm}[Main Theorem, universal version]
  Let $V(x)$ be a real analytic potential which is regular and
  whose equilibrium measure $\Psi(x) dx$ is supported on a single interval
  $[a,b]$.  Fix $\epsilon > 0$.  For each $n$, let
  $I_n \subset [a+\epsilon,b-\epsilon]$ be an interval
  contained in the bulk distribution of eigenvalues, so that
  $n|I_n|\to\infty$ as $n\to\infty$.

  Choose a matrix randomly from $UUE_n$ and let $Z_n$ be
  the closest spacing between any two eigenvalues whose average is
  in $I_n$.  Fix $\beta>0$.  Then as $n\to\infty$,
  \begin{equation*}
    \Pr\left( Z_n \left(
          \frac{\pi^2 n^4}{9} \int_{I_n} \Psi(x)^4 dx
        \right)^{\frac13}
      > \beta \right)
    \to e^{-\beta^3}.
  \end{equation*}
\end{thm}

%***************************************************************
\subsection{Overall Strategy}

We analyze the random variable $Z_n$ indirectly.  Let $\G$ count
the number of eigenvalue pairs whose average is in
$I_n$ and whose separation is at most $\gamma$.  For example, if
$\lambda_1<\lambda_2<\lambda_3<\lambda_4 <
\lambda_1 + \gamma$, then this contributes $6$ pairs to $\G$.
For each $\gamma$ we analyze the random variable $\G$ in detail,
but what interests us most is the probability $\Pr(\G = 0)$:
\begin{eqnarray*}
  (Z_n>\gamma) & \iff & (\G = 0) \\
  \Pr(Z_n>\gamma) & = & \Pr(\G = 0)
\end{eqnarray*}

Since $\G$ is an integer valued random variable, $\Pr(\G=0)$ is
accessible from the moments of $\G$. The following Theorem says
that the moments of $\G$ are approximately the moments of a
Poisson distribution.

\begin{thm}[Moment Estimation Theorem, universal version] \label{METUUE}
  Let $V(x)$ be a real analytic potential which is regular and
  whose equilibrium measure $\Psi(x) dx$ is supported on a single interval
  $[a,b]$.  Fix $\epsilon > 0$.  For each $n$, let
  $I_n \subset [a+\epsilon,b-\epsilon]$ be an interval
  contained in the bulk distribution of eigenvalues.

  For each $n$, let $\gamma_n>0$.
  Let $\G$ be the random variable which counts the number of $UUE_n$ eigenvalues
  whose average is in $I_n$ and whose difference is at most
  $\gamma_n$.  Let $G_{\mu}$ be the Poisson distribution with mean
  \begin{equation*}
    \mu_n = \frac{\pi^2\gamma_n^3 n^4}{9} \int_{I_n} \Psi(x)^4 dx.
  \end{equation*}
  Then for all $k\geq 1$, as $n\to \infty$,
  \begin{equation*}
    E(\G^k) = E(G_{\mu}^k) \left(1+ \bigO\left(n^{-1} + \gamma^2n^2
        + (\gamma^2n^2)^3 + (n|I|)^{-\frac23} \right) \right).
  \end{equation*}
  The constant implied by $\bigO$ depends only on $k$, the potential
  $V(x)$, and $\epsilon$.
\end{thm}

Proving the Moment Estimation Theorem is the major task of this
paper. First, let us assume the Moment Estimation Theorem and use
it to prove the Main Theorem.

%***************************************************************
\subsection{Proof that the Moment Estimation Theorem implies the Main Theorem}

In applying the Moment Estimation Theorem, we get to choose
$\gamma_n$.  We choose:
\begin{eqnarray*}
  \gamma_n & = & \beta \left( \frac{\pi^2 n^4}{9}
      \int_{I_n} \Psi(x)^4 dx \right)^{-\frac13} \\
  & \propto & n^{-\frac43}|I_n|^{-\frac13}.
\end{eqnarray*}
With this choice of $\gamma_n$, $\mu_n = \beta^3$ for all $n$.
Applying the Moment Estimation Theorem,
\begin{eqnarray*}
  E(\G) & = & E(G_{\beta^3}^k)
    \left( 1+ \bigO\left((n |I_n|)^{-\frac23} \right) \right).
\end{eqnarray*}
By assumption, $(n|I_n|) \to \infty$ as $n\to\infty$, so the
moments of $\G$ converge to those of $G_{\beta^3}$.  Thus,
\begin{equation*}
  \Pr(\G = 0) \to \Pr(G_{\beta^3} = 0) = e^{-\beta^3}
    \qquad \mbox{as} \qquad n\to\infty.
\end{equation*}
We retrace our steps back to the minimum spacing $Z_n$.  The
random variable $Z_n$ is greater than $\gamma_n$ if and only if
$\G=0$.  Substituting in the chosen value of $\gamma_n$, this is
equivalent to
\begin{equation*}
  Z_n \left(
          \frac{\pi^2 n^4}{9} \int_{I_n} \Psi(x)^4 dx
        \right)^{\frac13}
      > \beta,
\end{equation*}
completing the proof.

%***************************************************************
\subsection{Outline of Paper}

In the remainder of this paper we prove the Moment Estimation
Theorem.  We first prove the case of CUE because this case
is by far the easiest.  The novel calculations for CUE are
\begin{itemize}
  \item The combinatorial methods used enumerate the
    various contributions to the moment $E(\G)$.  Each contribution
    is expressed as an integral using the method of Gaudin.
  \item Calculation of the main contributions using asymptotics
    for the projection kernel $K_n(x,y)$ and its derivatives.
  \item Bounds for the undesired contributions, which are implied
    by bounds for $K_n(x,y)$ and its derivatives.
\end{itemize}
For $CUE$, the projection kernel is very easy to work with:
\begin{equation*}
  K_n(x,y) = \frac{1}{2\pi} \frac{e^{in(x-y)}-1}{e^{i(x-y)}-1}.
\end{equation*}
See Appendix~\ref{intoutCUEsubsec}.

We prove the UUE version of the Moment Estimation Theorem using exactly the
same techniques.  Parts of the proof are identical to the CUE case, and are
omitted.  The new feature the UUE case is that the projection
kernel $K_n(x,y)$ is less easy to work with:
\begin{eqnarray*}
  K_n(x,y) & = & \sum_{j=0}^{n} \eta_j(x) \eta_j(y) \\
  & = & (\mbox{const}) \frac{\eta_n(x)\eta_{n-1}(y) - \eta_n(y) \eta_{n-1}(x)}{x-y}.
\end{eqnarray*}
where $\phi_j(x)$ is the $(j)$th normalized orthogonal polynomial
with respect to the weight $e^{-nV(x)}$ and
$\eta_j(x) = e^{-\frac{n}{2}V(x)} \phi_j(x)$.

The asymptotics for the orthogonal polynomials $\phi_{n-1}$ and $\phi_n$
were obtained recently by Deift, Kreicherbauer, McLaughlin, Venakides and Zhou.
Based closely on~\cite{deiftzhou4} and~\cite{deift_book},
we outline the derivation of the leading order asymptotics.
By a very minor modification of these techniques,
we obtain the leading order asymptotics for the
derivatives of $\eta_{n-1}$ and $\eta_n$ -- which, as one would expect, are
the derivatives of the leading order asymptotics.

The case of the Gaussian unitary ensemble is of special interest.
The analogs of the Main Theorem and the Moment Estimation Theorem
could be obtained by rescaling the GUE and applying the universal
results for the potential $V(x) = x^2$.  We present an alternate proof of the GUE Theorems,
using Plancherel-Rotach asymptotics for
Hermite polynomials in place of the more general Deift-Zhou asymptotics
for orthogonal polynomials.

Given our results about the closest spacing for eigenvalues in an interval,
it is natural to ask about the distribution of the maximum spacing between
consecutive eigenvalues in some interval.  We have not studied this problem
in great detail, but believe it is more difficult than the minimum spacing
problem.  At least we have not been able to deal with it yet.
In section~\ref{MaxSpace},
we outline one possible approach to studying the maximum spacing.

We call the reader's attention to~\cite{edelman-condition} and to related
papers available on Edelman's web page.
In these paper Edelman finds the
correct scaling and distribution of the condition number
\begin{equation*}
  \frac{\sqrt{\sum |\lambda_j|^2}}{|\lambda_1|},
\end{equation*}
where $\lambda_1$ is the complex eigenvalue with smallest absolute value.
The condition number is an indicator of the difficulty of inverting a
matrix numerically.
Edelman considers the ensemble of matrices whose
real or complex entries are chosen independently from Gaussians.
Since the numerator is very tightly
distributed because of the law of large numbers, Edelman's result
concerns the distribution of the smallest eigenvalue for a matrix
in this ensemble.

%***************************************************************
\section{Case 1: The Circular Unitary Ensemble (CUE)}
%***************************************************************

We now state and prove the Moment Estimation Theorem for the circular unitary ensemble.
The Moment Estimation Theorem for the CUE case implies the Main Theorem for the CUE case,
which was already stated in the introduction.

\begin{thm}[Moment Estimation Theorem, CUE version] \label{METCUE}
  For each $n$, let $\gamma_n>0$ and $I_n \subset [0,2\pi)$.

  Let $\G$ be the random variable which counts the number of $CUE_n$ eigenvalues
  whose average is in $I_n$ and whose difference is at most
  $\gamma_n$.  Let $G_{\mu}$ be the Poisson distribution with mean
  \begin{equation*}
    \mu_n = \frac{|I_n| \gamma_n^3 n^4}{144 \pi^2}
  \end{equation*}
  Then for all $k\geq 1$, as $n\to \infty$,
  \begin{equation*}
    E(\G^k) = E(G_{\mu}^k)
      \left(1+ \bigO\left(
          n^{-1} + \gamma^2n^2 + (\gamma^2n^2)^3 + (n|I|)^{-\frac23}
        \right)
      \right).
  \end{equation*}
  The constant implied by $\bigO$ depends only on $k$.
\end{thm}

Before proving this Theorem, we will
show that it agrees with heuristic predictions.

%***************************************************************
\subsection{Heuristic Prediction Based on the Consecutive Spacing Distribution}
\label{heuristicCUE}

We make the simplifying assumption that the consecutive spacings
are independent random variables, all chosen from the consecutive
spacing distribution for unitary ensembles.  This assumption of
independence is false, since adjacent consecutive spacings are
anti-correlated.  See p. 111 of~\cite{mehta} for a contour
plot of the joint probability density function of two adjacent
consecutive spacings.  Observe that short spacings tend to be
followed by longer ones and vice-versa.  However, we will see that
the independence assumption does lead to correct predictions for
the minimum spacing.

For the unitary ensembles, the consecutive spacing distribution vanishes to order
two at the origin, so that very short consecutive spacings are unlikely and the
eigenvalues are said to ``repel.''  See~\cite{mehta} Chapter~5 and Appendix~13.
For a power series expansion of this
density function at the origin, the first few terms are:
\begin{equation*}
  p(x) = \frac{\pi^2}{3} x^2 - \frac{2\pi^4}{45} x^4 + \frac{\pi^6}{315} x^6 - \dots
\end{equation*}

When a matrix is chosen from $U_n$, the mean spacing of the
eigenvalues is the constant $\frac{2\pi}{n}$, independent of $\theta$.
Thus a separation of $\gamma$ is
equal to $\left( \frac{n\gamma}{2\pi} \right)$ times the mean spacing.  When
$\left(\frac{n\gamma}{2\pi}\right)$ is small,
the probability of any one of these spacings being less
than $\gamma$ is about
$(\frac{\pi^2}{9})\left(\frac{n\gamma}{2\pi}\right)^3$.  The number of
consecutive spacings is about $\frac{n|I_n|}{2\pi}$, so the expected number
of consecutive spacings less than $\gamma$ is about
\begin{equation*}
  \mu = \frac{n|I_n|}{2\pi}
      \left(\frac{\pi^2}{9}\right)\left(\frac{n\gamma}{2\pi}\right)^3
    = \frac{|I_n| n^4 \gamma_n^3}{144 \pi^2}
\end{equation*}
Since $\G$ is a sum of many
independent unlikely events, it is approximately Poisson with mean
$\mu$.

%***************************************************************
\section{The First Moment of $\G$}
%***************************************************************

Recall that $\G$ is the number of GUE-n eigenvalue pairs whose
average is in $I_n$ and whose difference is less than $\gamma$.
Thus $\G$ is the symmetrization to $n$ variables of a function of
$2$ variables:
\begin{eqnarray*}
  g(u,t) = \left\{
    \begin{array}{cl}
     \frac12 & \mbox{if} \quad |u-t| < \gamma \\
       & \quad \mbox{and} \quad \frac{u+t}{2} \in I \\
     0 & \mbox{otherwise}
  \end{array}
  \right\}\\
  \G(t_1,t_2,\dots,t_n) = \sum_{i\neq j}g(t_i,t_j).
\end{eqnarray*}

Using Theorem~\ref{intoutCue}, the expected value of $\G$ is
\begin{eqnarray*}
  E(\G) & = & \int\int_{[0,2\pi)^2} g(u,t)
    \left[
      \begin{array}{c c}
        K_n(u,u) & K_n(u,t) \\
        K_n(t,u) & K_n(t,t)
      \end{array}
    \right]
    du dt \\
  & = & \frac12 \int\int_{\Omega}
    \left[
      \begin{array}{c c}
        K_n(u,u) & K_n(u,t) \\
        K_n(t,u) & K_n(t,t)
      \end{array}
    \right]
    du dt.
\end{eqnarray*}
In the above formula, the region $\Omega$ is the set where
$g_2(t_1,t_2) = \frac{1}{2}$:
\begin{equation*}
  \Omega = \left\{ (u,t) \quad \mbox{ s.t. } \quad
    \begin{array}{c}
      \frac{u+t}{2} \in I \\
      |u-t| < \gamma
    \end{array} \right\}.
\end{equation*}

The projection kernel $K_n(u,t)$ for $U_n$ is
\begin{eqnarray*}
  K_n(u,t) & = & \left(\frac{1}{2\pi} \right) \frac{e^{in(u-t)}-1}{e^{i(u-t)}-1} \\
  & = & K_n(u-t)
\end{eqnarray*}

Since $(t-u)$ is small in the region $\Omega$, we change variables and
expand the determinant as a power series.  We let $x=\frac{u+t}{2}$,
$y=\frac{u-t}{2}$, and $dxdy = \frac12 dudt$.
\begin{eqnarray*}
  E(\G) & = & \int_{I_n} \int_{-\frac12\gamma}^{\frac12\gamma}
    \left[
      \begin{array}{c c}
        K_n(0) & K_n(2y) \\
        K_n(-2y) & K_n(0)
      \end{array}
    \right]
    du dt \\
  & = & \int_{I_n} \int_{-\frac12\gamma}^{\frac12\gamma}
    \begin{array}{c}
        4y^2\left((K'(0))^2 - K(0) K''(0)\right) \\
        + \bigO(y^4c_0c_4 + y^5c_1c_4 + \dots + y^8c_4c_4)
    \end{array} dy dx,
\end{eqnarray*}
where $c_j$ is an upper bound for the $(j)$th derivative of $K_n$.
We apply the following Lemma:

\begin{lem} \label{derivativeBounds}
  Consider the kernel $\frac{1}{2\pi} \frac{e^{inx}-1}{e^{ix}-1}$.
  For $x > x_0 \geq {\frac1n} (\mbox{mod}\ 2\pi)$, the $(k)$th derivative of this
  kernel is
  \begin{equation*}
    \bigO\left(\frac{1}{x_0} n^k \right)
  \end{equation*}
  Without the restriction on $x$, the derivative is $\bigO(n^{k+1}$).
\end{lem}
\begin{proof}
  Consider taking $k$ derivatives symbolically.  After taking $0\leq j\leq k$ of these
  derivatives, the result is a sum of terms of the form:
  \begin{equation*}
    \frac{\mbox{monomial}(n,e^{inx},e^{ix})}{(e^{ix}-1)^l}.
  \end{equation*}
  At each step, the worst that can happen is that either the exponent in the
  denominator increases by $1$, or the power of $n$ in the numerator increases
  by $1$.  Since we assumed that $x_0 \geq \frac1n$, the worse of these two outcomes
  is multiplying by $n$.  The result follows by induction, and observing that the
  number of terms is finite, depending only on $k$.

  For the bound without the restriction on $t$, express the kernel as a sum of complex
  exponentials and differentiate termwise.
\end{proof}

Returning to our expression for $E(\Fs)$,
\begin{eqnarray*}
  E(\Fs) & = & \int_{I_n} \int_{-\frac12\gamma}^{\frac12\gamma}
      4y^2\left((K'(0))^2 - K(0) K''(0)\right)
    + \bigO( y^4 n^6 + \dots + y^8 n^{10}  ) dydx \\
  & = & \frac{1}{4\pi^2} \int_{I_n} \int_{-\frac12\gamma}^{\frac12\gamma}
    4y^2 \left( -\frac{n^4}{4} + \frac{n^4}{3} +\bigO(n^3)  \right) dydx
    + \bigO(|I|\gamma^5 n^6 + \gamma^9n^{10}) \\
  & = & \frac{|I_n| \gamma^3 n^4}{144 \pi^2}
    \left( 1 + \bigO\left( \frac{1}{n} + \gamma^2n^2 + (\gamma^2n^2)^3 \right) \right).
\end{eqnarray*}
This agrees with the heuristic predictions in Subsection~\ref{heuristicCUE}.

%***************************************************************
\section{The Higher Moments of $\G$}
%***************************************************************

The $(k)$th power of $\G(t_1,t_2,\dots,t_n)$ may be written
\begin{eqnarray*}
  \G^k & = & \left( \sum_{i\neq j} g(t_i,t_j) \right)^k \\
       & = & \sum_{i_1\neq j_1;i_2\neq j_2;\dots;i_k\neq j_k} g(t_{i_1},t_{j_1})g(t_{i_2},t_{j_2})\dots g(t_{i_k},t_{j_k}),
\end{eqnarray*}
where
\begin{equation*}
  g(u,t) = \left\{
    \begin{array}{cl}
     \frac12 & \mbox{if} \quad |u-t| < \gamma \\
       & \quad \mbox{and} \quad \frac{u+t}{2} \in I \\
     0 & \mbox{otherwise}
  \end{array}
  \right\}
\end{equation*}
The indexing set for the sum on the right is not of the form required for Theorem~\ref{intoutCue}.
However it may be written as a disjoint union of indexing sets of the proper form.

%***************************************************************
\subsection{The Combinatorics of Collapses}

Let us, for example, partition the indexing set
$i_1\neq j_1;i_2\neq j_2;i_3\neq j_3$
into smaller indexing sets of the proper form.
We write this indexing set in shorthand as
\begin{equation*}
  (i_1|j_1)\wedge(i_2|j_2)\wedge(i_3|j_3),
\end{equation*}
and expand as follows.
Except for the subtle distinction between ``$|$'' and ``$\natural$'',
each step below should be clear.
\begin{eqnarray*}
  & & (i_1|j_1)\wedge(i_2|j_2)\wedge(i_3|j_3) \\
  & & \quad = (i_1|j_1)\left( (i_2|j_2\natural i_3|j_3) \vee (i_2,i_3|j_2|j_3) \vee (i_2,i_3|j_2,j_3)
    \vee (i_2,j_3|j_2|i_3) \vee (i_2,j_3|j_2,i_3) \right) \\
  & & \quad = \left( \left( (i_1|j_1\natural i_2|j_2 \natural i_3|j_3)
        \vee (i_1|j_1,i_2|j_2\natural i_3|j_3)
        \vee (i_1|j_1,j_2|i_2\natural i_3|j_3) \right. \right. \\
  & &   \qquad \qquad \vee \left. (i_2|j_1,i_3|j_3\natural i_2|j_2)
        \vee (i_1|j_1,j_3|i_3\natural i_2|j_2)
      \right) \\
  & & \qquad \vee \left( (i_1,i_2|j_1|j_2\natural i_3|j_3)
        \vee (i_1,i_2|j_1,j_2\natural i_3|j_3)
        \vee (i_1,i_2|j_1,i_3|j_2|j_3)
        \vee (i_1,i_2|j_1,j_3|j_2|i_3)
      \right) \\
  & & \qquad \vee \left( (i_1,j_2|j_1|i_2\natural i_3|j_3)
        \vee (i_1,j_2|j_1,i_2\natural i_3|j_3)
        \vee (i_1,j_2|j_1,i_3|i_2|j_3)
        \vee (i_1,j_2|j_1,j_3|i_2|i_3)
      \right) \\
  & & \qquad \vee \left( (i_1,i_3|j_1|j_3\natural i_2|j_2)
        \vee (i_1,i_3|j_1,i_2|j_2|j_3)
        \vee (i_1,i_3|j_1,j_2|i_2|j_3)
        \vee (i_1,i_3|j_1,j_3\natural i_2|j_2)
      \right) \\
  & & \qquad \vee \left. \left( (i_1,j_3|j_1|i_3\natural i_2|j_2)
        \vee (i_1,j_3|j_1,i_2|j_2|i_3)
        \vee (i_1,j_3|j_1,j_2|i_2|i_3)
        \vee (i_1,j_3|j_1,j_3\natural i_2|j_2)
      \right)
    \right) \\
  & & \qquad \vee \left(13 \mbox{ terms}\right)
        \quad \vee \quad \left(7 \mbox{ terms}\right)
        \quad \vee \quad \left(13 \mbox{ terms}\right) \\
  & & \qquad \vee \left( (i_1|j_1\natural i_2,j_3|j_2,i_3)
      \vee (i_1|j_1,i_2,j_3|j_2,i_3)
      \vee (i_1|j_1,j_2,i_3|i_2,j_3)
      \vee (i_1,i_3,j_3|j_1|j_2,i_3) \right. \\
  & & \qquad \qquad  \left. \vee (i_1,i_2,j_3|j_1,j_2,i_3)
      \vee (i_1,j_2,i_3|j_1|i_2,j_3)
      \vee (i_1,j_2,i_3|j_1,i_2,j_3)
    \right)
\end{eqnarray*}
Each term above is an indexing set of the form required for
Theorem~\ref{intoutCue}. We call such indexing sets ``collapses.''
Let us take one of the more complicated collapses above and
explain it's meaning:
\begin{equation*}
  (i_1|j_1,i_3|j_3 \natural i_2|j_2)
\end{equation*}
This means that $j_1=i_3$, but otherwise the indices
$i_1,j_1,i_2,j_2,i_3,j_3$ are distinct.  There are $53$ indexing sets
in the full expansion, but we only list $28$ of them above.
The symbols ``$|$'' and ``$\natural$'' both separate
the $i$'s and $j$'s into equivalence classes.  The symbol $\natural$ plays an additional
role which will not be clear until later in this section.
If $i_l$ and $j_l$ are in different equivalence classes, then
we say that these two equivalence classes fall into the same \emph{cluster}.
The clusters are separated by $\natural$.

Every cluster contains at least two equivalence classes. If any
cluster in a collapse contains more than two equivalence classes,
we call the collapse a {\it mixed collapse.}  If every cluster
contains exactly two equivalence classes we call it a {\it clean
collapse.}  We will show that the mixed collapses make a
negligible contribution to $E(\G^k)$.

There are $42$ mixed collapses for $k=3$, which we do not list explicitly.

The $11$ clean collapses for $k=3$ are:
\begin{eqnarray*}
  \left[ (i_1|j_1\natural i_2|j_2 \natural i_3|j_3) \right] & \iff & (1|2|3) \\
  \left[ (i_1,i_2|j_1,j_2\natural i_3|j_3)
    \vee (i_1,j_2|j_1,i_2\natural i_3|j_3) \right] & \iff & (12|3) \\
  \left[ (i_1,i_3|j_1,j_3\natural i_2|j_2)
    \vee (i_1,j_3|j_1,j_3\natural i_2|j_2) \right] & \iff & (13|2) \\
  \left[ (i_1|j_1\natural i_2,j_3|j_2,i_3)
    \vee (i_1|j_1\natural i_2,i_3|j_2,j_3) \right] & \iff & (1|23) \\
  \left[ \begin{array}{cccc}  & (i_1,i_2,j_3|j_1,j_2,i_3) & \vee & (i_1,j_2,i_3|j_1,i_2,j_3) \\
    \vee & (i_1,i_2,i_3|j_1,j_2,j_3) & \vee & (i_1,j_2,j_3|j_1,i_2,i_3)
  \end{array} \right] & \iff & (123)
\end{eqnarray*}
In the list above, we have collected those indexing sets together so that each
collection corresponds to one way of partitioning $\{1,2,3\}$.

The following Lemma is evident, based on consideration of the expansion
of $(i_1|j_1)\wedge(i_2|j_2)\wedge (i_3|j_3)$ into collapses which
we considered above.

\begin{lem} \label{countCollapses}
  When the indexing set $(i_1|j_1)\wedge(1_2|j_2)\wedge \dots \wedge (i_k|j_k)$
  is expanded into collapses, every cluster contains
  at least two equivalence classes.
  Clean collapses with $l$ clusters arise from a partition of the
  integers $\{1,2,\dots,k\}$ into $l$ nonempty subsets.
  Each of $l$ subsets of $\{1,2\dots,k\}$ leads to a cluster of
  equivalence classes of $\{i_1,j_1,i_2,j_2,\dots,i_k,j_k\}$.  Each
  partition of $\{1,2,\dots,k\}$ into $l$ equivalence classes
  corresponds to $2^{k-1}$ different clean collapses with $l$ clusters.
\end{lem}

We will see that clean collapses make the dominant contribution to
$E(\G^k)$.  For each clean collapse, the ``block diagonal'' id
the dominant contribution.  Using a combinatorial analysis, we'll
see that the sum over all clean collapse block diagonal terms is
approximately equal to $E(G_{\mu}^k)$, in agreement with Theorem~\ref{METCUE}.

%***************************************************************
\subsection{Estimating the clean collapses}

As an example, we select one of the clean collapses from $\G^3$,
use Theorem~\ref{intoutCue} to express its contribution to $E(\G^3)$ as
an integral, and estimate that integral.  We select:
\begin{equation*}
  (i_1,j_2|j_1,i_2 \natural i_3|j_3).
\end{equation*}
Making the substitutions $j_2=i_1$ and $i_2=j_1$, this indexing set makes the following
contribution to $E(\G^3)$:
\begin{equation*}
  \sum_{\begin{array}{c} i_1,j_1,i_3,j_3 \\ \mbox{distinct}\end{array}}
    g(t_{i_1},t_{j_1}) g(t_{j_1},t_{i_1}) g(t_{i_3},t_{j_3})
\end{equation*}
Since $g$ takes only the values $0$ or $\frac{1}{2}$ and is symmetric,
$g(t_{i_1},t_{j_1}) g(t_{j_1},t_{i_1}) = \frac12 g(t_{i_1},t_{j_1})$.
The contribution to $E(\G^3)$ simplifies to:
\begin{equation*}
  \frac12 \sum_{\begin{array}{c} i_1,j_1,i_3,j_3 \\ \mbox{distinct}\end{array}}
    g(t_{i_1},t_{j_1}) g(t_{i_3},t_{j_3})
\end{equation*}
This is the symmetrization of a function of four variables, so Theorem~\ref{intoutCue}
expresses the expected value of this sum as an integral involving four variables:

\begin{eqnarray*}
  & & E\left( \frac12 \sum_{\begin{array}{c} i_1,j_1,i_3,j_3 \\ \mbox{distinct}\end{array}}
    g(t_{i_1},t_{j_1}) g(t_{i_3},t_{j_3}) \right) \\
  & & \quad = \frac12 \int_{I^4} g(u_1,t_1) g(u_3,t_3)
    \left[4\times 4\right] du_1 dt_1 du_3 dt_3
\end{eqnarray*}
Notice that $g(t_{i_1},t_{j_1}) g(t_{i_3},t_{j_3})$ takes only the values $0$ and $\frac14$.
Instead of including $g(t_{i_1},t_{j_1}) g(t_{i_3},t_{j_3})$
in the integrand, we incorporate the constant $\frac14$ and restrict the region of integration:
\begin{equation*}
  \frac18 \int_{\Omega}
    \left[ \begin{array}{cccc}
      K_n(u_1,u_1) & K_n(u_1,t_1) & K_n(u_1,u_3) & K_n(u_1,t_3) \\
      K_n(t_1,u_1) & K_n(t_1,t_1) & K_n(t_1,u_3) & K_n(t_1,t_3) \\
      K_n(u_3,u_1) & K_n(u_3,t_1) & K_n(u_3,u_3) & K_n(u_3,t_3) \\
      K_n(t_3,u_1) & K_n(t_3,t_1) & K_n(t_3,u_3) & K_n(t_3,t_3)
    \end{array} \right]
    du_1 dt_1 du_3 dt_3
\end{equation*}
\begin{equation*}
  \Omega = \Omega_{I,(i_1,j_2|j_1,i_2 \natural i_3|j_3)}
  = \left\{ (u_1,t_1,u_3,t_3) \left|
    \begin{array}{c}
      \frac{u_1+t_1}{2} \in I \\
      |u_1 - t_1| < \gamma \\
      \frac{u_3+t_3}{2} \in I \\
      |u_3 - t_3| < \gamma
    \end{array} \right. \right\}.
\end{equation*}
We see for the first time the meaning of clustering the equivalence classes.  The variables
$u_1,t_1,u_3,t_3$ each represent one of the equivalence classes in $(i_1,j_2|j_1,i_2 \natural i_3|j_3)$.
The equivalence classes corresponding to $u_1$ and $t_1$ are in the same cluster and so the region of
integration $\Omega$ carries the restriction that $|u_1-t_1|$ is small.  Similarly,
$|u_3-t_3|$ is small because the equivalence classes $i_3$ and $j_3$ are in the same cluster.

%***************************************************************
\subsection{The main contribution from the clean collapses}

Of the $24 = 4!$ terms in the determinant, the ones which make a significant contribution to the
integral are the four terms on the block diagonal.  Their contribution is:
\begin{eqnarray*}
  & & \frac{1}{8}\int_{\Omega}
    \left[ \begin{array}{cccc}
      K_n(u_1,u_1) & K_n(u_1,t_1) & 0 & 0 \\
      K_n(t_1,u_1) & K_n(t_1,t_1) & 0 & 0 \\
      0 & 0 & K_n(u_3,u_3) & K_n(u_3,t_3) \\
      0 & 0 & K_n(t_3,u_3) & K_n(t_3,t_3)
    \end{array} \right] du_1 dt_1 du_3 dt_3 \\
  & = & \frac12 \left( \frac12 \int_{\Omega_{I,(i_1|j_1)}}
    \left[ \begin{array}{cc} K_n(u,u) & K_n(u,t) \\
      K_n(t,u) & K_n(t,t) \end{array} \right]
    du dt \right)^2 \\
  & = & \frac12 E(\G)^2 \\
  & = & \frac{1}{2} \left( \frac{|I_n| n^4 \gamma^3}{144\pi^2} \right)^2
    \left(1 + \bigO\left(\frac1n + n^2\gamma^2 + (n^2\gamma^2)^3 \right) \right)
\end{eqnarray*}

For a moment, let us take for granted that the block-diagonal terms of the clean collapses
are the dominant contributions to $E(\G^k)$.  Using identical techniques to the ones
above, the block diagonal terms of a clean collapse of $(i_1|j_1)\wedge(1_2|j_2)\wedge \dots \wedge (i_k|j_k)$
into $l$ clusters contributes
\begin{equation*}
  \left(\frac{1}{2}\right)^{k-l} E(\G)^l.
\end{equation*}
According to Lemma~\ref{countCollapses}, for every partition
of $\{1,2,\dots,k\}$ into $l$ equivalence classes there correspond
$2^{k-l}$ clean collapses with $l$ clusters.
Thus each partition of $\{1,2,\dots,k\}$ into $l$ equivalence classes
contributes $E(\G)^l$ to the expected value $E(\G^k)$.

In the following Subsection, we show that each partition
of $\{1,2,\dots,k\}$ contributes $\mu^l$ to the moment $E(G_{\mu}^k)$.  Taking for
granted that the block diagonal terms of clean collapses are the main
contribution and expecting other error terms to sneak in,
we have recovered the conclusions of Theorem~\ref{METCUE}.

%***************************************************************
\subsection{Moments of the Poisson Distribution}

The moments of a Poisson distribution
are given by polynomials in the
first moment with nonnegative integer coefficients:
\begin{eqnarray*}
  E(\G)   & = & \mu \\
  E(\G^2) & = & \mu + \mu^2 \\
  E(\G^3) & = & \mu + 3\mu^2 + \mu^3 \\
  E(\G^4) & = & \mu + 7\mu^2 + 6\mu^3 + \mu^4 \\
  E(\G^i) & = & \sum_{j=1}^{i} a_{ij} \mu^j \\
\end{eqnarray*}

\begin{lem} \label{momentCoefficients}
  The $a_{i,k}$ satisfy the following recurrence relation and initial conditions:
  \begin{eqnarray*}
    a_{1,1} & = & 1 \\
    a_{1,k} & = & 0 \qquad \mbox{for}\ k\neq 1 \\
    a_{i,k} & = & k a_{i-1,k} + a_{i-1,k-1} \qquad \mbox{ for } i>1.
  \end{eqnarray*}
  The $a_{i,j}$ have a combinatorial description: $a_{i,k}$ is the number of
  ways of partitioning $\{ 1,2,3,\dots,i\}$ into $k$ nonempty subsets.
\end{lem}

\begin{proof}
  Define the $a_{i,k}$ by the stated recurrence relation and initial conditions.

  Suppose that $X$ is Poisson with mean $\mu$.  Then
  \begin{equation*}
    \Pr(X=k) = \frac{\mu^k}{k!}e^{-\mu}.
  \end{equation*}
  Using the above formula one sees that
  \begin{equation*}
    \mu^j = E\left(X(X-1)(X-1)\dots(X-j+1) \right)
  \end{equation*}
  Let the coefficients $b^{ji}$ be the result of expanding the above in terms of moments:
  \begin{equation*}
    \mu^j = \sum_i b^{ji} E(\G^i)
  \end{equation*}
  The $b^{ji}$ are specified by the following recursion and initial conditions:
  \begin{eqnarray*}
    b^{1,1} & = & 1 \\
    b^{1,i} & = & 0 \qquad \mbox{for}\ j\neq 1 \\
    b^{j,i} & = & b^{j-1,i-1} - (j-1) b^{j-1,i}
  \end{eqnarray*}

  The following is true for $k=1, j\geq 1$, and for $k\geq1,j=1$:
  \begin{equation*}
    \sum_i b^{k,i} a_{i,j} = \delta^k_j
  \end{equation*}
  Assume for the purposes of induction that the above is true for
  all smaller values of $j$ and $k$.  The following symbolic manipulations,
  which use the Einstein summation convention, comprise the inductive step.

  \begin{eqnarray*}
    b^{j,i}a_{i,k} & = & \left(b^{j-1,i-1} - (j-1)b^{j-1,i}\right) a_{i,k}
      \left(k a_{i-1,k} + a_{i-1,k-1} \right) \\
    & = & b^{j-1,i-1} \left(k a_{i-1,k} + a_{i-1,k-1}\right) - (j-1)b^{j-1,i}a_{i,k} \\
    & = & b^{j-1,i-1}a_{i-1,k-1} + k b^{j-1,i-1}a_{i-1,k} - (j-1)\delta^{j-1}_k \\
    & = & \delta^{j-1}_{k-1} + k\delta^{j-1}_k - (j-1)\delta^{j-1}_k \\
    & = & \delta^{j-1}_{k-1} \\
    & = & \delta^j_k
  \end{eqnarray*}

  We have now shown that
  \begin{equation*}
    \sum_{i=1}^{k} b^{k,i} a_{i,j} = \delta(j=k)
  \end{equation*}
  for all $k\geq 1$.
  This equation uniquely determines the $a_{i,j}$.  This equation is
  satisfied if and only if
  \begin{eqnarray*}
    \mu^k & = & \sum b^{k,i} m_i \\
          & = & E\left(X(X-1)(X-1)\dots(X-k+1) \right),
  \end{eqnarray*}
  where   $ m_i = \sum_{j=1}^{i} a_{ij} \mu^j$.
\end{proof}

%***************************************************************
\section{The Errors Effecting Higher Moments of $\G$}
%***************************************************************

There are two types of contributions to $E(\G^k)$ which we have not yet considered.
We have yet to consider the off block diagonal terms of the clean collapses, and
the mixed collapses.  Since the block diagonal terms of clean collapses
contribute approximately $E(G_{\mu}^k)$, which is what we want,
these remaining terms are ``errors''.  Our main tool in bounding these
errors will be Lemma~\ref{derivativeBounds}.

%***************************************************************
\subsection{Controlling the Off Block Diagonal Terms}

We now estimate the off block diagonal terms.  For this, we
perform row and column operations on the matrix which leave the
determinant unchanged.  Specifically, we subtract each odd column
from the following even column, and subtract each odd row from the
following even row. In addition to leaving the determinant
unchanged, this operation does not effect which terms in the
determinant are on the block diagonal and which are off block
diagonal.
\begin{eqnarray*}
  & & \frac18 \int_{\Omega}
    \left[ \begin{array}{cccc}
      \frac{n}{2\pi} & K_n(u_1,t_1) & K_n(u_1,u_3) & K_n(u_1,t_3) \\
      K_n(t_1,u_1) & \frac{n}{2\pi} & K_n(t_1,u_3) & K_n(t_1,t_3) \\
      K_n(u_3,u_1) & K_n(u_3,t_1) & \frac{n}{2\pi} & K_n(u_3,t_3) \\
      K_n(t_3,u_1) & K_n(t_3,t_1) & K_n(t_3,u_3) & \frac{n}{2\pi}
    \end{array} \right]
    du_1 dt_1 du_3 dt_3 \\
  & & = \frac18 \int_{\Omega}
    \left[ \begin{array}{cc}
      \frac{n}{2\pi} & K_n(u_1,t_1) - \frac{n}{2\pi}                                             \\
      K_n(t_1,u_1) - \frac{n}{2\pi} & \frac{2n}{2\pi} - K_n(u_1,t_1) - K_n(t_1,u_1)              \\
      K_n(u_3,u_1) & K_n(u_3,t_1) - K_n(u_3,u_1)                                                 \\
      K(t_3,u_1) - K(u_3,u_1) & K(t_3,t_1) + K(u_3,u_1) -  K(u_3,t_1) - K(t_3,u_1)
    \end{array} \right. \\
  & &
    \left. \begin{array}{cc}
      K_n(u_1,u_3) & K_n(u_1,t_3)- K_n(u_1,u_3)                                               \\
      K(t_1,u_3) - K(u_1,u_3) & K(t_1,t_3) + K(u_1,u_3) - K(u_1,t_3) - K(t_1,u_3) \\
      \frac{n}{2\pi} & K_n(u_3,t_3) - \frac{n}{2\pi}                                          \\
      K_n(t_3,u_3) - \frac{n}{2\pi} & \frac{2n}{2\pi} - K_n(u_3,t_3) - K_n(t_3,u_3)
    \end{array} \right]
    du_1 \dots dt_3
\end{eqnarray*}

To estimate the individual entries in the matrix in the integrand,
we  use Lemma~\ref{derivativeBounds} in conjunction with
a Taylor series expansion of $K_n$.

%***************************************************************
\subsection{Dividing $\Omega$ into Two Regions}

We first divide the region $\Omega$ into two regions.
In the first region $\Omega_{\alpha}$, at least one pair of the variables $u_1$ and $u_3$ are closer than $\alpha$, where
$\alpha >> (\frac{1}{n})$.  In the latter region $\Omega_{\bar{\alpha}}$,
no pair of variables among $u_1$ and $u_3$ are closer than $\alpha$:
\begin{eqnarray*}
  \Omega_{\alpha}
    & = & \left\{ (u_1,t_1,u_3,t_3) \left|
      \begin{array}{c}
        \frac{u_1+t_1}{2} \in I \\
        |u_1 - t_1| < \gamma \\
        \frac{u_3+t_3}{2} \in I \\
        |u_3 - t_3| < \gamma \\
        |u_1 - u_3| < \alpha
      \end{array} \right. \right\} \\
  \Omega_{\bar{\alpha}}
    & = & \left\{ (u_1,t_1,u_3,t_3) \left|
      \begin{array}{c}
        \frac{u_1+t_1}{2} \in I \\
        |u_1 - t_1| < \gamma \\
        \frac{u_3+t_3}{2} \in I \\
        |u_3 - t_3| < \gamma \\
        |u_1 - u_3| \geq \alpha
      \end{array} \right. \right\}.
\end{eqnarray*}
The size of the region $\Omega_{\alpha}$ is $\bigO(|I|\alpha
\gamma^2)$.  Using Lemma~\ref{derivativeBounds} in conjunction
with the power series expansions of $K_n$, the entries in the
matrix are:
\begin{equation*}
\frac18 \int_{\Omega_{\alpha}}
  \left[ \begin{array}{cccc}
      \bigO(n)         & \bigO(n^2\gamma)    & \bigO(n)         & \bigO(n^2\gamma)   \\
      \bigO(n^2\gamma) & \bigO(n^3\gamma^2)  & \bigO(n^2\gamma) & \bigO(n^3\gamma^2) \\
      \bigO(n)         & \bigO(n^2\gamma)    & \bigO(n)         & \bigO(n^2\gamma)   \\
      \bigO(n^2\gamma) & \bigO(n^3\gamma^2)  & \bigO(n^2\gamma) & \bigO(n^3\gamma^2)
    \end{array} \right]
    du_1 dt_1 du_3 dt_3
\end{equation*}
Each term in the determinant, and hence the determinant itself, is $\bigO(n^8\gamma^4)$.  Multiplying this
by the size of the region of integration yields the total contribution of $\bigO(n^8 |I| \alpha \gamma^6)$
of off block diagonal terms in the region $\Omega_{\alpha}$.

The size of the region $\Omega_{\bar{\alpha}}$ is
$\bigO(|I|^2\gamma^2)$.  Again using Lemma~\ref{derivativeBounds},
the terms in the matrix are:
\begin{equation*}
  \frac18 \int_{\Omega_{\bar{\alpha}}}
    \left[ \begin{array}{cccc}
      \bigO(n)                      & \bigO(n^2\gamma)                  & \bigO(\frac{1}{\alpha})       & \bigO(\frac{n\gamma}{\alpha})     \\
      \bigO(n^2\gamma)              & \bigO(n^3\gamma^2)                & \bigO(\frac{n\gamma}{\alpha}) & \bigO(\frac{n^2\gamma^2}{\alpha}) \\
      \bigO(\frac{1}{\alpha})       & \bigO(\frac{n\gamma}{\alpha})     & \bigO(n)         & \bigO(n^2\gamma)   \\
      \bigO(\frac{n\gamma}{\alpha}) & \bigO(\frac{n^2\gamma^2}{\alpha}) & \bigO(n^2\gamma) & \bigO(n^3\gamma^2)
    \end{array} \right]
    du_1 dt_1 du_3 dt_3
\end{equation*}
The off block terms in the determinant above are $\bigO(\frac{n^6\gamma^4}{\alpha^2})$.
Multiplying this by the size of the region of integration yields $\bigO(\frac{n^6\gamma^6|I|^2}{\alpha^2})$.

Now we wish to choose $\alpha$ to minimize the total error from the off block diagonal terms.
We use the method of dominant balance to choose $\alpha = \left(\frac{|I|}{n^2}\right)^{\frac13}$,
so that the two contributions
from $\Omega_{\alpha}$ and $\Omega_{\bar{\alpha}}$ are comparable.  The resulting contribution
from the off block diagonal terms is
\begin{equation*}
  \bigO\left(     \left( \frac{1}{|I| n} \right)^{\frac23}      \left( n^4 |I| \gamma^3  \right)^2    \right).
\end{equation*}

Let us generalize the above example to other clean collapses.
Consider a clean collapse of $(i_1|j_1)\wedge(1_2|j_2)\wedge \dots \wedge (i_k|j_k)$
into $l$ clusters.  Using the same techniques as above, the contribution from
$\Omega_{\alpha}$ is $\bigO(n^{4l} |I|^{l-1} \alpha \gamma^{3l} )$.  The
contribution from $\Omega_{\bar{\alpha}}$ is
$\bigO\left( \left(\frac{1}{\alpha n}\right)^2 (|I| n^4\gamma^3)^l \right)$.
Using the method of dominant balance, we again choose
$\alpha = \left( \frac{|I|}{n^2} \right)^{\frac13}$.
The total contribution of the off-block-diagonal terms to a clean collapse
with $l$ clusters is
\begin{equation*}
  \bigO\left(     \left( \frac{1}{|I| n} \right)^{\frac23}      \left( n^4 |I| \gamma^3  \right)^{l}    \right).
\end{equation*}

%***************************************************************
\subsection{Controlling the Mixed Collapses}

Let us consider one of the mixed collapses from $\G^3$:
\begin{equation*}
  (i_1,j_2,i_3|j_1|i_2,j_3).
\end{equation*}
The contribution to $E(\G^3)$ from this mixed collapse may be written as an integral:

\begin{eqnarray*}
  & & E\left( \sum_{\begin{array}{c} i_1,j_1,i_2 \\ \mbox{distinct}\end{array}}
    g(t_{i_1},t_{j_1}) g(t_{i_2},t_{i_1}) g(t_{i_1},t_{i_2}) \right) \\
  & & \quad = \int_{I^3} g(u_1,t_1) g(u_2,u_1) g(u_1,u_2) \\
  & & \qquad \times \left[ \begin{array}{ccc}
      K_n(u_1,u_1) & K_n(u_1,t_1) & K_n(u_1,u_2) \\
      K_n(t_1,u_1) & K_n(t_1,t_1) & K_n(t_1,u_2) \\
      K_n(u_2,u_1) & K_n(u_2,t_1) & K_n(u_2,u_2)
    \end{array} \right]
    du_1 dt_1 du_2 \\
  & & \quad = \frac18 \int_{\Omega}
    \left[ \begin{array}{ccc}
      K_n(u_1,u_1) & K_n(u_1,t_1) & K_n(u_1,u_2) \\
      K_n(t_1,u_1) & K_n(t_1,t_1) & K_n(t_1,u_2) \\
      K_n(u_2,u_1) & K_n(u_2,t_1) & K_n(u_2,u_2)
    \end{array} \right]
    du_1 dt_1 du_2,
\end{eqnarray*}
where $\Omega$ is the region:
\begin{equation*}
  \Omega = \Omega_{I,(i_1,j_2,i_3|j_1|i_2,j_3)}
  = \left\{ (u_1,t_1,u_2) \left|
    \begin{array}{c}
      \frac{u_1+t_1}{2} \in I \\
      |u_1 - t_1| < \gamma \\
      \frac{u_1+u_2}{2} \in I \\
      |u_1 - u_2| < \gamma
    \end{array} \right. \right\}.
\end{equation*}

As before, we perform row and column operations.  For each cluster, one of the
equivalence classes is selected; in this case the equivalence class represented by
the variable $u_1$.  The $u_1$ column is then subtracted from the columns
corresponding to the other equivalence classes in its cluster.
Similarly, the $u_1$ row is subtracted from the others equivalence classes in its cluster.
We then use Lemma~\ref{derivativeBounds} to estimate the terms in the matrix.
\begin{eqnarray*}
  & & \frac18 \int_{\Omega}
    \left[ \begin{array}{cc}
      K_n(u_1,u_1)                 & K_n(u_1,t_1) - K_n(u_1,u_1)                 \\
      K(t_1,u_1) - K(u_1,u_1)  & K(t_1,t_1) + K(u_1,u_1) - K(u_1,t_1) - K(t_1,u_1) \\
      K(u_2,u_1) - K(u_1,u_1)  & K(u_2,t_1) + K(u_1,u_1) - K(u_1,t_1) - K(u_2,u_1)
    \end{array} \right. \\
  & & \qquad \left. \begin{array}{c}
      K_n(u_1,u_2) - K_n(u_1,u_1) \\
      K_n(t_1,u_2) + K_n(u_1,u_1) - K_n(u_1,u_2) - K_n(t_1,u_1) \\
      K_n(u_2,u_2) + K_n(u_1,u_1) - K_n(u_1,u_2) - K_n(u_2,u_1)
    \end{array} \right]
    du_1 dt_1 du_2 \\
  & & \quad = \frac18 \int_{\Omega}
    \left[ \begin{array}{ccc}
      \bigO(n)          & \bigO(n^2\gamma)    & \bigO(n^2\gamma)   \\
      \bigO(n^2\gamma)  & \bigO(n^3\gamma^2)  & \bigO(n^3\gamma^2) \\
      \bigO(n^2\gamma)  & \bigO(n^3\gamma^2)  & \bigO(n^3\gamma^2)
    \end{array} \right]
    du_1 dt_1 du_2.
\end{eqnarray*}
The size of each term in the determinant is $\bigO(n^7\gamma^4)$.  The size of the
region of integration is $\bigO(|I|\gamma^2)$, so the total contribution from
this mixed collapse is $\bigO(|I|n^7\gamma^6)$.

To generalize, consider a mixed collapse of $(i_1|j_1)\wedge(1_2|j_2)\wedge \dots \wedge (i_k|j_k)$
with $l_1$ equivalence classes and $l_2$ clusters.
The size of the region of integration is $\bigO(|I|^{l_2} \gamma^{l_1-l_2})$.  The size
of each term in the determinant is $\bigO(n^{l_1} (n\gamma)^{2(l_1-l_2)})$.  The total
contribution is
\begin{eqnarray*}
  & & \bigO(|I|^{l_2} \gamma^{3l_1 - 3l_2} n^{3l_1 - 2l_2}) \\
  & = & \bigO\left( \frac{\mu^{l_1-l_2}}{(n|I|)^{l_1-2l_2}} \right),
\end{eqnarray*}
where
\begin{equation*}
  \mu = \frac{|I|}{2\pi} \frac{n^4 \gamma^3}{72\pi}.
\end{equation*}

The mixed collapses of $(i_1|j_1)\wedge(1_2|j_2)\wedge \dots \wedge (i_k|j_k)$ have
the constraints $2 l_2 < l_1 < 2k$.  Thus $l_1 - 2l_2 \geq 1$.  Since $l_1\geq 2$,
\begin{equation*}
  l_1 - l_2 = \frac12 l_1 + \frac12 (l_1-2l_2) \geq \frac32 \geq 1.
\end{equation*}
Considering the expansion of $(i_1|j_1)\wedge(1_2|j_2)\wedge \dots \wedge (i_k|j_k)$,
one sees that $l_1 - l_2 \leq k$.  Since
\begin{equation*}
  1 \leq l_1 -l_2 \leq k,
\end{equation*}
the main contribution to $E(\G)$ contains contributions comparable to $\mu^{l_1-l_2}$.
Relative to the main term, the contribution from this, or any, mixed collapse is
\begin{equation*}
  \bigO\left( \frac{1}{n|I|} \right).
\end{equation*}

%***************************************************************
\section{Review of Case 1: The Circular Unitary Ensemble}
%***************************************************************

There are several sources of error in the estimation of $E(\G^k)$.
We summarize these sources in a table.  In this table,
$\mu = \frac{|I_n| n^4 \gamma^3}{144\pi^2}$.

\begin{tabular}{|l|c|c|}
  \hline
  description & contribution & relative size \\
  \hline
  \begin{tabular}{l} main term: \\ block diagonal clean collapse \end{tabular}
      & $\sum_{j=1}^{k} a_{k,j} \mu^j$ & 1  \\
  \hline
  \begin{tabular}{l} approximation $E(\G) \approx \mu$ \\ used to estimate main term \end{tabular}
      & $\mu^l \bigO\left( n^{-1} + n^2\gamma^2 + (n^2\gamma^2)^3\right)$
      & $\bigO\left(n^{-1} + n^2\gamma^2 + (n^2\gamma^2)^3\right)$  \\
  \hline
  \begin{tabular}{l} clean collapse with $l$ clusters, \\ off-block diagonal terms $\Omega_{\alpha}$ \end{tabular}
      & $\bigO\left(n^{4l}I^{l-1}\alpha \gamma^{3l}\right)$
      & $\bigO\left( \frac{\alpha}{I} \right)$ \\
  \hline
  \begin{tabular}{l} clean collapse with $l$ clusters, \\ off-block diagonal terms $\Omega_{\bar{\alpha}}$ \end{tabular}
      & $\bigO\left( \left(\frac{1}{\alpha n}\right)^2 (I n^4 \gamma^3 )^l  \right)$
      & $\bigO\left( \frac{1}{\alpha^2 n^2} \right)$ \\
  \hline
  \begin{tabular}{l} clean collapse with $l$ clusters, \\ total of off-block diagonal terms, \\
          choosing $\alpha = \left( \frac{I}{n^2}  \right)^{\frac13}$\end{tabular}
      & $\bigO\left( \left(\frac{1}{In} \right)^{\frac23} (I n^4 \gamma^3 )^l  \right)$
      & $\bigO\left( \left(\frac{1}{In} \right)^{\frac23} \right)$ \\
  \hline
  \begin{tabular}{l} mixed collapse, $l_1$ equivalence \\ classes, and $l_2$ clusters  \end{tabular}
      & $\bigO\left( \frac{\mu^{l_1-l_2}}{n|I|}  \right)$    & $\bigO\left( \frac{1}{n |I|} \right)$ \\
  \hline
\end{tabular}

%***************************************************************
\section{Case 2: Universal Unitary Ensembles}
%***************************************************************

Let $V(x)$ be a potential which is real analytic on $\RR$ and has
sufficient growth at $\infty$:
\begin{equation*}
  \lim_{|x|\to\infty} \frac{V(x)}{\log(x^2+1)} = \infty.
\end{equation*}

The universal unitary ensemble $UUE_n$, with potential $V(x)$, is
the set of $n\times n$ Hermitian matrices $M = (m_{ij})$ with
joint probability density function
\begin{equation*}
  (\mbox{const})\ \  e^{-n \sum V(\lambda_j)} dM
\end{equation*}
The joint probability density function for the eigenvalues is
obtained from the joint probability density function for the
matrix entries using Weyl integration.  See Appendix~\ref{WeylAppendix}.
\begin{equation*}
  (\mbox{const})\ \
  \prod_{i<j} (\lambda_i - \lambda_j)^2
  e^{-n\sum V(\lambda_i)} d\Lambda
\end{equation*}

As discussed in Section~\ref{potentialAssumptions}, we assume that
the potential is {\it regular} and that the equilibrium measure
$\Psi(x)dx$ is supported on a single interval $[a,b]$.

%***************************************************************
\subsection{Heuristic Prediction for the Universal Ensemble}

As in the case of CUE, we again make the
simplifying but false assumption that the consecutive spacings are
independent random variables. Thus $\G$, as a sum of many unlikely
events, will be approximately Poisson. We now estimate its mean.

At a point $x\in I$, the local density of eigenvalues is given by
the diagonal of the projection kernel:
\begin{equation*}
  \mbox{density} = K_n(x,x).
\end{equation*}

At the point $x$, a separation of $\gamma$ is equal to $\gamma
K_n(x,x)$ times the local mean spacing.  Thus the probability of
any one of these spacings being less than $\gamma$ is
approximately
\begin{equation*}
  \left(\frac{\pi^2}{9}\right)   \left(\gamma K_n(x,x)\right)^3.
\end{equation*}

In the region $[x,x+dx]$, the number of consecutive spacings is
approximately $K_n(x,x) dx$.  Integrating over $x\in I$, the
expected value of $\G$ is approximately:
\begin{equation*}
  E(\G) \approx \frac{\pi^2 \gamma^3}{9} \int_{I} K_n(x,x)^4 dx.
\end{equation*}
For $x\in [a+\epsilon,b-\epsilon]$, $K_n(x,x) = n\Psi(x) (1+\bigO(n^{-1}))$.
See~\cite{deiftzhou4}.
Thus, the heuristic prediction agrees with the Moment Estimation Theorem.

%***************************************************************
\section{The First Moment in the Universal Case}
\label{firstUUE}
%***************************************************************

As in the case of CUE, $\G$ is the symmetrization of a function of
$2$ variables:
\begin{eqnarray*}
  g_2(x,y) = \left\{
    \begin{array}{cl}
     \frac12 & \mbox{if} \quad \begin{array}{c} |x-y| < \gamma \\ \left( \frac{x+y}{2} \right) \in I_n \end{array} \\
     0 & \mbox{otherwise}
  \end{array}
  \right\}\\
  \G(t_1,t_2,\dots,t_n) = \sum_{i\neq j}g_2(t_i,t_j).
\end{eqnarray*}

Using the method of Gaudin (see Appendix~\ref{intoutapp}),
the expected value of $\G$ is
\begin{eqnarray*}
  E(\G) & = & \int\int_{\RR^2} g_2(u,t)
    \left[
      \begin{array}{c c}
        K_n(u,u) & K_n(u,t) \\
        K_n(t,u) & K_n(t,t)
      \end{array}
    \right]
    du dt \\
  & = & \frac12 \int\int_{\Omega}
    \left[
      \begin{array}{c c}
        K_n(u,u) & K_n(u,t) \\
        K_n(t,u) & K_n(t,t)
      \end{array}
    \right]
    du dt.
\end{eqnarray*}
In the above formula, the region $\Omega$ is the set where
$g_2(t_1,t_2) = \frac12$:
\begin{equation*}
  \Omega = \left\{ (u,t) \quad \mbox{ s.t. } \quad
    \begin{array}{c}
      \frac{u+t}{2} \in I \\
      |t - u| < \gamma
    \end{array} \right\}.
\end{equation*}

The projection kernel $K_n(x,y)$ is defined in Section~\ref{intoutapp}.

Compared to the $CUE_n$ case, it is now slightly more difficult
to estimate $E(\G)$ because the kernel $K_n(x,y)$ no longer
depends solely on the difference $(y-x)$.

The region $\Omega$ is narrow in the $(t-u)$ direction. This
suggests expanding the integrand in a Taylor series. Let
$x=\frac{u+t}{2}$ and $y=\frac{u-t}{2}$.  Then $dx dy = \frac12 du
dt$. In terms of the new variables $x$ and $y$, the first moment
is:
\begin{eqnarray*}
  & & E(\G) \\
  & & = \int_{I} \int_{-\frac12\gamma}^{\frac12\gamma}
    \left[
      \begin{array}{c c}
        K_n(x+y,x+y) & K_n(x+y,x-y) \\
        K_n(x-y,x+y) & K_n(x-y,x-y)
      \end{array}
    \right]
    dx dy \\
  & & = \int_{I} \int_{-\frac12\gamma}^{\frac12\gamma}
    \begin{array}{c}
       K_n(x+y,x+y) K_n(x-y,x-y) \\
       -  K_n(x+y,x-y) K_n(x-y,x+y)
    \end{array} dx dy \\
  & & = \int_{I} \int_{-\frac12\gamma}^{\frac12\gamma}
    \left(  \begin{array}{c}
        y^2 \left(K \Kff - K \Kbb - {\Kf}^2 \right) \\
        + \bigO(y^4c_0c_4 + y^5c_1c_4 + \dots + y^8c_4c_4)
    \end{array} \right)
    dx dy, \label{firstmoment}
\end{eqnarray*}
where
\begin{eqnarray*}
  \Kf(x,y)  & = & \Df K(x,y)\\
  \Kff(x,y) & = & \Df \Df K(x,y) \\
  \Kbb(x,y) & = & \Db \Db K(x,y)
\end{eqnarray*}
and $c_j$ is the maximum over $\Omega$ of any $(j)$th partial
derivative of $K_n(x,y)$. Using the Lemma~\ref{trickyLemma}, the
integral of $\bigO(y^4 c_0 c_4)$ over $\Omega$ is $\bigO\left(
|I_n| \gamma_n^5 n^6 \right)$, the integral of $\bigO(y^5c_1c_4)$ is
$\bigO(|I_n|\gamma^6n^7)$, and so on, and the integral of $\bigO(y^8c_4c_4)$
is $\bigO(|I_n|\gamma^9n^{10})$.  Dropping the redundant error terms,
\begin{eqnarray*}
  E(\G) & = & \int_I \frac{\gamma^3}{12}
      \left(K \Kff - K \Kbb - (\Kf)^2 \right) dx \\
  & & \quad + \bigO(|I_n|\gamma^5 n^6 + |I_n|\gamma^9n^{10}).
\end{eqnarray*}

We first evaluate the derivatives $\Kf$, $\Kff$, and $\Kbb$ symbolically
as a function of $x$ and $y$.  Let
\begin{equation*}
  [i,j,k] = b_{n-1} \frac{\eta_{n}^{(i)}(x)\eta_{n-1}^{(j)}(y)
      - \eta_{n}^{(j)}(y)\eta_{n-1}^{(i)}(x)}{(x-y)^k},.
\end{equation*}
where $\eta_j(x) = \phi_j(x) e^{-\frac{n}{2}V(x)}$, and $\phi_j$ is the $(j)$th
normalized orthogonal polynomial with respect to $e^{-nV(x)}dx$.
In terms of this new notation, we have the following expressions for $K(x,y)$
and its derivatives:
\begin{eqnarray*}
  K & = & [0,0,1] \\
  \frac{\partial}{\partial x} K & = & [1,0,1] - [0,0,2] \\
  \frac{\partial}{\partial y} K & = & [0,1,1] + [0,0,2] \\
  \frac{\partial^2}{\partial x^2} K & = & [2,0,1] - 2[1,0,2] + 2[0,0,3] \\
  \frac{\partial^2}{\partial x \partial y} K & = & [1,1,1] - [0,1,2] + [1,0,2] -2 [0,0,3] \\
  \frac{\partial^2}{\partial y^2} K & = & [0,2,1] + 2[0,1,2] + 2[0,0,3] \\
  \Kf & = & \left(\frac{\partial}{\partial x} + \frac{\partial}{\partial y} \right) K \\
    & = & [1,0,1] + [0,1,1] \\
  \Kff & = & \left( \frac{\partial^2}{\partial x^2}  + 2\frac{\partial^2}{\partial x \partial y}
        + \frac{\partial^2}{\partial y^2} \right) K \\
    & = & [2,0,1] + 2[1,1,1] + [0,2,1] \\
  \Kbb & = & \left( \frac{\partial^2}{\partial x^2}  + 2\frac{\partial^2}{\partial x \partial y}
        + \frac{\partial^2}{\partial y^2} \right) K \\
    & = & ([2,0,1] - 2[1,1,1] + [0,2,1]) + (-4[1,0,2] + 4[0,1,2]) + 8[0,0,3].
\end{eqnarray*}
We wish to evaluate the above quantities $\Kf$, $\Kff$, and $\Kbb$ along
the diagonal $x=y$.  At first we are taken aback by the presence of $(x-y)$ in the
denominator, but then recall that $K_n(x,y)$ is smooth everywhere, including the diagonal.
We determine the limiting values along the diagonal by use of L'Hopital's rule,
which in the case of $\Kbb$ must be applied three times:
\begin{eqnarray*}
  \Kf(x,x) & = & [200] + [110] \\
  \Kff(x,x) & = & [300] + 2[210] + [120] \\
  \Kbb(x,x) & = & (x-y)^2([2,0,3] - 2[1,1,3] + [0,2,3]) \\
    & & + (x-y)(-4[1,0,3] + 4[0,1,3]) + 8[0,0,3] \\
  & = & \frac13\Big( (x-y)^2([3,0,2]-2[2,1,2] + [1,2,2]) \\
    & & \quad + 2(x-y) ([2,0,2]-2[1,1,2] + [0,2,2]) \\
    & & \quad + (x-y)(-4[2,0,2]+4[1,1,2]) \\
    & & \quad + (-4[1,0,2] + 4[0,1,2]) + 8[1,0,2] \Big) \\
  & = & \frac13\left( [3,0,0] - 2[2,1,0] + [1,2,0] \right) \\
    & & + \frac13 \Big( (x-y) (-2[2,0,2] + 2[0,2,2]) + (4[1,0,2]+4[0,1,2])  \Big) \\
  & = & \frac13\left( [3,0,0] - 2[2,1,0] + [1,2,0] \right) \\
    & & + \frac16 \Big( (x-y)(-2[3,0,1] + 2[1,2,1]) \\
    & & \quad + (-2[2,0,1]+2[0,2,1]) \\
    & & \quad + (4[2,0,1]+4[1,1,1]) \Big) \\
  & = & \left( \frac23[2,1,0] + \frac23[1,2,0] \right)
    + \left(\frac13[2,0,1] + \frac23[1,1,1] + \frac13[0,2,1]\right) \\
  & = & \left( \frac23[2,1,0] + \frac23[1,2,0] \right)
    + (\frac13[3,0,0] + \frac23[2,1,0] + \frac13[1,2,0]) \\
  & = & \frac13[3,0,0] + [1,2,0].
\end{eqnarray*}

and substitute the results into the integrand:
\begin{eqnarray*}
  & & \left( K \Kff - K \Kbb - (\Kf)^2 \right) \\
  & & = \left(
        [1,0,0] \left([3,0,0]+[2,1,0]\right)
        - [1,0,0] \left(\frac13[3,0,0]-[2,1,0]\right)
        - [2,0,0]^2
      \right) \\
  & & = \left( \frac23[1,0,0][3,0,0] + 2[1,0,0][2,1,0] - [2,0,0]^2 \right).
\end{eqnarray*}

We derive estimates for $[3,0,0]$ and then state without proof the
analogous estimates for $[1,0,0]$, $[2,1,0]$, and $[2,0,0]$.
\begin{eqnarray*}
  & & [3,0,0] = \left( \frac{(b-a)}{4} + \bigO\left( \frac1n \right)\right)
      \left( \eta_n^{(3)}(x)\eta_{n-1}(x)
      - \eta_n(x) \eta_{n-1}^{(3)}(x) \right) \\
  & & \quad = \left( \frac{(b-a)}{4} \right)
    \left(n\pi\Psi(x)\right)^3 \frac{2}{(b-a)\pi} \\
  & & \qquad \times \left[
        \Re(-ie^{i\theta(x)} u(x)) \Re(e^{i\theta(x)} v(x))
        - \Re(e^{i\theta(x)} u(x)) \Re(-ie^{i\theta(x)} v(x))
        + \bigO\left(\frac1n\right)
      \right] \\
  & & \quad = \frac{\left(n\pi\Psi(x)\right)^3}{2\pi}
    |u(x)| |v(x)| \left[
        \sin(\theta+\alpha) \cos(\theta+\beta)
        - \cos(\theta+\alpha) \sin(\theta+\beta)
        + \bigO\left(\frac1n\right)
      \right] \\
  & & \quad = \frac{\left(n\pi\Psi(x)\right)^3}{2\pi}
    \left( \sqrt{\frac{b-x}{x-a}} + \sqrt{\frac{x-a}{b-x}} \right)
    \left( \sin(\alpha-\beta) +\bigO(n^{-1}) \right) \\
  & & \quad = \frac{\left(n\pi\Psi(x)\right)^3}{2\pi}
    \left( \sqrt{\frac{b-x}{x-a}} + \sqrt{\frac{x-a}{b-x}} \right)
    \left( \Im\left(\frac{u(x)}{v(x)}\right) +\bigO(n^{-1}) \right) \\
  & & \quad = \frac{\left(n\pi\Psi(x)\right)^3}{\pi}
    \left( 1 +\bigO(n^{-1}) \right).
\end{eqnarray*}
where
\begin{eqnarray*}
  \theta(x) & = & n\pi \int_x^b \Psi(s) ds \\
  u(x) & = & e^{i\frac{\pi}{4}} \left( \frac{b-x}{x-a} \right)^{\frac14}
    + e^{-i\frac{\pi}{4}} \left( \frac{x-a}{b-x} \right)^{\frac14} \\
  v(x) & = & e^{-i\frac{3\pi}{4}} \left( \frac{b-x}{x-a} \right)^{\frac14}
    + e^{-i\frac{\pi}{4}} \left( \frac{x-a}{b-x} \right)^{\frac14} \\
  \alpha(x) & = & \arg(u(x)) \\
  \beta(x) & = & \arg(v(x)).
\end{eqnarray*}

Using the same techniques,
\begin{eqnarray*}
  \ [2,1,0] & = & \frac{\left(n\pi\Psi(x)\right)^3}{\pi}
    \left( -1 +\bigO(n^{-1}) \right) \\
  \ [2,0,0] & = & \frac{\left(n\pi\Psi(x)\right)^2}{\pi}
    \left( 0 +\bigO(n^{-1}) \right) \\
  \ [1,0,0] & = & \frac{\left(n\pi\Psi(x)\right)}{\pi}
    \left( -1 +\bigO(n^{-1}) \right).
\end{eqnarray*}

Returning to our expression for $E(\G)$,
\begin{eqnarray*}
  & & E(\G) = \bigO(|I|\gamma^5 n^6 + |I|\gamma^9n^{10})
    + \int_I \frac{\gamma^3}{12}
      \left(K \Kff - K \Kbb - (\Kf)^2 \right) dx \\
  & & \quad = \bigO(|I|\gamma^5 n^6,|I|\gamma^9n^{10}) \\
  & & \qquad \quad + \int_I \frac{\gamma^3}{12}
      \frac{(n\pi \Psi(x))^4}{\pi^2}
      \left( \frac23(-1)(1) + 2(-1)(-1) - 0^2 + \bigO(n^{-1}) \right)
    dx \\
  & & \quad = \bigO(|I|\gamma^5 n^6 + |I|\gamma^9n^{10}) + \int_I
      \frac{\pi^2 \gamma^3 n^4 \Psi(x)^4}{9}
      \left( 1+\bigO(n^{-1}) \right) dx \\
  & & \quad = \frac{\pi^2\gamma^3 n^4}{9} \int_I \Psi(x)^4 dx
      \left( 1 + \bigO\left(n^{-1} + \gamma^2n^2 + (\gamma^2n^2)^3 \right)  \right).
\end{eqnarray*}

These calculations confirm that, for the first moment $E(\G)$, the
conclusions of Theorem~\ref{METUUE} are correct.  We now estimate the
higher moments.

%***************************************************************
\section{Higher Moments in the Universal Case}
%***************************************************************

As in the case of a random unitary matrix, we express $E(\G)$ as a
sum of several terms, with each term corresponding to a collapse
of
\begin{equation*}
  (i_1|j_1)\wedge(i_2|j_2)\wedge(i_3|j_3).
\end{equation*}
Using the method of Gaudin, each term is expressed as an integral
involving only a few variables.  As before, the main contribution
will come from the block-diagonal terms of the clean collapses:
\begin{eqnarray*}
  \mbox{(clean collapse, block diagonal)} & = &
      \sum_{j=1}^{k} a_{k,j} E(\G)^j \\
  & = & \sum_{j=1}^{k} a_{k,j} \mu^j \left( 1 +\bigO(n^{-1}\gamma^2 n) \right),
\end{eqnarray*}
where
\begin{equation*}
  \mu = \frac{\pi^2 \gamma^3 n^4}{9} \int_{I} \Psi(t)^4 dt.
\end{equation*}

The errors in this approximation have the same sources as in the
CUE case:
\begin{itemize}
  \item the approximation $E(\G) = \mu$ used in estimating the
    main term.
  \item clean collapse, off-block diagonal terms.
  \item mixed collapses.
\end{itemize}
The bounds for these errors are obtained using exactly the same reasoning
as in the CUE case.  The table from the CUE case applies here with only minor changes.
The only differences are that
\begin{itemize}
  \item The value of $\mu$ has changed.
  \item We now use Lemma~\ref{derivativeBoundsUniversal} instead of Lemma~\ref{derivativeBounds} to
    bound $K_n(x,y)$ and its derivatives.
\end{itemize}

%***************************************************************
\section{Controlling the Derivatives of $K_n(x,y)$ in the Universal Case}
%***************************************************************

For bounding the error terms for higher moments in the UUE case,
we require estimates for $K_n(X,y)$ and its derivatives.  In order
to apply the Christoffel-Darboux formula from
Lemma~\ref{chrisdar}, we approximate the coefficient $b_{n-1}$.
We refer the
reader to~\cite{deiftzhou4} and state the result and error term without proof.
This approximation is valid when the potential $V(x)$ is regular and the
equilibrium measure is supported on a single interval $[a,b]$.
\begin{eqnarray*}
  b_{n-1} & = & \intline x \phi_n(x) \phi_{n-1}(x) e^{-nV(x)} dx \\
  & = & \intline x \eta_n(x) \eta_{n-1}(x) dx \\
  & = & \int_a^b x \eta_n(x) \eta_{n-1}(x) dx + \bigO\left(\frac1n\right)\\
  & = & \int_a^b x \frac{2}{(b-a)\pi} \left(
        - \frac12 \left(\frac{b-x}{x-a}\right)^{\frac12}
        + \frac12 \left(\frac{x-a}{b-x}\right)^{\frac12}
      \right)dx + \bigO\left( \frac1n \right)\\
  & = & \frac{(b-a)}{4} + \bigO\left( \frac1n \right).
\end{eqnarray*}

The projection kernel is
\begin{equation*}
  K_n(x,y) = b_{n-1} \frac{\eta_n(x)\eta_{n-1}(y) - \eta_n(y)
  \eta_{n-1}(x)}{x-y}.
\end{equation*}

\begin{lem}\label{derivativeBoundsUniversal}
  Suppose that the potential $V(x)$ is regular and that it's equilibrium
  measure $\Psi(x)dx$ is supported on the interval $[a,b]$.
  Fix $\epsilon>0$ and let $R$ be the region
  \begin{equation*}
    [a+\epsilon,b-\epsilon]\times [a+\epsilon,b-\epsilon].
  \end{equation*}
  For $(x,y)\in R$, a mixed partial derivative of $K_n(x,y)$ of total degree
  $k\geq 0$ is $\bigO\left(n^{k+1}\right)$.  When $|x-y|\geq \frac1n$, we have
  the stronger bound $\bigO\left( \frac{1}{|x-y|} n^k \right)$.
\end{lem}

\begin{proof}
  Take $k$ partial derivatives of $K_n(x,y)$ symbolically.  The
  result is a finite number of terms of the form
  \begin{equation*}
    b_{n-1} c_{k_1,k_2,k_3} \frac{
        \eta_{n}^{(k_1)}(x)\eta_{n-1}^{(k_2)}(y)
        - \eta_{n}^{(k_2)}(y) \eta_{n-1}^{(k_1)}(x)
      }{(x-y)^{1+k_3}},
  \end{equation*}
  where $k_1+k_2+k_3=k$, the coefficients $c_{k_1,k_2,k_3}$ are
  integers which do not depend on $n$, and a superscript in
  parenthesis indicates multiple differentiation.

  We divide the region $R$ into two regions:
  \begin{eqnarray*}
    R_{diag} & = & \left\{ (x,y) \in R \left| |x-y| < \frac{1}{n}  \right. \right\} \\
    R_{bulk} & = & \left\{ (x,y) \in R \left| |x-y| \geq \frac{1}{n}  \right. \right\}
  \end{eqnarray*}

  In the region $R_{bulk}$, $\eta_n^{(k_1)}(x)$ and
  $\eta_{n-1}^{(k_1)}(x)$ are of size $\bigO(n^{k_1})$;
  $\eta_n^{(k_2)}(y)$ and $\eta_{n-1}^{(k_2)}(y)$ are of size
  $\bigO(n^{k_2})$; and $\frac{1}{(x-y)^{1+k_3}}$ is of size
  $\bigO\left(\frac{1}{|x-y|}n^{k_3} \right)$. The total
  contribution in the bulk region is then
  \begin{equation*}
    \bigO\left( \frac{1}{|x-y|} n^k \right).
  \end{equation*}

  The region $R_{diag}$ is more subtle.  Consider the term
  \begin{equation*}
    \frac{
        \eta_{n}^{(k_1)}(x)\eta_{n-1}^{(k_2)}(y)
        - \eta_{n}^{(k_2)}(y) \eta_{n-1}^{(k_1)}(x)
      }{(x-y)^{1+k_3}},
  \end{equation*}
  We treat both $\eta_{n}^{(k_1)}(x)\eta_{n-1}^{(k_2)}(y)$ and
  $\eta_{n}^{(k_2)}(y)\eta_{n-1}^{(k_1)}(x)$ the same way, so let us discuss
  $\eta_{n}^{(k_1)}(x)\eta_{n-1}^{(k_2)}(y)$.
  We expand $\eta_{n-1}^{(k_2)}(y)$ in a Taylor series centered at
  $x$.
  \begin{equation*}
    \eta_{n-1}^{(k_2)}(y) = \eta_{n-1}^{(k_2)}(x) + \dots
      +\frac{(y-x)^{k_3}}{k_3!}\eta_{n-1}^{(k_2+k_3)}(x)
      +\frac{(y-x)^{k_3+1}}{(k_3+1)!}\eta_{n-1}^{(k_2+k_3+1)}(\tilde{x}),
  \end{equation*}
  for some $\tilde{x}$ between $x$ and $y$.  Selecting the
  remainder term yields
  \begin{eqnarray*}
    & & \frac{ \eta_n^{(k_1)}(x) \frac{(y-x)^{k_3+1}}{(k_3+1)!}
        \eta_{n-1}^{(k_2+k_3+1)}(\tilde{x})  }{(y-x)^{1+k_3}} \\
    & & = \bigO(n^{k_1+k_2+k_3+1}) = \bigO(n^{1+k}),
  \end{eqnarray*}
  which is within the desired bound.
  Selecting one of the other terms in the expansion of
  $\eta_{n-1}^{(k_2)}(y)$ yields
  \begin{eqnarray*}
    & & \frac{ \eta_n^{(k_1)}(x) \frac{(y-x)^{j}}{j!}
        \eta_{n-1}^{(k_2+j)}(x)  }{(y-x)^{1+k_3}} \\
    & & = \frac{ \eta_n^{(k_1)}(x) \frac{1}{j!}
        \eta_{n-1}^{(k_2+j)}(x)  }{(y-x)^{1+k_3-j}},
  \end{eqnarray*}
  for some $0\leq j \leq k_3$.  Observe that this term is of the
  form
  \begin{equation*}
    e^{-nV(x)} \frac{p(x)}{(y-x)^{1+k_3-j}},
  \end{equation*}
  where $p(x)$ is a polynomial in $x$, and the exponent $1+k_3-j$
  in the denominator is positive.  Upon adding all such terms in
  the Taylor expansion, for every term in the symbolic
  differentiation of $K_n(x,y)$, the result is
  \begin{equation*}
    \sum_{j=0}^{k_3} e^{-nV(x)}
    \frac{p_j(x)}{(y-x)^{1+k_3-j}}.
  \end{equation*}
  Since $K_n(x,y) = \sum_{i=0}^{n-1} \eta_i(x) \eta_i(y)$, it is
  smooth along the diagonal $x=y$ and $p_j(x)$ is identically
  zero for $0\leq j \leq k_3$.
\end{proof}

%***************************************************************
\section{Asymptotics for Orthogonal Polynomials with General Weights}
%***************************************************************

We outline the derivation of the leading order asymptotics for
the $\eta_n(x)$ and $\eta_{n-1}(x)$, where
\begin{equation*}
  \eta_j(x) = \frac{\phi_j(x)}{e^{\frac{n}{2}V(x)}}
\end{equation*}
and  $\phi_j(x)$ is the $(j)$th
normalized orthogonal polynomial with
respect to the measure $e^{-nV(x)}$.

Our derivation follows the exposition of~\cite{deiftzhou4} and~\cite{deift_book}
very closely, and omits the proofs of several facts which are proven in
these articles.  The strongest results in this direction are proven
in~\cite{deiftzhou1} and~\cite{deiftzhou2}.
The only departure from their presentation comes when we obtain bounds for
the derivatives of $L(z)$, and then use these bounds to derive the
leading order asymptotics for the derivatives of
$\eta_n(x)$ and $\eta_{n-1}(x)$.

The leading order asymptotics of the derivatives of $\eta_n$
and $\eta_{n-1}$ turn out to be the derivatives of the leading
order asymptotics.  Of course one expects this, but it is not
something which can be taken for granted.

%***************************************************************
\subsection{Facts About the Equilibrium Measure of a Potential}
\label{potentialAssumptions}

Let $V(x)$ be a potential which is real analytic on $\RR$ and has
sufficient growth at $\infty$:
\begin{equation*}
  \lim_{|x|\to\infty} \frac{V(x)}{\log(x^2+1)} = \infty
\end{equation*}
This growth condition guarantees, in particular, that all the
moments of $e^{-nV(x)}$ are finite.

Given a probability measure $\mu$ on $\RR$, let $I^{V}$ be the
energy functional:
\begin{equation*}
  I^{V}(\mu) = \int_{\RR}\int_{\RR}
      \log|t-s|^{-1} d\mu(t) d\mu(s)
    + \int_{\RR} V(t) d\mu(t)
\end{equation*}

\begin{thm}
  Let $V$ and $I^{V}$ be as above.  Then there exists a unique
  probability measure $\mu = \mu^{V}$ such that
  \begin{equation*}
    E^{V} = \inf_{\mu\in M^1} I^{V}(\mu) = I^{V}(\mu^{V}),
  \end{equation*}
  where the infimum is over all probability measures on $\RR$.
  The equilibrium measure $\mu^{V}$ is compactly supported.
\end{thm}

Let the support of $\mu^{V}$ be the following disjoint union of
intervals:
\begin{equation*}
  \mbox{Supp}(\mu) = \bigcup_{j=1}^{l} [a_j,b_j].
\end{equation*}

Now if
\begin{itemize}
  \item $V(x)$ is real analytic
  \item $V(x)$ grows fast enough at infinity: $\lim_{|x|\to\infty} \frac{V(x)}{\log(x^2+1)} = \infty$,
\end{itemize}
then the equilibrium measure is $\Psi(x) dx$, where
\begin{eqnarray*}
  \Psi(x) & = & R_+(x)^{\frac12} h(x) \\
  R(x) & = & - \prod_{j=1}^{l} (z-a_j)(z-b_j),
\end{eqnarray*}
and $h(x)$ is real analytic on $\RR$.  See~\cite{deiftzhou10}.

Since $\Psi$ is continuous one can use the calculus of variations
to derive the Euler-Lagrange equations for the equilibrium
measure:

\begin{thm}
  There is a constant $l\in\RR$ such that the equilibrium measure
  $\mu^{V}$ satisfies the following conditions.
  \begin{itemize}
    \item $2\int \log |x-y|^{-1} d\mu^{V}(y) + V(x) \geq l$ for all $x\in\RR$.
    \item $2\int \log |x-y|^{-1} d\mu^{V}(y) + V(x) = l$
      for $x$ in the support of $\Psi$.
  \end{itemize}
  Conversely, if a compactly supported measure $\mu$ satisfies
  the above conditions for some $l$, then it is the equilibrium
  measure $\mu^{V}$.
\end{thm}

If $\psi(x)>0$ except at the endpoints of $J$, and we have strict
inequality in the above theorem for $x\not\in J$, then
we say that the potential $V(x)$ is {\it regular}.
Otherwise we call the potential {\it singular}.

For simplicity, we consider only the case when $V(x)$ is regular
and the support of $\psi$ is a single interval $J = [a,b]$.  This
is always the case when $V(x)$ is convex.

%***************************************************************
\subsection{Uniqueness for 2x2 Riemann Hilbert Problem}
% From p.193-194.

Let $\Sigma$ a contour in $\CC$, and $\Sigma_0$ be the same
contour excluding the points of intersection. Suppose that
$\nu(z): \Sigma_0 \to GL(2,\CC)$, such that $\nu$ is smooth,
bounded, and approaches $Id$ rapidly on the unbounded components
of $\Sigma_0$.

\begin{lem}
  Suppose that $m: \CC\setminus\Sigma \to GL(2,\CC)$ satisfies the
  Riemann-Hilbert problem $(\Sigma,\nu)$ if
  \begin{enumerate}
    \item $m$ is analytic in $\CC \setminus \Sigma$
    \item $m_+(z) = m_-(z) \nu(z)$ for $z\in \Sigma_0$.
    \item $m(z) \to Id$ as $z\to \infty$.
  \end{enumerate}
  Assume further that $\det(\nu) = 1$.
  Then the solution of the Riemann-Hilbert problem is unique, if it exists.
\end{lem}

\begin{proof}
  We sketch the proof from p. 194-198 of~\cite{deift_book}.
  First, because $\det(\nu)=1$, $\det(m)$ is
  analytic across the contour.  So $\det(m)$ is an analytic function
  and, because of the behavior of $m$ at infinity, $\det(m)$ must be
  the constant $1$.  In particular $m$ is always invertible.

  Now suppose that $m$ and $\tilde{m}$ are two solutions.  Then,
  using an algebraic trick specific to $2\times 2$ matrices with
  determinant $1$, it turns out that $H = \tilde{m} m^{-1}$
  satisfies a Riemann-Hilbert problem with the same contour but
  $\nu = Id$.  In other words $H$ is analytic, approaches $Id$ at
  infinity, hence is identically equal to $Id$, establishing
  uniqueness.
\end{proof}

%***************************************************************
\subsection{Expressing Orthogonal Polynomials as the Solution of Riemann Hilbert Problems}
% p. 198

Let $\pi_j(z)$ be the monic orthogonal polynomials with respect to
the weight $e^{-nV(z)}$.

\begin{thm}[Fokas, Its, Kitaev] \label{FIK}
  Let $Y(z)$ be the $2 \times 2$ matrix-valued function satisfying the following RHP:
  \begin{enumerate}
    \item $Y(z)$ is analytic in $\CC \setminus \RR$.
    \item $Y(z) = (1+\bigO(z^{-1})) \left(\begin{array} {c c} z^q & 0 \\ 0 & z^{-q} \end{array} \right)$  as  $z \to \infty$, $z \in \CC \setminus \RR$.
    \item $Y_{+}(z) = Y_{-}(z) \left( \begin{array} {c c} 1 & e^{-nV(z)} \\ 0 & 1 \end{array} \right) $  for  $z \in \RR$.
  \end{enumerate}
  Then $Y$ encodes information about $\pi_q(z)$, the $(q)$th monic
  orthogonal polynomial with respect to the measure $e^{-nV(x)}dx$:
  \begin{equation*}
    Y(z) = \left(\begin{array}{c c}
        \pi_q(z)
          & \int_{\RR} \frac{\pi_q(s) e^{-nV(s)}}{s-z} \frac{ds}{2\pi i} \\
        \gamma_{q-1} \pi_{q-1}(z)
          & \gamma_{q-1} \int_{\RR}
            \frac{\pi_q{q-1}(s) e^{-nV(s)}}{s-z} \frac{ds}{2\pi i}
      \end{array}\right),
  \end{equation*}
  where the constant $\gamma_{q-1}$ is
  \begin{equation*}
    \gamma_{q-1} = \frac{-2\pi i}
        {\int \pi_{q-1}^2(s) e^{-nV(s)}ds}.
  \end{equation*}
\end{thm}

The $2\times 2$ jump matrix in this Theorem has determinant $1$,
so the solution of the Riemann-Hilbert problem is unique if it
exists.  To prove this Theorem, one verifies that the
given solution satisfies the Riemann-Hilbert problem.

%***************************************************************
\subsection{Removing the Bulk of the Oscillatory Behavior}
%p. 202.

For $\Im(z) \neq 0$, let
\begin{equation*}
  g(z) = \int \log(z-s) \Psi(s) ds.
\end{equation*}

We choose $\log(z)$ to have a branch cut along the negative real
axis, and to be real on the positive real axis.

Observe that $g(z)$ is analytic in the region $\CC \setminus
(-\infty,b_l]$.  For real $z$, let $g_+(z)$ be the limiting value
from above the branch cut, and $g_-(z)$ the limiting value from
below. Then
\begin{eqnarray*}
  g_{\pm}(z) & = & \int_a^b \log|z-s|\Psi(s)ds
    \pm \pi i \int_z^b \Psi(s) ds \\
  g_+(z) + g_-(z) & = & 2\int \log |z-s| \Psi(s) ds \\
  g_+(z) - g_-(z) & = & 2\pi i \int_z^b \Psi(s) ds
\end{eqnarray*}

Observe that because $\Psi(s) ds$ is a probability measure,
$g_+(z) - g_-(z) = 2\pi i$ for $z<a$.  Thus $e^{ng(z)}$ is analytic in
$\CC \setminus [a,b]$.

For $z\in\CC\setminus\RR$, let
\begin{equation*}
  \mo(z) =
    \left(\begin{array}{c c} e^{\frac{nl}{2}} & 1 \\ 0 & e^{-\frac{nl}{2}} \end{array} \right)
    Y(z) e^{-ng(z)}
    \left(\begin{array}{c c} e^{-\frac{nl}{2}} & 1 \\ 0 & e^{\frac{nl}{2}} \end{array} \right)
\end{equation*}

Then $\mo$ satisfies a simpler Riemann-Hilbert problem:
\begin{enumerate}
  \item $\mo(z)$ is analytic in $\CC\setminus\RR$
  \item $\mo_+(z) = \mo_-(z) \vo(z)$ for $z\in\RR$
  \item $\mo(z) = I + \bigO(z^{-1})$ as $z\to \infty$,
\end{enumerate}
where
\begin{eqnarray*}
  \vo(z) & = &
      \left(\begin{array}{c c}
        e^{\frac{nl}{2}+ng_-(z)} & 1 \\
        0 & e^{-\frac{nl}{2}-ng_-(z)}
      \end{array} \right)
      \left(\begin{array}{c c} 1 & e^{-nV(z)} \\ 0 & 1 \end{array} \right)
      \left(\begin{array}{c c}
        e^{-ng_+(z) -\frac{nl}{2}} & 1 \\
        0 & e^{ng_+(z)+\frac{nl}{2}}
      \end{array} \right)\\
  & = & \left( \begin{array}{c c}
      e^{n(g_-(z) - g_+(z))} & e^{n(g_-(z) + g_+(z) -V(z) + l)} \\
      0 & e^{n(g_+(z) - g_-(z))}
    \end{array} \right).
\end{eqnarray*}

The upper right entry in the jump matrix $\nu$ is
$e^{n(g_+ + g_- -V + l)}$.
We analyze the exponent in detail.
\begin{eqnarray*}
  & & g_+(x) + g_-(x) - V(x) + l \\
  & & \quad = 2 \intline \log(|x-s|) \Psi(s) ds - V(x) + l
\end{eqnarray*}
Recall that the equilibrium measure $\Psi$ satisfies the
Euler-Lagrange equations.  Thus
\begin{equation*}
  g_+(x) + g_-(x) - V(x) + l \quad
    \left\{ \begin{array}{c @{\quad \mbox{for} \quad}l}
      < 0 & x<a \\
      = 0 & a\leq x\leq b \\
      < 0 & x>b
    \end{array} \right.
\end{equation*}
Because of our assumption that the potential $V$ is regular, the
inequalities above are sharp.

A further simplification to the jump matrix is this.  For $x<a$
and $x>b$,
\begin{equation*}
  e^{n(g_+(x)-g_-(z))}=e^{n(g_+(x)-g_-(z))}=1.
\end{equation*}

To summarize, the jump matrix defined on $\RR$ has the following
properties in the three regions:
\begin{equation*}
  \nu^{(1)}(x) = \left\{ \begin{array}{c @{\quad \mbox{for} \quad} l}
      \left( \begin{array}{c c}
          1 & e^{n(g_-(x) + g_+(x) -V(x) + l)} \\
          0 & 1
        \end{array} \right)
        & x < a \\
      \left( \begin{array}{c c}
          e^{-n(g_+(x)-g_-(x))} & 1 \\
          0 & e^{n(g_+(x)-g_-(x))}
        \end{array} \right)
        & a \leq x \leq b \\
      \left( \begin{array}{c c}
          1 & e^{n(g_-(x) + g_+(x) -V(x) + l)} \\
          0 & 1
        \end{array} \right)
        & x > b
    \end{array} \right.
\end{equation*}

%***************************************************************
\subsection{Analytic Continuation}

The difference $g_+(x)-g_-(x)$ has been defined on the real axis
as the difference between $g(z)$ on opposite sides of a branch
cut. Recall from the previous Subsection that $g_+(x)-g_-(x)$ has an
integral representation for $x\in\RR$.
Call this difference $G(x)$.
\begin{equation*}
  G(x) = 2\pi i \int_x^b \Psi(s) ds.
\end{equation*}
Recall that because $V(x)$ is real analytic, $\Psi(z)$ is an
analytic function times $\sqrt{(z-a)(b-z)}$:
\begin{equation*}
  \Psi(z) = h(z) \sqrt{(z-a)(b-z)}.
\end{equation*}
The function $h(z)$ is analytic in a neighborhood of $\RR$, and
real on the real axis. We choose $\Psi(z)$ to be positive on the
interval $[a,b]$ and have branch cuts along $(-\infty,a]$ and
$[b,\infty)$. Because of its integral representation, $G(z)$ is
also analytic in this region.

%***************************************************************
\subsection{Factoring the Jump Matrix}

Along $(a,b)$, $G'(x) = -2\pi i \Psi(x)$.  Thus there is a complex
neighborhood $U$ of $(a,b)$ such that $\Re(G)>0$ in $U\cap \CC^+$,
and $\Re(G)<0$ in $U\cap \CC^-$.  Because $G$ behaves as $z^{3/2}$
at the endpoints $a$ and $b$, $U$ cannot form an angle greater
than $\frac{\pi}{3}$, with respect to the real axis, at either
endpoint of $[a,b]$.  See Figure~\ref{createLenses}.

In the interval $[a,b]$, we factor the jump matrix as
\begin{eqnarray*}
  \nu & = &
    \left( \begin{array}{c c}
        1 & 0 \\
        e^{nG(z)} & 1
      \end{array} \right)
    \left( \begin{array}{c c}
        0 & 1 \\
        -1 & 0
      \end{array} \right)
    \left( \begin{array}{c c}
        1 & 0 \\
        e^{-nG(z)} & 1
      \end{array} \right) \\
  & = & \nu_-^{(1)} \nu_0^{(1)} \nu_+^{(1)}
\end{eqnarray*}

As illustrated in Figure~\ref{createLenses}, we deform the
original contour $\Sigma^{(1)}$ to obtain a Riemann Hilbert
problem on a new contour $\Sigma^{(2)}$.  It is possible
to translate between a solution of one Riemann-Hilbert
problem and a solution of the other.  Except for $z$ in the region
enclosed by the lenses, $m^{(2)}(z) = m^{(1)}(z)$.  Inside the
upper region,
\begin{equation*}
  m^{(2)}(z) = m^{(1)}(z) \left(\nu_+^{(1)}(z)\right)^{-1}.
\end{equation*}
For $z$ inside the lower region,
\begin{equation*}
  m^{(2)}(z) = m^{(1)}(z) \nu_-^{(1)}(z).
\end{equation*}

\begin{figure}
  \begin{center}
    \includegraphics[scale=0.6]{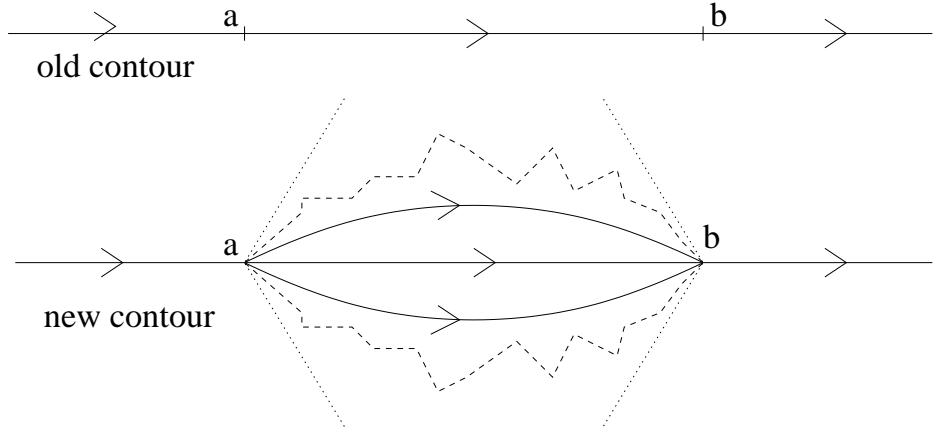}
    \caption{\label{createLenses}
        It is easy to translate between a solution of one
        Riemann-Hilbert problem and a solution of the other.}
  \end{center}
\end{figure}

The jump functions on the new contour are:

\begin{equation*}
  \nu^{(2)}(z) = \left\{ \begin{array}{c @{\quad \mbox{for} \quad} l}
      \left( \begin{array}{c c}
          1 & e^{n(g_-(z) + g_+(z) -V(z) + l)} \\
          0 & 1
        \end{array} \right)
        & z < a \\
      \left( \begin{array}{c c}
          1 & 0 \\
          e^{nG(z)} & 1
        \end{array} \right)
        & z \mbox{ on\ \ lower\ \ lens} \\
      \left( \begin{array}{c c}
          0 & 1 \\
          -1 & 0
        \end{array} \right)
        & a \leq z \leq b \\
      \left( \begin{array}{c c}
          1 & 0 \\
          e^{-nG(z)} & 1
        \end{array} \right)
        & z \mbox{ on\ \ upper\ \ lens} \\
      \left( \begin{array}{c c}
          1 & e^{n(g_-(z) + g_+(z) -V(z) + l)} \\
          0 & 1
        \end{array} \right)
        & z > b
    \end{array} \right.
\end{equation*}

All of the exponents appearing in these jump functions are negative.
As $n\to\infty$, these terms disappear.  The
limiting jump functions are:
\begin{equation*}
  \nu^{(\infty)}(z) = \left\{ \begin{array}{c @{\quad \mbox{for} \quad} l}
      \left( \begin{array}{c c}
          1 & 0 \\
          0 & 1
        \end{array} \right)
        & z < a \\
      \left( \begin{array}{c c}
          1 & 0 \\
          0 & 1
        \end{array} \right)
        & z \mbox{\ \ on\ \ the\ \ lower\ \  lip} \\
      \left( \begin{array}{c c}
          0 & 1 \\
          -1 & 0
        \end{array} \right)
        & a \leq z \leq b \\
      \left( \begin{array}{c c}
          1 & 0 \\
          0 & 1
        \end{array} \right)
        & z \mbox{\ \ on\ \ the\ \ upper\ \  lip} \\
      \left( \begin{array}{c c}
          1 & 0 \\
          0 & 1
        \end{array} \right)
        & z > b
    \end{array} \right.
\end{equation*}

Let $m^{(\infty)}$ be the solution of this limiting
Riemann-Hilbert problem.  The convergence $\nu^{(2)} \to
\nu^{(\infty)}$ as $n\to \infty$ is not uniform and occurs more
slowly in neighborhoods of $a$ and $b$.  Thus we cannot
automatically conclude that $m_2 \to m_{\infty}$ as $n\to
\infty$.

Because we assumed the potential $V(x)$ is regular, there are no
other areas of slow convergence.

%***************************************************************
\subsection{Solution of the Limiting Riemann Hilbert Problem}

To find the solution $m_{\infty}$, we perform a change of basis to
diagonalize the jump function.  This replaces the $2\times 2$
Riemann-Hilbert problem for $m_{\infty}$ with a pair of scalar
Riemann-Hilbert problems.

For $a < x < b$, the diagonalization of the jump function is:
\begin{equation*}
  \left( \begin{array}{c c} 0 & 1 \\ -1 & 0 \end{array} \right)
  = \left( \begin{array}{c c} 1 & 1 \\ i & -i \end{array} \right)
    \left( \begin{array}{c c} i & 0 \\ 0 & -i \end{array} \right)
    \left( \begin{array}{c c} 1 & 1 \\ i & -i \end{array} \right)^{-1}
\end{equation*}

The transformed Riemann-Hilbert problem is:
\begin{eqnarray*}
  \hat{m}(z)
    & = & \left( \begin{array}{c c} 1 & 1 \\ i & -i \end{array} \right)^{-1}
        m_{\infty}(z)
        \left( \begin{array}{c c} 1 & 1 \\ i & -i \end{array} \right)\\
  \hat{\nu}(z)
    & = & \left\{ \begin{array}{c @{\quad \mbox{for} \quad} l}
      \left( \begin{array}{c c}
          1 & 0 \\
          0 & 1
        \end{array} \right)
        & x < a \\
      \left( \begin{array}{c c}
          i & 0 \\
          0 & -i
        \end{array} \right)
        & a \leq x \leq b \\
      \left( \begin{array}{c c}
          1 & 0 \\
          0 & 1
        \end{array} \right)
        & x > b
    \end{array} \right.
\end{eqnarray*}

The solution of this pair of Riemann-Hilbert problems is
\begin{eqnarray*}
  \hat{m}(z) & = &
      \left( \begin{array}{c c}
          \beta & 0 \\
          0 & \beta^{-1}
        \end{array} \right) \\
  \beta & = & \left( \frac{z-b}{z-a} \right)^{\frac14},
\end{eqnarray*}
where $\beta$ is analytic on $\CC\setminus [a,b]$ and $\beta(z)
\to 1$ as $n \to \infty$.

Changing back to the original basis we recover the solution
$m_{\infty}(z)$:
\begin{eqnarray*}
  m_{\infty}(z) & = &
    \left( \begin{array}{c c} 1 & 1 \\ i & -i \end{array} \right)
    \hat{m}(z)
    \left( \begin{array}{c c} 1 & 1 \\ i & -i \end{array} \right)^{-1} \\
  & = & \left( \begin{array}{c c} 1 & 1 \\ i & -i \end{array} \right)
    \hat{m}(z)
    \left( \begin{array}{c c} \frac12 & \frac{-i}2 \\ \frac12 & \frac{i}2 \end{array} \right)\\
  & = & \left( \begin{array}{c c}
      \frac{\beta + \beta^{-1}}{2} & \frac{\beta - \beta^{-1}}{2i} \\
      -\frac{\beta - \beta^{-1}}{2i} & \frac{\beta +
      \beta^{-1}}{2}
    \end{array} \right).
\end{eqnarray*}

%***************************************************************
\subsection{Parametrices at the Endpoints}

The convergence $\nu^{(2)} \to \nu^{(\infty)}$ is not uniform at
the endpoints $a$ and $b$.  Thus the approximation $m^{(2)}(z)
\approx m^{(\infty)}(z)$ is inappropriate near $a$ and $b$.

Let $O_a$ and $O_b$ be neighborhoods of $a$ and $b$, which can be
chosen as small as desired.  Inside the neighborhood $O_a$, one
constructs a parametrix $m_a$ which satisfies the following
Riemann-Hilbert problem:
\begin{itemize}
  \item $m^{(a)}$ is analytic on $O_a \setminus \Sigma^{(2)}$
  \item $m^{(a)}_+(z) = m^{(a)}_+(z) \nu^{(2)}(z)$
    for $z \in \Sigma^{(2)} \cap O_a$
  \item $m^{(a)}(z) = m^{(\infty)}(z) \left(1 + \bigO\left(\frac1n \right)
    \right)$ for $z \in \partial O_a$.
\end{itemize}
In other words, $m_a$ satisfies the Riemann-Hilbert problem
$(\Sigma^{(2)},\nu^{(2)})$ exactly in a neighborhood of $a$, and
matches $m^{(\infty)}$ to within $\bigO(n^{-1})$.  Similarly, one
constructs a parametrix $m^{(b)}$ in $O_b$.  We will not discuss
the construction of these parametrices, which comprise a large
part of the work of~\cite{deiftzhou1} and~\cite{deiftzhou2}.

We now patch together an approximate solution of the
Riemann-Hilbert problem for $m^{(2)}$:
\begin{equation*}
  m^{(p)} = \left\{ \begin{array}{c @{\quad\mbox{ for } \quad} l}
      m^{(\infty)} & z\in \CC\setminus O_a \setminus O_b \\
      m^{(a)} & z\in O_a \\
      m^{(b)} & z\in O_b
    \end{array} \right.
\end{equation*}

The actual solution $m^{(2)}$ will be recovered as a perturbation
of the approximate solution.  Let $L(z) = m^{(2)}(z)(m^{(p)}(z))^{-1}$,
so that $m^{(2)}$ can be recovered from $m^{(p)}$ and $L$.  Then
$L$ approaches $Id$ at $\infty$ and has a jump function as
indicated in Figure~\ref{barbell}.

\begin{figure}
  \begin{center}
    \includegraphics[scale=0.6]{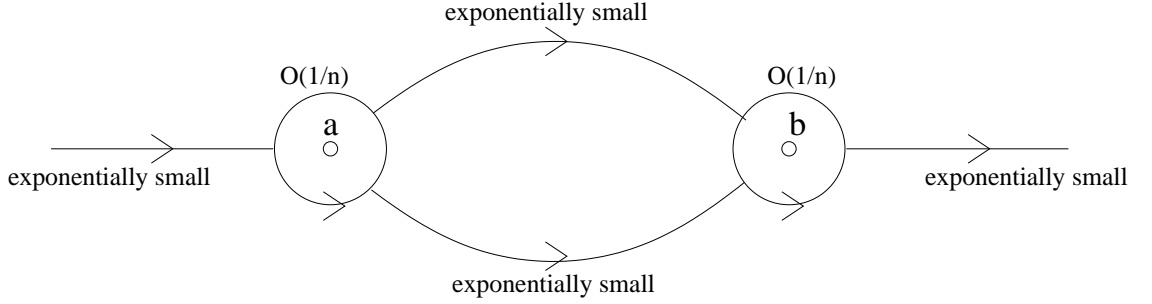}
    \caption{\label{barbell}
        The Riemann-Hilbert problem for $L$.}
  \end{center}
\end{figure}

Observe that the jump function for $L(z)$ in in both $L^2(\Sigma)$
and $L^{\infty}(\Sigma)$.

%***************************************************************
\subsection{The Solution $L(z)$ and its Properties}

Since $m^{(2)}(z) = L(z) m^{(p)}(z)$, the quantity $L(z) -Id$
plays the role of an error term in the approximation
$m^{(2)}(z) = m^{(p)}(z)$.  In order to bound these errors,
we make use of Theorem~\ref{abstractNonsense}, which appears
in~\cite{deift_book}, near p. 219.

Suppose that for a given Riemann-Hilbert problem $(\Sigma,\nu)$,
the jump function $\nu$ factors as $\nu = b_-^{-1}b_+$, where the
$b_{\pm}$ are bounded and invertible. Let
\begin{equation*}
  b_{\pm} = Id \pm \omega_{\pm} \qquad \mbox{and} \qquad \omega = \omega_+ + \omega_-.
\end{equation*}
For $f\in L^2(\Sigma)$, let $C_{\pm}$ be the limits of the Cauchy
operators:
\begin{eqnarray*}
  (Cf)(z) & = & \frac{1}{2\pi i} \int_{\Sigma} \frac{f(s)}{s-z} ds
    \qquad \mbox{for} \quad z\in CC\setminus\Sigma \\
  (C_{\pm}f)(z) & = & \lim_{z'\to z} Cf(z') \qquad \mbox{for} \quad
    z\in \Sigma_0
\end{eqnarray*}
Let $C_{\omega}f = C_+(f\omega_-) + C_-(f\omega_+)$. Assume that
$\omega_{\pm}$ are in $L^2(\Sigma)$, so that $C_{\omega}Id \in
L^2(\Sigma)$.

\begin{thm} \label{abstractNonsense}
  Suppose that $I-C_{\omega}$ is invertible on $L^2(\Sigma)$ and
  let $\mu \in I + L^2(\Sigma)$ be the unique solution of
  \begin{equation*}
    (I-C_{\omega})\mu = Id
  \end{equation*}
  or, more properly,
  \begin{equation*}
    (I-C_{\omega})(\mu-Id) = C_{\omega} Id \in L^2(\Sigma).
  \end{equation*}
  Then
  \begin{equation*}
    m(z) = I + (C(\mu\omega))(z) \qquad
      \mbox{for} \quad z\in \CC\setminus\Sigma
  \end{equation*}
  is the solution of the Riemann Hilbert problem.
\end{thm}

We now depart from~\cite{deift_book}  in order to bound
$L(z)$ and its derivatives.

In our case, we may choose $b_- = Id$, $b_+ = \nu$, and $\omega_+
= \nu-Id$.  Then $C_{\omega}f = C_-(f\omega_+)$.

For smooth contours, the $C_{\pm}$ are bounded on $L^2(\Sigma)$
with constants depending only on $\Sigma$.  Since $\omega_+\in
L^{\infty}(\Sigma)$,
\begin{eqnarray*}
  ||C_{\omega}f||^{L^2(\Sigma)}
    & \leq & c ||\omega_+ f||_{L^2(\Sigma)} \\
  & \leq & c ||\omega_+||_{L^{\infty}(\Sigma)}
      || f||_{L^2(\Sigma)}.
\end{eqnarray*}

Since $||\omega_+||_{L^{\infty}(\Sigma)} = \bigO(\frac1n)$, we
eventually have $||C_{\omega}||^{L^2\to L^2} < \frac12$.  Once
this occurs, $(I-C_{\omega})$ is invertible with $L^2\to L^2$ norm
at most $2$.  This allows us to bound $(\mu-Id)$ in $L^2(\Sigma)$:
\begin{eqnarray*}
  ||\mu - I||_{L^2(\Sigma)} & \leq &
    ||(I-C_{\omega})C_{\omega} I ||_{L^2(\Sigma)} \\
  & \leq & 2 ||C_{\omega} I ||_{L^2(\Sigma)} \\
  & \leq & 2 c ||\omega_+||_{L^2(\Sigma)} \\
  & = & \bigO\left(\frac1n\right).
\end{eqnarray*}

Finally we have an expression for $L(z)$:
\begin{equation*}
  L(z) = Id + \frac{1}{2\pi i} \int_{\Sigma} \frac{\mu
  \omega}{s-z} ds.
\end{equation*}
Since $\mu\in L^2(\Sigma)$ and $\omega \in L^2(\Sigma)$,
$(\mu\omega) \in L^1(\Sigma)$ and
\begin{eqnarray*}
  ||\mu\omega||_{L^1(\Sigma)}
    & \leq & ||\mu||_{L^2(\Sigma)}^{\frac12}
      ||\omega||_{L^2(\Sigma)}^{\frac12} \\
  & = & \bigO\left(\frac1n\right)^{\frac12}  \bigO\left(\frac1n\right)^{\frac12} \\
  & = & \bigO\left(\frac1n\right).
\end{eqnarray*}
For $z$ not within distance $d$ of the contour $\Sigma$,
\begin{eqnarray*}
  |L(z) - Id| & = &
    \left|\frac{1}{2\pi i} \int_{\Sigma} \frac{\mu\omega}{s-z} ds \right| \\
  & \leq & \frac{1}{2\pi d} ||\mu \omega||_{L^1(\Sigma)} \\
  & = & \bigO\left(\frac{1}{nd}\right).
\end{eqnarray*}
For $|z-\Sigma|\geq d$, we can also control the derivatives of
$L(z)$:
\begin{eqnarray*}
  \left|\frac{\partial^k}{\partial z^k}  L(z)\right| & = &
    \left|\frac{k!}{2\pi i} \int_{\Sigma} \frac{\mu\omega}{(s-z)^{k+1}} ds\right| \\
  & \leq & \frac{k!}{2\pi d^{k+1}} ||\mu \omega||_{L^1(\Sigma)} \\
  & = & \bigO\left(\frac{1}{nd^{k+1}}\right).
\end{eqnarray*}

%***************************************************************
\subsection{Retracing Our Steps}

We now retrace our steps in order to obtain asymptotics for
orthogonal polynomials, on the real axis for $a<z<b$.

First, we choose a proper subinterval of
$[a+\epsilon,b-\epsilon]\subset[a,b]$. Next, review the above
steps to make sure that the neighborhoods $O_a$ and $O_b$ have
radius strictly less than $\epsilon$.  Let $d$ be the distance
between the contour $\Sigma$ and the interval
$[a+\epsilon,b-\epsilon]$.

For $z\in [a+\epsilon,b-\epsilon]$, we obtain the asymptotics for
$Y(z)$ by approaching the real axis from above.  Recall that
$Y_{11}(z)=\pi_n(z)$ and $Y_{21}(z)=-i(\mbox{const})\pi_{n-1}(z)$
are the quantities that interest us, and are unaffected by the
jump function along the real axis.
\begin{eqnarray*}
  Y(z) & = & \left( \begin{array}{c c}
      e^{-\frac{nl}{2}} & 0 \\
      0 & e^{\frac{nl}{2}}
    \end{array} \right)
    L(z)
    m^{(\infty)}(z)
    \left( \begin{array}{c c}
      1 & 0 \\
      e^{-nG(z)} & 1
    \end{array} \right) \\
  & & \quad \times
    \left( \begin{array}{c c}
      e^{\frac{nl}{2}} & 0 \\
      0 & e^{-\frac{nl}{2}}
    \end{array} \right)
    e^{n g_+(z)} \\
  & = & \left( \begin{array}{c c}
      e^{-\frac{nl}{2}} L_{11}(z) & e^{-\frac{nl}{2}} L_{12}(z) \\
      e^{\frac{nl}{2}}  L_{21}(z) & e^{\frac{nl}{2}}  L_{22}(z)
    \end{array} \right)
    \left( \begin{array}{c c}
      \frac{\beta + \beta^{-1}}{2} & \frac{\beta - \beta^{-1}}{2i} \\
      -\frac{\beta - \beta^{-1}}{2i} & \frac{\beta +
      \beta^{-1}}{2}
    \end{array} \right) \\
  & & \quad \times
    \left( \begin{array}{c c}
      e^{\frac{nl}{2} + ng_+(z)} & 0 \\
      e^{\frac{nl}{2} + ng_+(z) - n G(z)} & e^{-\frac{nl}{2}+ng_+(z)}
    \end{array} \right)
\end{eqnarray*}

The entry $Y_{11}(z)$ in the above product is
\begin{eqnarray*}
  Y_{11}(z) & = & e^{ng_+(z)}\left(
        L_{11}(z)\frac{\beta+\beta^{-1}}{2}
        - L_{12}(z)\frac{\beta - \beta^{-1}}{2i}
      \right) \\
  & & \quad + e^{ng_-(z) - nG(z)} \left(
        L_{11}(z)\frac{\beta - \beta^{-1}}{2i}
        + L_{12}\frac{(z)\beta+\beta^{-1}}{2}
      \right) \\
  & = & \Re\left[e^{ng_+(z)}\left(
      e^{i\frac{\pi}{4}}(L_{11}(z)+iL_{12}(z))
        \left(\frac{b-z}{z-a}\right)^{\frac14} \right. \right.\\
  & & \qquad \left. \left. + e^{-i\frac{\pi}{4}}(L_{11}(z)-iL_{12}(z))
        \left(\frac{z-a}{b-z}\right)^{\frac14} \right) \right] \\
  & = & e^{n\int_a^b \log|z-s| \Psi(s) ds} \times \\
  & & \Re\left[ e^{n i\pi \int_z^b \Psi(s) ds}
      \left(e^{i\frac{\pi}{4}}(L_{11}(z)+iL_{12}(z))
        \left(\frac{b-z}{z-a}\right)^{\frac14} \right. \right.\\
  & & \qquad \left. \left. + e^{-i\frac{\pi}{4}}(L_{11}(z)-iL_{12}(z))
        \left(\frac{z-a}{b-z}\right)^{\frac14} \right) \right].
\end{eqnarray*}

The entry $Y_{21}(z)$ in the above product is:
\begin{eqnarray*}
  & = & e^{nl + n\int_a^b \log|z-s| \Psi(s) ds} \times \\
  & & i\times \Im\left[ e^{n i\pi \int_z^b \Psi(s) ds}
      \left(e^{i\frac{3\pi}{4}}(L_{22}(z)-iL_{21}(z))
        \left(\frac{b-z}{z-a}\right)^{\frac14} \right. \right.\\
  & & \qquad \left. \left. + e^{-i\frac{3\pi}{4}}(L_{22}(z)+iL_{21}(z))
        \left(\frac{z-a}{b-z}\right)^{\frac14} \right) \right].
\end{eqnarray*}

In the Christoffel-Darboux formula, it is easier to deal with
normalized orthogonal polynomials $\phi_j(z)$ than with monic
orthogonal polynomials $\pi_j(z)$:
\begin{equation*}
  \phi_j(x) = \frac{\pi_j(x)}
    {\left( \intline (\pi_j(x))^2 e^{-nV(x)} dx \right)^{\frac12}}
\end{equation*}
When $j=n$, we approximate heuristically the normalizing constant
in the denominator, refer the reader to to~\cite{deiftzhou4} for the precise
result, and state without proof the error term in our
approximation:
\begin{eqnarray*}
  & & \intline (\pi_n(x))^2 e^{-nV(x)} dx \\
  & & \quad = \left(1 + \bigO\left( \frac1n \right)\right)
    \int_a^b (\pi_n(x))^2 e^{-nV(x)} dx \\
  & & \quad = \left(1 + \bigO\left( \frac1n \right)\right)
    \times \int_a^b e^{2n\int_a^b log|x-s| \Psi(s) ds -nV(x)} \\
  & & \qquad \quad \times \Re \left[ \mbox{oscillatory} \cdot
      \sqrt{ \left(\frac{b-x}{x-a}\right)^{\frac12} +
        \left(\frac{x-a}{b-x}\right)^{\frac12} }
    \right]^2 dx \\
  & & \quad = \left(1 + \bigO\left( \frac1n \right)\right)
    e^{-nl} \int_a^b \frac12
    \left( \left(\frac{b-x}{x-a}\right)^{\frac12} +
      \left(\frac{x-a}{b-x}\right)^{\frac12}
    \right) dx \\
  & & \quad = e^{-nl} \frac{(b-a)\pi}{2}
    \left(1 + \bigO\left( \frac1n \right)\right).
\end{eqnarray*}

Even more convenient than the normalized orthogonal polynomials,
for later computations, are
\begin{equation*}
  \eta_j(x) = \frac{\phi_j(x)}{e^{\frac{n}{2}V(x)}}.
\end{equation*}
The advantage of using $\eta_j(x)$ is that they are orthonormal
with respect to Lebesgue measure on $\RR$.

The entry $Y_{21}(z)$ is a negative imaginary constant times
$\pi_{n-1}(z)$. Using techniques similar to those above, we
approximately normalize this polynomial, obtaining an
approximation for $\phi_{n-1}(x)$.

Using our knowledge of $Y_{11}(z)$, $Y_{21}(z)$, the normalizing
constants, and $L(z)$, we obtain the leading order asymptotics for
$\eta_n(x)$, $\eta_{n-1}(x)$, and their derivatives.

\begin{thm}
  Let $\epsilon>0$, and $k$ a nonnegative integer.
  For $z\in [a+\epsilon,b-\epsilon]$, the
  leading order asymptotics of
  \begin{equation*}
    \frac{\partial^k}{\partial z^k} \eta_n(z)
    \qquad \mbox{and} \qquad
    \frac{\partial^k}{\partial z^k} \eta_{n-1}(z)
  \end{equation*}
  as $n\to \infty$ are, respectively,
  \begin{eqnarray*}
    & & \left(n \pi \Psi(z)\right)^k \sqrt{\frac{2}{(b-a)\pi}} \\
    & & \qquad \Re\left[ i^k e^{n i\pi \int_z^b \Psi(s) ds}
        \left(e^{i\frac{\pi}{4}}
          \left(\frac{b-z}{z-a}\right)^{\frac14}
        + e^{-i\frac{\pi}{4}}
          \left(\frac{z-a}{b-z}\right)^{\frac14} \right)
        + \bigO\left( \frac1n \right)
      \right]
  \end{eqnarray*}
  and
  \begin{eqnarray*}
    & & \left(n\pi \Psi(z)\right)^k \sqrt{\frac{2}{(b-a)\pi}} \\
    & & \qquad \Re\left[ i^k e^{n i\pi \int_z^b \Psi(s) ds}
        \left(e^{-i\frac{3\pi}{4}}
          \left(\frac{b-z}{z-a}\right)^{\frac14}
        + e^{-i\frac{\pi}{4}}
          \left(\frac{z-a}{b-z}\right)^{\frac14} \right)
        + \bigO\left( \frac1n \right)
      \right].
  \end{eqnarray*}
  The constants implied by $\bigO$ depend only on $k$, the
  potential $V(x)$, and $\epsilon$.
\end{thm}

%***************************************************************
\section{Special Case: The Gaussian Unitary Ensemble}
%***************************************************************

The ensemble $GUE_n$ is the set of $n\times n$ Hermitian matrices
with joint probability density function (j.p.d.f.)
\begin{eqnarray*}
  & & \left( \pi^{-\frac{n^2}{2}} 2^{\frac{n(n+1)}{2}} \right)
      \exp\left(-\sum_{i} m_{i,i}^2 - 2\sum_{i<j} |m_{i,j}|^2 \right)dM \\
  & & = C_n e^{-Tr(M^2)} dM
\end{eqnarray*}
In other words, the entries on the diagonal are real, and are chosen
independently from the Gaussian distribution with variance $1$.
In the upper triangle, the real and imaginary parts of each entry are chosen
independently, all from the Gaussian distribution with variance $\frac12$.

The j.p.d.f. for the eigenvalues is obtained from
the j.p.d.f. for the matrix entries using Weyl integration:
\begin{equation*}
    \left( \prod_{j=0}^{n-1}\left(\frac{2^j}{j!} \right)
    \pi^{-n/2} \frac{1}{n!} \right)
 \prod_{i<j} (\lambda_i - \lambda_j)^2 e^{-\sum \lambda_i^2}
\end{equation*}

At first glance the Gaussian unitary ensemble seems not to be
an instance of the universal unitary ensemble, because the
j.p.d.f. contains the term $e^{-\sum \lambda_i^2}$ instead
of $e^{-n \sum \lambda_i^2}$.  Since the potential $V(x)=x^2$
is homogeneous, this difference is eliminated by rescaling.
Thus, the Moment Estimation Theorem we present for GUE is
really a special case of the universal version of the Theorem.

Instead of just using the universal Theorem, we prove
the Moment Estimation Theorem in the GUE case by using
Plancherel-Rotach asymptotics for Hermite polynomials
instead of the more general asymptotics by
Deift, Kriecherbauer, McLaughlin, Venakides, and Zhou.
The standard reference for Plancherel-Rotach asymptotics
is~\cite{szego}.

We present this alternative proof because
many readers will be familiar with Plancherel-Rotach asymptotics
but not the Deift-Zhou asymptotics.   For such readers,
the GUE version of the Moment Estimation Theorem is a good
warm-up for the universal case.

\begin{thm} [Moment Estimation Theorem, GUE version]
  Let $\epsilon>0$.  For each $n$, let $I_n \subset
  \left(-(\sqrt{2}-\epsilon)\sqrt{n},(\sqrt{2}-\epsilon)\sqrt{n}\right)$
  be an interval, such that $\sqrt{n}|I_n| \to \infty$ as $n\to
  \infty$.  Let $\G$ be the random variable which counts the number of $GUE_n$ eigenvalues
  whose average is in $I_n$ and whose difference is at most
  $\gamma_n$.  Let $G_{\mu}$ be the Poisson distribution with mean
  \begin{equation*}
    \mu_n = \frac{\pi^2 \gamma^3}{9} \int_{I} K_n(x,x)^4 dx.
  \end{equation*}
  Then for all $k\geq 1$,
  \begin{equation*}
    E(\G^k) = E(G_{\mu}^k)\left(1 +
        \bigO\left(
          n^{-1}, n\gamma^2,(n^{\frac12}|I|)^{-\frac23}
        \right)
    \right).
  \end{equation*}
  The constant implied by $\bigO$ depends only on $k$ and $\epsilon$.
\end{thm}

The novel aspects of the proof of the GUE Theorem are the estimation
of the first moment using Plancherel-Rotach asymptotics,
the different scaling, and a slightly different proof of the
bounds for $K_n(X,y)$ and its derivatives.
We only treat these aspects of the GUE proof.

%***************************************************************
\subsection{The First Moment in the GUE Case}

As in the CUE and UE cases, one uses the method of Gaudin to express
the expected value of $\G$ as a two-dimensional integral.
\begin{eqnarray*}
  E(\G) & = & \int\int_{\RR^2} g_2(u,t)
    \left[
      \begin{array}{c c}
        K_n(u,u) & K_n(u,t) \\
        K_n(t,u) & K_n(t,t)
      \end{array}
    \right]
    du dt \\
  & = & \frac12 \int\int_{\Omega}
    \left[
      \begin{array}{c c}
        K_n(u,u) & K_n(u,t) \\
        K_n(t,u) & K_n(t,t)
      \end{array}
    \right]
    du dt,
\end{eqnarray*}
where $\Omega$ is the region
\begin{equation*}
  \Omega = \left\{ (u,t) \quad \mbox{ s.t. } \quad
    \begin{array}{c}
      \frac{u+t}{2} \in I \\
      |t - u| < \gamma
    \end{array} \right\}.
\end{equation*}

We state without proof an expression for the projection kernel:
\begin{eqnarray*}
  K_n(x,y) & = & \frac{e^{-\frac{(x^2+y^2)}{2}}}{2^n(n-1)!\sqrt{\pi}}
      \left(\frac{H_n(x) H_{n-1}(y) - H_{n-1}(x) H_n(y)}{x-y}\right),
\end{eqnarray*}
where $H_n(x)$ is the $(n)$th Hermite polynomial.  This expression is
an instance of the Christoffel-Darboux formula, which we
encountered in the UUE case.  The only challenge in the above formula
is to derive the correct normalizing constant.

As compared to the $U_n$ case, it is now slightly more difficult to
estimate $E(\G)$ because the kernel $K_n(x,y)$ no longer depends
solely on the difference $(y-x)$.

As in the case of UUE (Subsection~\ref{firstUUE}), we change
variables and expand the integrand in a Taylor series:

The region $\Omega$ is narrow in the $(t-u)$ direction.
This suggests expanding the integrand in a Taylor series.
Let $x=\frac{u+t}{2}$ and $y=\frac{u-t}{2}$.  Then
$dx dy = \frac12 du dt$.
In terms of the new variables $x$ and $y$, the first moment is:
\begin{eqnarray*}
  E(\G) & = & \int_{I} \int_{-\frac12\gamma}^{\frac12\gamma}
    \left(  y^2 \left(K \Kff - K \Kbb - {\Kf}^2\right)
      + \bigO(y^4 c_0 c_4 + y^8c_4c_4)
    \right)
    dx dy,
\end{eqnarray*}
where
and $c_j$ is the maximum over $\Omega$ of any $(j)$th partial
derivative of $K_n(x,y)$. Using the Lemma~\ref{trickyLemma} instead
of Lemma~\ref{derivativeBoundsUniversal}, the
integral of $\bigO(y^4 c_0 c_4)$ and $\bigO(y^5c_4c_4)$ over $\Omega$ are
$\bigO\left( |I| \gamma^5 n^3 \right)$ and $\bigO\left( |I| \gamma^9 n^5 \right)$
respectively.

We explicitly expand the integrand in terms of Hermite polynomials.
After collecting similar terms,
\begin{eqnarray*}
  E(\G) & = & \bigO\left( |I| \gamma^5 n^3 + |I|\gamma^9 n^5 \right)
    + \int_{I} \frac{\gamma^3}{12}
    \left(\frac{e^{-2t^2}}{2^{2n}(n-1)!^2\pi} \right) \times \\
  & & \quad \left[ \begin{array}{c}
        2(H_{n-1}H_n' - H_nH_{n-1}') (H_{n-1}'H_n'' - H_n'H_{n-1}'') \\
        + \frac23 (H_{n-1}H_n''' - H_nH_{n-1}''')(H_{n-1}H_n' - H_nH_{n-1}') \\
        -(H_nH_{n-1}'' - H_{n-1}H_n'')^2
      \end{array} \right] dt.
\end{eqnarray*}
All Hermite polynomials or their derivatives are evaluated at
$t$; this is omitted to save space.
In the next section, we will use the Plancherel Rotach asymptotics to derive an
approximation to the integrand valid as $n \to \infty$.

%***************************************************************
\subsection{Application of Plancherel Rotach Asymptotics to the
First Moment}

Let's use Plancherel Rotach to estimate $E(\G)$.
Unlike other systems of orthogonal polynomials (such as we will encounter in the
case of universal unitary ensembles), the Hermite polynomials can
be differentiated easily.  For fixed $k$ as $n\to \infty$:
\begin{equation*}
  H_{n-k}'(x) = 2(n-k)H_{n-k-1}(x) = 2nH_{n-k-1}(x)(1 + \bigO(n^{-1})).
\end{equation*}
This shortcut for differentiating Hermite polynomials allows us to simplify the integrand from
the most recent formula for $E(\G)$:
\begin{eqnarray*}
  E(\G) & = & \bigO\left( |I| \gamma^5 n^3 + |I|\gamma^9 n^5 \right)
    + \int_{I} \frac{4\gamma^3}{3} \left(\frac{e^{-2t^2}}{2^{2n}(n-1)!^2\pi} \right) n^4
       \left(1 + \bigO(n^{-1})\right) \\
  & & \qquad \times  \left[ \begin{array}{c}
        2(H_{n-1}^2 - H_nH_{n-2}) (H_{n-2}^2 - H_{n-1}H_{n-3}) \\
        + \frac23 (H_{n-1}H_{n-3} - H_nH_{n-4})(H_{n-1}^2 - H_nH_{n-2}) \\
        -(H_nH_{n-3} - H_{n-1}H_{n-2})^2
      \end{array} \right] dt.
\end{eqnarray*}

Notice that we factored the $(1+\bigO(n^{-1}))$ outside of the brackets
in order to save space.  Technically this is dangerous because
of the possibility that the terms inside brackets interfere destructively,
but we will see that destructive interference does not occur in our
case.

We use the Plancherel Rotach asymptotics in Theorem~\ref{PlancherelRotachTheorem}
to write the Hermite polynomials as an envelope times a phase function and
estimate each of the above groupings.  For the first grouping, the result is:
\begin{eqnarray*}
  & & (H_{n-1}H_{n-1} - H_nH_{n-2}) = e^{t^2} 2^{n} (n-1)^{n-1} e^{1-n} \left( 1-\frac{t^2}{2n} \right)^{-\frac12} \\
  & & \qquad \times \left[ \Re(\Phi(,n-1))^2 - \Re(\Phi(t,n))\Re(\Phi(t,n-2)) + \bigO(n^{-1}) \right].
\end{eqnarray*}

We will now use Lemma~\ref{consecutivePRA}.  Let
\begin{eqnarray*}
  u & = & \Phi(x,n) \\
  v & = & \left( \frac{x - i\sqrt{2n-x^2}}{\sqrt{2n}} \right).
\end{eqnarray*}
The in terms of these variables,
\begin{eqnarray*}
  \Phi(x,n) & = & u \\
  \Phi(x,n-1) & = & uv \\
  \Phi(x,n-2) & = & uv^2
\end{eqnarray*}
Using the identity $\Re[uv]^2 -\Re[u]\Re[uv^2] = \Im[v]^2$ for
unimodular $u$ and $v$, we have
\begin{equation*}
  (H_{n-1}^2 H_n H_{n-2}) = e^{t^2} 2^n (n-1)^{n-1} e^{1-n} \left(1 - \frac{t^2}{2n} \right)^{\frac12} \left[1 + \bigO(n^{-1}) \right].
\end{equation*}
Applying the same techniques to other groupings,
\begin{eqnarray*}
  (H_{n-1}^2 - H_nH_{n-2}) & = & e^{t^2} 2^n (n-1)^{n-1} e^{1-n}
      s^{-1} \left(s^2\right) \left[1 + \bigO(n^{-1}) \right]  \\
  (H_{n-2}^2 - H_{n-1}H_{n-3}) & = &  e^{t^2} 2^{n-1} (n-2)^{n-2} e^{2-n}
      s^{-1} \left(s^2\right) \left[1 + \bigO(n^{-1}) \right]  \\
  (H_{n-1}H_{n-3} - H_nH_{n-4}) & = &  e^{t^2} 2^{n-1} (n-2)^{n-2} e^{2-n}
      s^{-1} \left(3s^2 - 4s^4\right) \left[1 + \bigO(n^{-1}) \right]  \\
  (H_nH_{n-3} - H_{n-1}H_{n-2}) & = &  e^{t^2} 2^{n-\frac12} n^{\frac{n}{2}} (n-\frac32)^{n-\frac32} e^{\frac32 - n}
      s^{-1} \left(-2cs^2\right) \left[1 + \bigO(n^{-1}) \right]  \\
\end{eqnarray*}
where
\begin{eqnarray*}
  s & = & \sqrt{1-\frac{t^2}{2n}} \\
  c & = & t/\sqrt{2n}.
\end{eqnarray*}
Incorporating these newly derived asymptotics into the recent expression for $E(\G)$ yields:
\begin{eqnarray*}
  E(\G) & = & \bigO\left( I \gamma^5 n^3 + |I|\gamma^9 n^5 \right)
    + \int_{I} \frac{4\gamma^3 n^2}{3} s^{-2}
       \left(1 + \bigO(n^{-1})\right)
    \left[ \begin{array}{c}
      2(s^2)(s^2) \\
      \frac23 (s^2)(3s^2-4s^4) \\
      - (-2c s^2)^2 + \bigO(n^{-1})
    \end{array} \right] dt \\
  & = & \bigO\left( I \gamma^5 n^3 + |I|\gamma^9 n^5 \right)
    + \int_{I} \frac{4\gamma^3 n^2}{3} s^{-2}
       \left(1 + \bigO(n^{-1})\right)
    \left[\frac{4}{3} s^6 \right] dt
\end{eqnarray*}
Recall the Wigner Semicircle law from Appendix~\ref{WignerSemicircleLaw}
which says that
\begin{equation*}
  K_n(t,t) = \frac{\sqrt{2n}}{\pi}
    \sqrt{1 - \frac{t^2}{2n}} (1 + \bigO(n^{-1})).
\end{equation*}
We therefore see that:
\begin{equation*}
  E(\G) = \bigO\left( I \gamma^5 n^3 + |I|\gamma^9 n^5 \right) +
    \frac{\pi^2 \gamma^3}{9} \int_{I} K_n(t,t)^4 (1 + \bigO(n^{-1})) dt.
\end{equation*}
This agrees with our heuristic prediction.

%***************************************************************
\subsection{Controlling the Derivatives of $K_n(x,y)$ in the GUE Case}

The error terms in our estimates of the moments of $E(\G^k)$ are
bounded in terms of derivatives of $K_n(x,y)$. We now bound these
derivatives. Unlike the projection kernels of universal ensembles
to come, we can now exploit a shortcut when differentiating the
orthogonal polynomials.  In this case the orthogonal polynomials
are Hermite polynomials, and the shortcut is
\begin{equation*}
  \delx H_n(x) = 2 n H_{n-1}(x).
\end{equation*}

The following Theorem is the GUE version of
Lemma~\ref{derivativeBounds}.

\begin{lem} \label{trickyLemma}
  Fix $\epsilon > 0$.  For $(x,y)$ in the region
  \begin{equation*}
    R = \fatreg\times\fatreg,
  \end{equation*}
  a mixed partial derivative of $K_n(x,y)$
  of total degree k is $\bigO\left( n^{\frac{1+k}{2}} \right)$.
  When $|x-y| \geq n^{-\frac12}$, we have the stronger bound
  $K_n(x,y) = \bigO\left( \frac{1}{|x-y|} n^{\frac{k}{2}} \right)$.
  The constants implied by $\bigO$ depend only on $k$ and $\epsilon$.
\end{lem}

\begin{proof}
  Recall the closed form expression for $K_n(x,y)$:
  \begin{eqnarray*}
    K_n(x,y) & = & \frac{e^{-\frac{(x^2+y^2)}{2}}}{2^n(n-1)!\sqrt{\pi}}
        \frac{H_n(x) H_{n-1}(y) - H_{n-1}(x) H_n(y)}{x-y}
  \end{eqnarray*}

  Take $k$ partial derivatives of $K_n(x,y)$.  The result is a sum of several terms
  of the form:
  \begin{equation*}
    \frac{e^{-\frac{(x^2+y^2)}{2}}}{2^n(n-1)!\sqrt{\pi}}
    \frac{p_{k_0}(n) q_{k_1}(x,y) H_{n-k_2}(x) H_{n-k_3}(y)}{(x-y)^{k_4}},
  \end{equation*}
  where $p_{k_0}(n)$ and $q_{k_1}(x,y)$ are integral polynomials
  of total degrees $k_0$ and $k_1$ respectively.
  Before taking any derivatives, i.e. $k=0$, there are two terms, each
  with $(k_0,k_1,k_2+k_3,k_4) = (0,0,1,1)$.  Each new partial derivative
  can hit in one of four places, and increments the worst-case
  vector $(k_1,k_2+k_3,k_4)$.  Here are the four cases:
  \begin{itemize}
    \item It hits $e^{\frac{-(x^2+y^2)}{2}}$.  The increment is $(0,1,0,0)$.
    \item It hits $q$.  The increment is $(0,-1,0,0)$.
    \item It hits a Hermite polynomial.  The increment is $(1,0,1,0)$.
    \item It hits the denominator.  The increment is $(0,0,0,1)$.
  \end{itemize}

  We divide the region $R$ into two regions
  \begin{eqnarray*}
    R_{diag} & = & \left\{ (x,y) \in R \left| |x-y| < \frac{1}{\sqrt{n}}  \right. \right\} \\
    R_{bulk} & = & \left\{ (x,y) \in R \left| |x-y| \geq \frac{1}{\sqrt{n}}  \right. \right\}
  \end{eqnarray*}

  Our estimates are straightforward in the bulk region $R_{bulk}$.
  Consider any of the finitely many terms resulting from taking $k$ partial
  derivatives of $K_n(x,y)$:
  \begin{eqnarray*}
    & & \frac{e^{-\frac{(x^2+y^2)}{2}}}{2^n(n-1)!\sqrt{\pi}}
        \frac{p_{k_0}(n) q_{k_1}(x,y) H_{n-k_2}(x) H_{n-k_3}(y)}{(x-y)^{k_4}} \\
    & & \quad = \bigO\left( \frac{e^{-\frac{x^2+y^2}{2}}}{2^n n^{n-1} e^{-n} \sqrt{n}}
          \frac{ n^{k_0} n^{\frac{k_1}{2}}
              2^{\frac{n}{2}} n^{\frac{n-k_2}{2}} e^{-\frac{n}{2}} e^{\frac{x^2}{2}}
              2^{\frac{n}{2}} n^{\frac{n-k_3}{2}} e^{-\frac{n}{2}} e^{\frac{y^2}{2}}
          }{|x-y| n^{-\frac{k_4-1}{2}}}
      \right) \\
    & & \quad = \bigO\left( \frac{1}{|x-y|} n^{\frac{2k_0 + k_1 - k_2 - k_3 + k_4}{2}} \right) \\
    & & \quad = \bigO\left( \frac{1}{|x-y|} n^{\frac{1+k}{2}} \right) \\
    & & \quad = \bigO\left( n^{\frac{1+k}{2}} \right)
  \end{eqnarray*}

  Estimates for the diagonal region $R_{diag}$ are more subtle.

  Consider the term
  \begin{equation*}
    \frac{e^{-\frac{(x^2+y^2)}{2}}}{2^n(n-1)!\sqrt{\pi}}
    \frac{p_{k_0}(n) q_{k_1}(x,y) H_{n-k_2}(x) H_{n-k_3}(y)}{(x-y)^{k_4}}.
  \end{equation*}
  We expand $H_{n-k_3}(y)$ and $q_{k_1}(x,y)$ as a Taylor series centered at $x$:
  \begin{eqnarray*}
    q_{k_1}(x,y) & = & q_{k_1}(x,x) + (y-x) \left(\partial_2 q_{k_1}(x,x)\right) + \dots \\
      & & \quad + \frac{(y-x)^{k_4-1}}{(k_4-1)!} \left(\partial_2^{k_4-1} q_{k_1}(x,x)\right)
        + \frac{(y-x)^{k_4}}{k_4!} \left(\partial_2^{k_4} q_{k_1}(x,x_1)\right) \\
    H_{n-k_3}(y) & = & H_{n-k_3}(x) + (y-x)H_{n-k_3}'(x) + \dots \\
      & & \quad + \frac{(y-x)^{k_4-1}}{(k-1)!} H_{n-k_3}^{k_4-1}(x)
        + \frac{(y-x)^{k_4}}{k!} H_{n-k_3}^{k_4}(x_3),
  \end{eqnarray*}
  for some $x_1 \in (x,y)$ and $x_3 \in (x,y)$.
  Here's what happens
  if we select the remainder term from the expansion of $H_{n-k_3}(y)$, and a lower order term,
  $0 \leq j_1 < k_4$, in the expansion of $q_{k_1}(x,y)$:
  \begin{eqnarray*}
    & & \frac{e^{-\frac{(x^2+y^2)}{2}}}{2^n(n-1)!\sqrt{\pi}}
        \frac{p_{k_0}(n)
        \left( \frac{(y-x)^{j_1}}{j_1!} \left(\partial_2^{j_1} q_{k_1}(x,x)\right) \right)
            H_{n-k_2}(x) \left( \frac{(y-x)^{k_4}}{k_4!}H_{n-k_3}^{k_4}(x_3)\right) }
          {(x-y)^{k_4}} \\
    & & \quad = \bigO\left(   n^{ \frac{1+2k_0+k_1-j_1-k_2-k_3+k_4}{2} }   \right)
        = \bigO\left( n^{\frac{1+k-j_1}{2}} \right)
  \end{eqnarray*}
  Here's what happens when we take a lower order term, $0\leq j_3 < k_4$, from
  the expansion of $H_{n-k_3}(y)$, and the remainder term for the expansion
  of $q_{k_1}(x,y)$:
  \begin{eqnarray*}
    & & \frac{e^{-\frac{(x^2+y^2)}{2}}}{2^n(n-1)!\sqrt{\pi}}
        \frac{p_{k_0}(n)
        \left( \frac{(y-x)^{k_4}}{k_4!} \left(\partial_2^{k_4} q_{k_1}(x,x)\right) \right)
            H_{n-k_2}(x) \left( \frac{(y-x)^{j_3}}{j_3!}H_{n-k_3}^{j_3}(x_3)\right) }
          {(x-y)^{k_4}} \\
    & & \quad = \bigO\left(   n^{ \frac{1+2k_0+k_1-k_2-k_3}{2} }   \right)
        = \bigO\left( n^{\frac{1+k-k_4}{2}} \right)
  \end{eqnarray*}
  A term with lower order terms $j_1$ and $j_3$ with $j_1+j_3\geq k_4$
  also make a small contribution,
  \begin{equation*}
    \bigO\left( n^{\frac{1+2k_0+k_1-k_2-k_3+k_4-j_1}{2}} \right) = \bigO\left( n^{\frac{1+k-j_1}{2}}   \right).
  \end{equation*}

  Most subtle of all are the remaining terms
  \begin{equation*}
    \frac{e^{-\frac{(x^2+y^2)}{2}}}{2^n(n-1)!\sqrt{\pi}}
        \frac{p_{k_0}(n)
        \left( \frac{(y-x)^{j_1}}{j_1!} \left(\partial_2^{j_1} q_{k_1}(x,x)\right) \right)
            H_{n-k_2}(x) \left( \frac{(y-x)^{j_3}}{j_3!}H_{n-k_3}^{j_3}(x) \right)  }
          {(x-y)^{k_4}},
  \end{equation*}
  where $j_1 + j_3 < k_4$.
  Adding all of these terms together, we obtain, for each $1\leq l \leq k_4$,
  a polynomial in $x$ times
  \begin{equation*}
    \frac{e^{-\frac{x^2+y^2}{2}}}{(x-y)^l}.
  \end{equation*}
  Since $K_n(x,y)$ is smooth, there are no singularities on the diagonal, and all
  these polynomials in $x$ must vanish identically.  Thus all the terms
  with $j_1+j_3 < k_4$ cancel out against each other.
\end{proof}

%***************************************************************
\subsection{Derivation of Plancherel-Rotach Asymptotics}

We will use the method of steepest descent (see~\cite{bender})
to prove the Plancherel Rotach asymptotics
for Hermite polynomials.

\begin{thm}[Plancherel-Rotach] \label{PlancherelRotachTheorem}
  Fix $\epsilon >0$, and let $(x,n)$ be a pair for which
  $|x| < (\sqrt{2} - \epsilon)\sqrt{n}$.  Then:
  \begin{equation*}
    H_n(x) = e^{\frac{x^2}{2}} 2^{\frac{n+1}2} n^{\frac{n}{2}} e^{-\frac{n}{2}}
       \left(1-\frac{x^2}{2n} \right)^{-\frac14} \left[ Re(\Phi(x,n)) + \bigO(n^{-1}) \right],
  \end{equation*}
  where the phase $\Phi(x,n)$ is:
  \begin{equation*}
    \Phi(x,n) = e^{i \frac{\pi}{4}}
            \left( \frac{x}{\sqrt{2n}} + i \sqrt{1 - \frac{x^2}{2n} }  \right)^{n-\frac12}
            e^{-\frac12 i x \sqrt{2n-x^2}}
  \end{equation*}
  The constant implied by $\bigO$ depends only on $\epsilon$.
\end{thm}

\begin{proof}
  Begin with the well-known integral representation for Hermite polynomials:
  \begin{equation*}
    H_n(z) = \frac{2^n}{\sqrt{\pi}} \int_{-\infty}^{\infty} (z + it)^n e^{-t^2} dt.
  \end{equation*}
  One may verify this integral representation by checking that $H_0(x)$
  and $H_1(x)$ are correct, and then verifying the three term recurrence
  $H_n(t) = 2tH_{n-1}(t) - 2(n-1)H_{n-2}(t)$ for Hermite polynomials.

  First change variables so that the locations of
  the saddle points will remain constant as
  $n \to \infty$.  Let $w = \frac{z}{\sqrt{n}}$ and
  $u = \frac{t}{\sqrt{n}}$. The integral representation becomes:
  \begin{equation*}
    H_n(\sqrt{n}w) = \frac{2^n n^{\frac{n+1}{2}}}{\sqrt{\pi}} \intline e^{n(\log(w+iu)-u^2)} du.
  \end{equation*}
  For fixed
  $w\in\left(-(\sqrt{2}-\epsilon) ,\sqrt{2}-\epsilon \right)$
  let us apply stationary phase to the
  above integral in the $n \to \infty$ limit.
  The saddle points occur when $\frac{i}{w+iu} - 2u = 0$, or
  \begin{equation*}
    u_{\pm} = \frac12\left(iw \pm \sqrt{2-w^2} \right).
  \end{equation*}
  The second derivatives of the phase function at these points are
  \begin{equation*}
    f_{\pm} = 2\left( w^2-2 \mp i w \sqrt{2-w^2}\right).
  \end{equation*}
  The absolute value of $f_{\pm}$ is:
  \begin{equation*}
    2\sqrt{4-2w^2}.
  \end{equation*}

  \begin{figure}
    \begin{center}
      \includegraphics[scale=0.6]{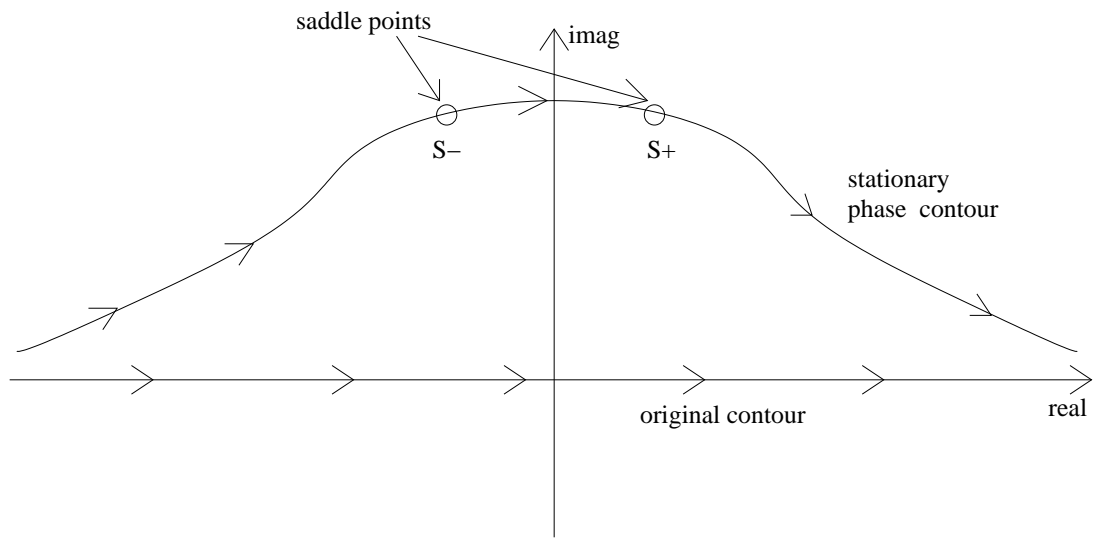}
      \caption{\label{plancherelRotachContour}
          The contour for the integral representation of Hermite polynomials
          deformed so that it passes through the saddle points.}
    \end{center}
  \end{figure}

  The slope of the stationary phase contour at the saddle point $u_{+}$ is
  \begin{equation*}
    \sqrt{\frac{-f_{+}}{|f_{+}|}},
  \end{equation*}
  with the sign chosen to have positive real part.

  The value of the integrand is $v_{\pm}^n$, where
  \begin{equation*}
    v_{\pm} = \left(\frac12 w \pm \frac12 i \sqrt{2-w^2} \right)
        e^{-\frac14 (iw \pm \sqrt{2-w^2})^2}.
  \end{equation*}

  The integral representation for $H_n(\sqrt{n} w)$ has the real
  axis as its contour of integration.  We change this contour to a stationary phase
  contour, that is $\Im(\log(w+iu)-u^2) = const$,
  which passes through the two saddle points.
  The stationary phase contour begins at $-\infty$, hits
  the left saddle point, crosses over to the right saddle point, then
  goes to $+\infty$, and overall looks a bit like a Gaussian.
  Reducing to small neighborhoods $S_{-}$ and $S_{+}$ of the saddle points
  produces only exponentially small errors.

  The contributions from the two neighborhoods are complex conjugates
  of each other, so we focus on $S_{+}$:

  \begin{eqnarray*}
    H_n(\sqrt{n} w) & = & 2\Re \left[ \frac{2^n n^{\frac{n+1}{2}}}{\sqrt{\pi}} \int_{S_{+}}
        e^{n \left( \log(w+i u) - u^2\right)} du \left(1 + \bigO(e^{-c\sqrt{n}})  \right)\right] \\
    & = & 2\Re \left[ \frac{2^n n^{\frac{n+1}{2}}}{\sqrt{\pi}} v_{+}^{n}
        \sqrt{\frac{-f_{+}}{|f_{+}|}} \sqrt{\frac{2\pi}{n|f_{+}|}}
        \left(1 + \bigO(n^{-1})\right) \right] \\
    & = & 2^{\frac{n+1}{2}} n^{\frac{n}{2}} e^{\frac12 n (w^2-1)} \left(1-\frac{w^2}{2}\right)^{-\frac14} \\
    & & \quad \times \quad \Re\left[ \left(\frac{w}{\sqrt{2}} + \frac{i \sqrt{2-w^2}}{\sqrt{2}} \right)^n e^{-\frac12 i w n \sqrt{2-w^2} }
        \sqrt{ \frac{\sqrt{2-w^2}}{\sqrt{2}} + \frac{iw}{\sqrt{2}} }      \right].
  \end{eqnarray*}
  Now recall that $w = \frac{z}{\sqrt{n}}$, and notice that two
  of the terms in the phase function can be combined:
  \begin{eqnarray*}
    H_n(z) & = & 2^{\frac{n+1}{2}} n^{\frac{n}{2}} e^{-\frac{n}{2}} e^{\frac{z^2}{2}}
        \left( 1 - \frac{z^2}{2n}  \right)^{\-\frac14} \\
    & & \quad \times \quad \Re\left[ e^{i \frac{\pi}{4}}
            \left( \frac{z}{\sqrt{2n}} + i \sqrt{1 - \frac{z^2}{2n} }  \right)^{n-\frac12}
            e^{-\frac12 i z \sqrt{2n-z^2}}  + \bigO(n^{-1})
        \right]
  \end{eqnarray*}
\end{proof}

We'll often use Plancherel-Rotach by comparing $H_n(x)$
for consecutive values of $n$.  The following simple
Lemma describes the way in which $\Phi$ varies with $n$:
\begin{lem} \label{consecutivePRA}
  Fix $\epsilon>0$.  Then for pairs $(x,n)$ for which
  $|x| < (\sqrt{2}-\epsilon)\sqrt{n}$,
  \begin{eqnarray*}
    \frac{\Phi(x,n+1)}{\Phi(x,n)} =
      \left(\frac{x+i\sqrt{2n-x^2}}{\sqrt{2n}} \right) + \bigO(n^{-1})
  \end{eqnarray*}
\end{lem}

\begin{proof}
  \begin{eqnarray*}
    \frac{\Phi(x,n+1)}{\Phi(x,n)} & = &
        \frac{\left( \frac{z}{\sqrt{2n+2}} + i \sqrt{1 - \frac{z^2}{2n+2} }   \right)^{n+\frac12}}
            {\left( \frac{z}{\sqrt{2n}} + i \sqrt{1 - \frac{z^2}{2n} }  \right)^{n-\frac12}}
        e^{ -\frac12 i z \sqrt{2n+2-z^2} + \frac12 i z \sqrt{2n-z^2} } \\
    & = & \left( 1 + i \frac{z}{\sqrt{2n}}\sqrt{1 - \frac{z^2}{2n+2}}
            - i \frac{z}{\sqrt{2n+2}}\sqrt{1 - \frac{z^2}{2n}} + \bigO(n^{-2})
        \right)^{n} \\
    & & \qquad \times
        \left( \frac{z}{\sqrt{2n+2}} + i \sqrt{1 - \frac{z^2}{2n+2} }   \right)
        e^{-\frac{i z}{2 \sqrt{2n-z^2}}}
        + \bigO(n^{-1}) \\
    & = & \left( \frac{z}{\sqrt{2n}} + i \sqrt{1 - \frac{z^2}{2n} } \right)
        + \bigO(n^{-1})
  \end{eqnarray*}
\end{proof}

%***************************************************************
\subsection{Derivation of Wigner Semicircle Law using Plancherel Rotach Asymptotics}
\label{WignerSemicircleLaw}

We use Plancherel Rotach asymptotics, Theorem~\ref{PlancherelRotachTheorem},
to prove the Wigner semicircle law.  This is similar to our estimation
of $E(\G)$ using Plancherel-Rotach asymptotics.

\begin{lem}
  For fixed $\epsilon > 0$ and $x\in \fatreg$,
  the diagonal of the kernel $K_n(x,y)$ satisfies the following
  asymptotics as $n\to \infty$:
  \begin{equation*}
    K_n(x,x) = \frac{\sqrt{2n}}{\pi}
        \sqrt{1 - \frac{x^2}{2n}} (1 + \bigO(n^{-1})).
  \end{equation*}
\end{lem}

\begin{proof}
  Since $K_n(x,y)$ is
  smooth everywhere, it is continuous on the
  diagonal $x=y$. By continuity,
  \begin{eqnarray*}
    & & K_n(x,x) = \lim_{h\to 0} K_n(x+h,x) \\
    & & \quad = \lim_{h\to 0} \frac{e^{-x^2}}{2^n(n-1)!\sqrt{\pi}}
      \left(\frac{H_n(x+h) H_{n-1}(x) - H_{n-1}(x+h) H_n(x)}{h}\right) \\
    & & \quad = \frac{e^{-x^2}}{2^n(n-1)!\sqrt{\pi}}
      \left(H_n'(x) H_{n-1}(x) - H_{n-1}'(x) H_n(x)\right) \\
    & & \quad = \frac{e^{-x^2}}{2^n(n-1)!\sqrt{\pi}}
      \left(2n H_{n-1}(x)^2 - 2(n-1)H_{n-2}(x) H_n(x)\right) \\
    & & \quad = \frac{\sqrt{2n}}{\pi}\left(1 - \frac{x^2}{2n} \right)^{-\frac12}
      \left(\Re(\Phi(x,n-1))^2 - \Re(\Phi(x,n)) \Re(\Phi(x,n-2)) + \bigO\left(\frac1n \right) \right) \\
  \end{eqnarray*}
  We will now use Lemma~\ref{consecutivePRA}.  Let
  \begin{eqnarray*}
    u & = & \Phi(x,n) \\
    v & = & \left( \frac{x - i\sqrt{2n-x^2}}{\sqrt{2n}} \right).
  \end{eqnarray*}
  The in terms of these variables,
  \begin{eqnarray*}
    \Phi(x,n) & = & u \\
    \Phi(x,n-1) & = & uv \\
    \Phi(x,n-2) & = & uv^2
  \end{eqnarray*}
  Using the identity $\Re[uv]^2 -\Re[u]\Re[uv^2] = \Im[v]^2$ for
  unimodular $u$ and $v$, we have
  \begin{eqnarray*}
    K_n(x,x)  & = & \frac{\sqrt{2n}}{\pi}\left(1 - \frac{x^2}{2n} \right)^{-\frac12}
      \left(\Re(uv)^2 - \Re(u) \Re(uv^2) + \bigO\left(\frac1n \right) \right) \\
    & = & \frac{\sqrt{2n}}{\pi} \left(1-\frac{x^2}{2n} \right)^{\frac12}
      \left(1 + \bigO\left( \frac1n \right) \right) \\
  \end{eqnarray*}
\end{proof}

%***************************************************************
\section{Appendix: Weyl integration}
\label{WeylAppendix}
%***************************************************************

The following discussion of Weyl integration is from~\cite{mehta}, p.62-63.

\begin{lem}
  Suppose that an ensemble of Hermitian matrices has the following
  joint probability density function for the matrix entries:
  \begin{equation*}
    C_n e^{-n \sum V(\theta_i)} dM,
  \end{equation*}
  where $V(x)$ is a potential, the $\theta_i$ are eigenvalues, and
  $C_n$ is a normalization constant.  Then the joint probability
  density function for the eigenvalues is
  \begin{equation*}
    \tilde{C}_n  \prod_{i<j} (\theta_j-\theta_i)^2
      e^{-n \sum V(\theta_i)} d\Theta,
  \end{equation*}
  where $\tilde{C}_n$ is a new normalization constant.
\end{lem}

Let $U$ be a unitary matrix and $\Theta$ be a diagonal matrix so
that
\begin{equation*}
  H = U\Theta U^*.
\end{equation*}
Except on a set of measure zero, the eigenvalues will be distinct.
Assume without loss of generality that the diagonal entries of
$\Theta$ are in ascending order.  We also assume without loss of
generality that the first nonzero entry in each column of $U$ is
positive real.  With these conventions, the $\Theta$ and $U$ are
uniquely determined by $H$, and $H$ is of course uniquely
determined by $U$ and$\Theta$.  Let $p_{\mu}$ be $n(n-1)$ real
variables which specify a unitary matrix.
\begin{equation*}
  J(\theta,p) = \det\left[
    \frac{\partial\left(H_{11}^{(0)},\dots,H_{nn}^{(0)},H_{12}^{(0)},H_{12}^{(1)},\dots \right)}
    {\partial\left( \theta_1,\theta_2,\dots,\theta_n,p_1,p_2,\dots,p_{n(n-1)}   \right)   }
  \right]
\end{equation*}
If one takes $J(\theta,p)$ above, and integrates out the variables
$p$, then multiplying by the joint probability density function
for matrix entries, one has the joint probability density function
for the eigenvalues themselves. It will turn out that
$J(\theta,p)$ is $\prod_{i<j} (\theta_j-\theta_i)^2$ times a function of
$p$, so that all that remains is a function of
$\theta$ times a constant.

We will multiply the above $n^2\times n^2$ matrix for
$J(\theta,p)$ by another matrix of the same dimensions depending
only on $U$, and the determinant of the result will be
$\prod_{i<j} (\theta_j-\theta_i)^2$ times a function depending
only on $U$.

Now $UU^{*}=1$, so for each variable $p_{\mu}$,
\begin{equation*}
  S^{(\mu)} = U^{*} \frac{\partial U}{\partial p_{\mu}} = - \frac{\partial U^{*}}{\partial p_{\mu}} U
\end{equation*}
is conjugate Hermitian.  Then
\begin{equation*}
  U^{*} \frac{\partial H}{\partial p_{\mu}} U = S^{\mu} \Theta -
  \Theta S^{\mu},
\end{equation*}
or
\begin{eqnarray*}
  & & \sum_{j,k} U^{*}_{\alpha,j} \frac{\partial H_{j,k}}{\partial p_{\mu}} U_{k,\beta}
    = S_{\alpha,\beta}^{\mu}(\theta_\beta-\theta_{\alpha}) \\
  & & \quad = \sum_{j} \frac{\partial H_{j,j}}{\partial p_{\mu}}
      \left[ U_{\alpha,j}^* U_{j,\beta}\right]
    + \sum_{j<k} \Re\left( \frac{\partial H_{j,k}}{\partial p_{\mu}} \right)
      \left[ U_{\alpha,j}^* U_{k,\beta} + U_{\alpha,k}^* U_{j,\beta} \right] \\
  & & \quad \qquad + i \sum_{j<k} \Im\left( \frac{\partial H_{j,k}}{\partial p_{\mu}} \right)
      \left[ U_{\alpha,j}^* U_{k,\beta} - U_{\alpha,k}^* U_{j,\beta} \right]
\end{eqnarray*}
Similarly,
\begin{equation*}
  \sum_{j,k} U^{*}_{\alpha,j} \frac{\partial H_{j,k}}{\partial
  \theta_{\gamma}} U_{k,\beta} =
  \delta_{\alpha,\beta,\gamma}.
\end{equation*}

Writing these two equations in matrix form,

\begin{eqnarray*}
  & & \quad \left[ \begin{array}{c c c}
      \frac{\partial H_{jj}^{(0)}}{\partial \theta_{\gamma}}
        & \frac{\partial H_{jk}^{(0)}}{\partial \theta_{\gamma}}
        & \frac{\partial H_{jk}^{(1)}}{\partial \theta_{\gamma}} \\
      \frac{\partial H_{jj}^{(0)}}{\partial p_{\mu}}
        & \frac{\partial H_{jk}^{(0)}}{\partial p_{\mu}}
        & \frac{\partial H_{jk}^{(1)}}{\partial p_{\mu}} \\
    \end{array} \right] \\
  & & \qquad \times \quad \left[ \begin{array}{c c c}
      [U_{\alpha,j}^{*} U_{j,\alpha}]
        & \Re[U_{\alpha,j}^{*} U_{j,\alpha}]
        & \Im[U_{\alpha,j}^{*} U_{j,\alpha}] \\

      \ [U_{\alpha,j}^{*} U_{k,\beta} + U_{\alpha,k}^{*} U_{j,\beta} ]
        & \Re[U_{\alpha,j}^* U_{k,\beta} + U_{\alpha,k}^* U_{j,\beta} ]
        & \Im[U_{\alpha,j}^* U_{k,\beta} + U_{\alpha,k}^* U_{j,\beta} ] \\
      \ [U_{\alpha,j}^* U_{k,\beta} - U_{\alpha,k}^* U_{j,\beta} ]
        & -\Im[U_{\alpha,j}^* U_{k,\beta} - U_{\alpha,k}^* U_{j,\beta} ]
        & \Re[U_{\alpha,j}^* U_{k,\beta} - U_{\alpha,k}^* U_{j,\beta} ]
    \end{array} \right] \\
  & & \quad = \quad \left[ \begin{array}{c c c}
      I & 0 & 0 \\
      0 & \Re(S_{\alpha,\beta}^{\mu}) (\theta_{\beta}-\theta_{\alpha})
        & \Im(S_{\alpha,\beta}^{\mu}) (\theta_{\beta}-\theta_{\alpha})
    \end{array} \right]
\end{eqnarray*}

%***************************************************************
\section{Appendix: The Method of Gaudin}
\label{intoutapp}
%***************************************************************

Suppose that $\Fsp$ is a symmetric function
of $n$ variables, which only depends on the variables $m$ at a
time.  Applying $\Fs$ to the eigenvalues of an \nbyn
matrix chosen randomly from some ensemble makes $\Fs$ a random variable.  We
express the expected value of $\Fs$ as an $m$ dimensional
integral.

%***************************************************************
\subsection{The Universal Unitary Ensemble}

Let $V(x)$ be a real-analytic potential function growing sufficiently
rapidly at infinity.  Then for the ensemble $UUE_n$, the j.p.d.f.
of the eigenvalues is
\begin{equation*}
  \kapone  \prod_{0\leq j<k<n} (t_j-t_k)^2 e^{-n \sum V(t_j)} dT,
\end{equation*}
where $\kapone$ is a constant depending only on $V(x)$ and $n$.
See Appendix~\ref{WeylAppendix}.

This leads to an expression for the expected value of $\Fs$.
\begin{eqnarray*}
  E(\Fs) & = & \kapone \int \Fs(t_1,\dots,t_n) \\
  & & \quad \times
      \prod_{0\leq j<k<n} (t_j-t_k)^2
      e^{-n\sum V(t_j)} dT
\end{eqnarray*}

\begin{thm} \label{intoutGue}
  Let $m \geq 1$ and $F$ a function of $m$ variables.
  Define the symmetrization of $F$ to $n$ variables as:
  \begin{equation*}
    \Fsp = \sum_{i_1,i_2,\dots,i_m \mbox{ distinct}} \Fp.
  \end{equation*}
  Then the expected value of $\Fs$, applied to the eigenvalues of
  a random unitary matrix from $UUE_n$, can be expressed as
  an integral in $m$ variables:
  \begin{eqnarray*}
    E(\Fs) & = & \int F(t_1,t_2,\dots,t_m)
      \det_{m\times m} \left[ K_n(t_j,t_k) \right]
      dt_1 \dots dt_m.
  \end{eqnarray*}
  The projection kernel $K_n(x,y)$ appearing inside the determinant will
  be defined below.
\end{thm}

%***************************************************************
\subsection{Definition of the Projection Kernel $K_n(x,y)$}

We form the sequence of {\it normalized orthogonal polynomials} by
performing Gram-Schmidt orthonormalization on the sequence
$1,x,x^2,x^3,\dots$, with respect to the inner product
\begin{equation*}
  (f,g) = \intline f(x) g(x) e^{-nV(x)} dx.
\end{equation*}
We always choose the sign of the leading coefficient of the polynomials
to be positive.  We call th resulting sequence $\phi_{n,j}(x)$ the
normalized orthogonal polynomials with respect to the
measure $e^{-nV(x)}dx$.

For each polynomial we have a choice of sign -- we choose so that
the leading coefficient of each polynomial is positive. Let the
resulting sequence of polynomials be $\phi_j$.  Thus,
\begin{equation*}
  \intline \phi_j(x) \phi_k(x) e^{-nV(x)} dx = \delta_{j,k}.
\end{equation*}

An alternative normalization is to choose the leading coefficient
of each polynomial to be $1$. This would result in {\it monic
orthogonal polynomials.}  We will use $\phi_{n,j}$ to denote
normalized orthogonal polynomials and $\pi_{n,j}$ to denote monic
orthogonal polynomials.  When the context is clear, we will omit the
subscript $n$.

The most familiar examples of orthogonal polynomials are Hermite
polynomials.  Hermite polynomials are orthogonal polynomials with
respect to $e^{-x^2}dx$.  They are normalized so that the leading
term is $2^jx^j$.

\begin{lem}[Christoffel-Darboux + three-term recurrence relation] \label{chrisdar}
  Using the above notation, let
  \begin{eqnarray*}
    a_k & = & \int x \phi_k^2(x) e^{-nV(x)} dx \qquad \mbox{for} \quad k\geq 0\\
    b_k & = & \int x \phi_{k+1}(x)\phi_k(x) e^{-nV(x)} dx \qquad \mbox{for} \quad k\geq 0.
  \end{eqnarray*}
  Then
  \begin{eqnarray*}
    x\phi_0(x) & = & b_0\phi_1(x) + a_0 \phi_0(x) \\
    x\phi_k(x) & = & b_k\phi_{k+1}(x) + a_k \phi_k(x) +
      b_{k-1}\phi_{k-1}(x) \quad \mbox{ for } k\geq 1.
  \end{eqnarray*}
  and
  \begin{equation*}
    \sum_{j=0}^{n-1} \phi_j(x)\phi_j(y)
    = b_{n-1} \frac{\phi_n(x)\phi_{n-1}(y) - \phi_n(y)\phi_{n-1}(x)}{x-y}.
  \end{equation*}
\end{lem}

\begin{proof}
  To prove the three term recurrence relation, observe that
  $x\phi_k(x)$ is a polynomial of degree $(k+1)$, so it is
  orthogonal to all polynomials of degree higher than $(k+1)$.  It
  is also orthogonal to polynomials of degree $j$ less than
  $(k-1)$:
  \begin{equation*}
    \int (x\phi_k(x))\phi_j(x)d\mu(x) = \int \phi_k(x)
    (x\phi_j(x)) d\mu(x) = 0.
  \end{equation*}
  The second integral above is zero because $(x\phi_j)$ is a
  polynomial of degree at most $(k-1)$, and $\phi_k(x)$ is
  orthogonal to polynomials of such low degree.
  The Christoffel-Darboux formula itself is proven by induction.
  The inductive step is the following:
  \begin{eqnarray*}
    & & (x-y) \left( b_{k} \frac{\phi_{k+1}(x)\phi_{k}(y)
      - \phi_{k+1}(y)\phi_{k}(x)}{x-y} \right. \\
    & & \qquad \left. - b_{k-1} \frac{\phi_k(x)\phi_{k-1}(y)
      - \phi_k(y)\phi_{k-1}(x)}{x-y} \right) \\
    & & = \left( b_k\phi_{k+1}(x)
      + b_{k-1}\phi_{k-1}(x) \right) \phi_k(y) \\
    & & \qquad - \phi_k(x)\left( b_k\phi_{k+1}(y)
      + b_{k-1}\phi_{k-1}(y) \right) \\
    & & = \left( b_k\phi_{k+1}(x)
      + b_{k-1}\phi_{k-1}(x) + a_k\phi_k(x) \right) \phi_k(y) \\
    & & \qquad - \phi_k(x)\left( b_k\phi_{k+1}(y)
      + b_{k-1}\phi_{k-1}(y) + a_k\phi_k(y) \right)\\
    & & =  x\phi_k(x)\phi_k(y) - \phi_k(x) y \phi_k(y) \\
    & & = (x-y) \phi_k(x)\phi_k(y)
  \end{eqnarray*}
\end{proof}

For computations elsewhere, we find it more convenient to use $\eta_{n,j}(x)$, where
\begin{equation*}
  \eta_{n,j}(x) = \frac{\phi_{n,j}(x)}{e^{\frac{n}{2}V(x)}}.
\end{equation*}
The advantage of using $\eta_{n,j}(x)$ is that they are orthonormal with
respect to Lebesgue measure on $\RR$.  Again, we will omit the subscript $n$
when the context is clear.  Using this notation, let
\begin{eqnarray*}
  K_n(x,y) & = & \sum_{j=0}^{n-1} \eta_j(x)\eta_j(y) \\
  & = & b_{n-1} \frac{\eta_n(x)\eta_{n-1}(y) - \eta_n(y)\eta_{n-1}(x)}{x-y}.
\end{eqnarray*}

%***************************************************************
\subsection{Proof of Theorem~\ref{intoutGue}}

We will use the following Lemma to express $E(\Fs)$ as the
integral of an $n\times n$ determinant:

\begin{lem} \label{vantokerGue}
  Let $K_n(x,y)$ be the projection kernel above.  Then:
  \begin{equation*}
    \prod_{0\leq j<k<n} (t_j-t_k)^2
      e^{-n \sum V(t_j)}
    = \kaptwo^2 \det_{n\times n} [K_n(t_j,t_k)],
  \end{equation*}
  where $\kaptwo$ depends only on the potential $V(x)$ and $n$.
\end{lem}

\begin{proof}
  We recognize the component $\prod_{0\leq j<k<n} (t_j-t_k)$ as
  a Vandermonde determinant:
  \begin{equation*}
    \prod_{1\leq j<k\leq n} (t_j-t_k)  =
        \det \left[\begin{array} {c c @{\:\dots\:} c}
            1 & t_1 & t_1^{n-1} \\
            \vdots & \vdots & \vdots \\
            1 & t_n & t_n^{n-1}
        \end{array}  \right]
  \end{equation*}
  Performing column operations we replace the monomials $t_j^k$
  with orthogonal polynomials $\pi_{n,k}(t_j)$.
  To replace the monic orthogonal polynomials $\pi_{n,k}$
  with the normalized orthogonal polynomials $\phi_{n,k}$,
  we divide by the leading coefficients.
  We next replace $\phi_j$ with $\eta_j$ in order
  to absorb a factor of  $e^{-\frac{n}{2} \sum V(t_j)}$:
  \begin{eqnarray*}
    & & \prod_{1\leq j<k\leq n} (t_j-t_k)    e^{-\frac{n}{2} \sum V(t_j)}    \\
    & & \quad =
        \det \left[\begin{array} {c c @{\:\dots\:} c}
            \pi_0(t_1) & \pi_1(t_1) & \pi_{n-1}(t_1) \\
            \vdots & \vdots & \vdots \\
            \pi_0(t_n) & \pi_1(t_n) & \pi_{n-1}(t_N)
        \end{array}  \right]
        e^{-\frac{n}{2} \sum V(t_j)} \\
    & & \quad = \kaptwo
        \det \left[\begin{array} {c c @{\:\dots\:} c}
            \phi_0(t_1) & \phi_1(t_1) & \phi_{n-1}(t_1) \\
            \vdots & \vdots & \vdots \\
            \phi_0(t_n) & \phi_1(t_n) & \phi_{n-1}(t_N)
        \end{array}  \right]
        e^{-\frac{n}{2} \sum V(t_j)} \\
    & & \quad = \kaptwo
        \det \left[\begin{array} {c c @{\:\dots\:} c}
            \eta_0(t_1) & \eta_1(t_1) & \eta_{n-1}(t_1) \\
            \vdots & \vdots & \vdots \\
            \eta_0(t_n) & \eta_1(t_n) & \eta_{n-1}(t_N)
        \end{array}  \right] \\
    & & \quad = \kaptwo \det [M].
  \end{eqnarray*}
  Squaring the above formula,
  \begin{eqnarray*}
    & & \prod_{1\leq j<k\leq n} (t_j-t_k)^2    e^{-\frac{n}{2} \sum V(t_j)} \\
    & & \quad = \kaptwo^2 \det(M)^2 \\
    & & \quad = \kaptwo^2 \det(MM^T)^2 \\
    & & \quad = \kaptwo^2 \det_{n\times n} [K_n(t_j,t_k)].
  \end{eqnarray*}
\end{proof}

We now use the results of Lemma~\ref{vantokerGue} to express $E(\Fs)$
as the integral of an $n\times n$ determinant:

\begin{eqnarray*}
  E(\Fs) & = & \frac{\kapone}{\kaptwo^2} \int \Fs(t_1,\dots,t_n)
      \det_{n\times n}\left[ K_n(t_j,t_k) \right] dt_1 dt_2 \dots dt_n
\end{eqnarray*}

Unfortunately, this is
an $n$ dimensional integral instead of $m$ dimensional.
We use the next Lemma to ``integrate out'' the extra dimensions one at a time.

\begin{lem} \label{oneoutlemGue}
  Let $K_n(x,y) = \sum_{j=0}^{n-1} \eta_j(x)\eta_j(y)$ be the
  projection of $n$ functions orthonormal with respect
  to Lebesgue measure.  Then integrating
  \begin{equation*}
    \det_{l\times l} [K_n(t_j,t_k)]
  \end{equation*}
  over the last variable $t_l$ is the same as multiplying
  by $(n-l+1)$.  That is,
  \begin{equation*}
    \intline \det_{l\times l}[K_n(t_j,t_k)] dt_l =
        (n-l+1) \det_{(l-1)\times(l-1)}[K_n(t_j,t_k)].
  \end{equation*}
\end{lem}

\begin{proof}
  Expand by minors along the last column.  The last entry in this column
  contributes
  \begin{eqnarray*}
    & & \det_{(l-1)\times(l-1)}[K_n(t_j,t_k)]  \left( \int K_n(t_l,t_l) dt_l \right) \\
    & & \qquad = n \det_{(l-1)\times(l-1)}[K_n(t_j,t_k)].
  \end{eqnarray*}

  Now for the $(m-1)$ other minors.  Consider the minor by expanding along
  $K_n(t_j,t_l)$, where $1\leq j \leq (l-1)$:
  \begin{equation*}
     (-1)^{l-j} \int \det \left[
         \begin{array} {c c @{\:\dots\:} c}
            K_n(t_1,t_1) & K_n(t_1,t_2) & K_n(t_1,t_{l-1}) \\
            \vdots & \vdots & \vdots \\
            K_n(t_{j-1},t_1) & K_n(t_{j-1},t_2) & K_n(t_{j-1},t_{l-1}) \\
            K_n(t_{j+1},t_1) & K_n(t_{j+1},t_2) & K_n(t_{j+1},t_{l-1}) \\
            \vdots & \vdots & \vdots \\
            K_n(t_l,t_1) & K_n(t_l,t_2) & K_n(t_l,t_{l-1})
         \end{array}
       \right] K_n(t_j,t_l) dt_l
  \end{equation*}
  Each of the $(l-1)!$ terms in the determinant has exactly one term in row
  $l$, say $K-n(t_l,t_i)$, but otherwise has no dependence of $t_l$.  To
  integrate this term against $K_n(t_j,t_l) dt_l$, observe that
  $K_n(x,y)$ is a projection kernel:
  \begin{equation*}
    \int K_n(x,y) K_n(y,z) dy = K_n(x,z).
  \end{equation*}
  The resulting contribution is:
  \begin{eqnarray*}
     & & (-1)^{l-j} \det \left[ \begin{array} {c c @{\:\dots\:} c}
            K_n(t_1,t_1) & K_n(t_1,t_2) & K_n(t_1,t_{l-1}) \\
            \vdots & \vdots & \vdots \\
            K_n(t_{j-1},t_1) & K_n(t_{j-1},t_2) & K_n(t_{j-1},t_{l-1}) \\
            K_n(t_{j+1},t_1) & K_n(t_{j+1},t_2) & K_n(t_{j+1},t_{l-1}) \\
            \vdots & \vdots & \vdots \\
            K_n(t_j,t_1) & K_n(t_j,t_2) & K_n(t_j,t_{l-1})
        \end{array}  \right] \\
     & & \quad = - \det \left[ \begin{array} {c c @{\:\dots\:} c}
            K_n(t_1,t_1) & K_n(t_1,t_2) & K_n(t_1,t_{l-1}) \\
            \vdots & \vdots & \vdots \\
            K_n(t_{l-1},t_1) & K_n(t_{l-1},t_2) & K_n(t_{l-1},t_{l-1})
        \end{array}  \right]
  \end{eqnarray*}
  Adding the contribution from the last minor to the other $(m-1)$ minors
  yields the conclusion of the Lemma.
\end{proof}

Using Lemmas~\ref{vantokerGue} and~\ref{oneoutlemGue} allows
an induction on the following statement for $m \leq l \leq n$:
\begin{equation*}
  E(\Fs) = \frac{\kapone}{\kaptwo^2}
    \frac{(n-l)!}{(n-m)!} \int F(t_1,t_2,\dots,t_m)
      \det_{l \times l} [K_n(t_j,t_k)]  dt_1 dt_2\dots dt_l
\end{equation*}
To prove the base case $l=n$, observe that $\Fs$ is a sum of $\frac{n!}{(n-m)!}$
terms.  Each of these terms contributes as much to $E(\Fs)$ as
the term $F(t_1,t_2,\dots,t_m)$ does.  Specializing to the case $l=m$,
we obtain:
\begin{eqnarray*}
  E(\Fs) & = & \frac{\kapone}{\kaptwo^2}
    \int F(t_1,t_2,\dots,t_m)
      \det_{m\times m} \left[ K_n(t_j,t_k) \right]
      dt_1 \dots dt_m.
\end{eqnarray*}
 This proves Theorem~\ref{intoutGue}, except for the presence of
a multiplicative constant depending only on $V$ and $n$.  To find this
constant, let $F(x) \equiv 1$ be a function of one variable.  Then
\begin{equation*}
  E(\Fs) = \sum_{1\leq j \leq n} 1 = n
\end{equation*}
and
\begin{equation*}
  \int K_n(x,x) dx = n,
\end{equation*}
establishing that the multiplicative constant is $1$.

%***************************************************************
\subsection{The Circular Unitary Ensemble}
\label{intoutCUEsubsec}

Applying the method of Gaudin to the circular unitary ensemble (CUE)
yields the following Theorem, which differs from
Theorem~\ref{intoutGue} only in the definition of
the projection kernel $K_n(x,y)$.

\begin{thm} [Method of Gaudin, CUE version] \label{intoutCue}
  Let $m \geq 1$ and $F$ a function of $n$ variables.
  Define the symmetrization of $F$ to $n$ variables as:
  \begin{equation*}
    \Fsp = \sum_{i_1,i_2,\dots,i_m \mbox{ distinct}} \Fp.
  \end{equation*}
  Then the expected value of $\Fs$, applied to the eigenvalues of
  a random unitary matrix from $U_n$, can be expressed as
  an integral in $m$ variables:
  \begin{eqnarray*}
    E(\Fs) & = & \int F(t_1,t_2,\dots,t_m)
      \det_{1\leq j,k \leq m} \left[ K_n(t_j,t_k) \right]
      dt_1 \dots dt_m,
  \end{eqnarray*}
  where the projection kernel is
  \begin{equation*}
    K_n(x,y) = \frac{1}{2\pi} \frac{e^{in(x-y)}-1}{e^{i(x-y)}-1}.
  \end{equation*}
\end{thm}

\begin{proof}
  For the CUE, the j.p.d.f. for the eigenvalues is
  \begin{eqnarray*}
    & & \kapthree \prod_{j<k} |e^{i\theta_j}-e^{i\theta_k}|^2 \\
    & & = \kapthree \prod_{j<k} (e^{i\theta_j}-e^{i\theta_k})
        \prod_{j<k} (e^{-i\theta_j}-e^{-i\theta_k}) \\
    & & = \kapthree \left[ \begin{array}{c c @{\dots} c}    
          1 & e^{i\theta_1} & e^{i(n-1)\theta_1} \\
          \vdots & \vdots & \vdots \\
          1 & e^{i\theta_n} & e^{i(n-1)\theta_n}
        \end{array} \right]
        \left[ \begin{array}{c @{\dots} c}
          1 & 1 \\
          e^{-i\theta_1} & e^{-i\theta_n} \\
          \vdots & \vdots \\
          e^{-i(n-1)\theta_1} & e^{-i(n-1)\theta_n}
        \end{array} \right]\\
    & & = \kapfour \det_{n\times n} [K_n(\theta_j,\theta_k)],
  \end{eqnarray*}
  where
  \begin{equation*}
    K_n(\theta_1,\theta_2) = \frac{1}{2\pi} \frac{e^{i n (\theta_1-\theta_2)} - 1}{e^{i(\theta_1-\theta_2)} - 1}.
  \end{equation*}
  Proceeding as in the proof of Theorem~\ref{intoutGue}, one must verify that
  $K_n$ is a projection kernel:
  \begin{equation*}
    \int K_n(\theta_1,\theta_2) K_n(\theta_2,\theta_3) dy = K_n(\theta_1,\theta_3).
  \end{equation*}
  It is this requirement which dictates the choice of normalization constant for $K_n$.
  The remainder of the proof is identical to the proof of Theorem~\ref{intoutGue}.
\end{proof}

%***************************************************************
\section{Appendix: One Possible Approach to the Maximum Spacing Problem}
\label{MaxSpace}
%***************************************************************

In this paper, we considered the following question: given an
interval $I$ and an ensemble of random matrices, what is the
minimum spacing between two consecutive eigenvalues in the
interval $I$?

It is natural to ask the same question about maximum spacing: what
is the {\it maximum} spacing between two consecutive eigenvalues
in the interval $I$?

We suspect that maximum spacing is a more difficult question than
minimum spacing; at least we have not been able to deal with it
yet.  We begin with heuristic predictions based on the tail
probabilities of the GUE consecutive spacing distribution. As far
as we can tell, even these tail probabilities have not been proven
rigorously.

In a region of unit mean spacing, in an appropriate limit, the
probability that an interval $I$ is free of eigenvalues is the
Fredholm determinant of the integral operator
\begin{equation*}
  \left(Id - A_I \psi\right)(x) = \psi(x) - \int_I K(x,y) \psi(y) dy,
\end{equation*}
where $K(x,y)$ is the sine kernel
\begin{equation*}
  K(x,y) = \frac{\sin(x-y)}{\pi(x-y)}.
\end{equation*}
In the above statement and formula, we could replace
$I$ by a union $J$ of intervals.  See~\cite{widom-fredholm-multiple}
Let $F(J)$ be the value of this
Fredholm determinant.

For a single interval, $F(I)$ depends only on the length of the
interval $s = |I|$.  Let $E_2(0;s)$ be this probability.  Then the
consecutive spacing distribution is
\begin{equation*}
  p_2(0;s) = \frac{\partial^2}{\partial s^2} E_2(0;s).
\end{equation*}
See~\cite{mehta}, Chapter 5 and Appendix 13.
In the minimum spacing problem, we used the Taylor expansions of
$p_2(0;s)$ centered at $s=0$:
\begin{eqnarray*}
  E_2(0;2) & = & 1 - s + \frac{\pi^2 s^4}{36}
    - \frac{\pi^4 s^6}{675} + \frac{\pi^6 s^8}{17640}
    + \dots \\
  p_2(0;s) & = & \frac{\pi^2 s^2}{3} - \frac{2 \pi^4 s^4}{45}
    + \frac{\pi^6 s^6}{315} + \dots
\end{eqnarray*}
For maximum spacing, we require the asymptotics of $p_2(0;s)$ for
large $s$.  Dyson~\cite{dyson-fredholm} derived the following expansion for large $s$:
\begin{equation*}
  \log (E_2(0;s)) = -\frac{s^2}{8} - \frac14\log(s) + 3\zeta'(-1)
  + \frac13\log(2) + \littleo(1).
\end{equation*}
Unfortunately Dyson's methods did not yield rigorous asymptotics.
The first rigorous asymptotics are more limited, and were obtained
only recently by Widom~\cite{widom-orthopoly}:
\begin{eqnarray*}
  & & \frac{\partial}{\partial s} \log(E_2(0;s))
    = \frac{s}{4} + \bigO(1) \\
  & & \implies \log(E_2(0;s)) = -\frac{s^2}{8} + \bigO(s)
\end{eqnarray*}
Since we are only doing heuristics, we choose to use Dyson's
expansion.  Truncating this expansion and differentiating twice we
obtain asymptotics for $p_2(0;s)$ for large $s$:
\begin{eqnarray*}
  \int_s^{\infty} p_s(0;t) dt
    & = & \left(\frac{s^{\frac34}}{4} + \frac{s^{-\frac34}}{4}  \right)
    e^{ -\frac{s^2}{8} + 3 \zeta'(-1) + \frac13\log(2) } \\
  p_2(0;s) & = & \left(\frac{s^{\frac74}}{16} -\frac{s^{-\frac14}}{8}  \right)
    e^{ -\frac{s^2}{8} + 3 \zeta'(-1) + \frac13\log(2) } \\
  & = & \frac{s^{\frac74}}{16}
    e^{ -\frac{s^2}{8} + 3 \zeta'(-1) + \frac13\log(2) }
\end{eqnarray*}

As in the minimum spacing problem,
we assume that the consecutive spacings are independent
random variables, aware that the assumption is false but
convenient.  Recycling notation from the minimum spacing problem,
let $\G$ be the number of consecutive spacings, of a random matrix
from an ensemble of $n\times n$ matrices, in $I_n$, which are
larger than $\gamma$.  Let $Z_n$ be the maximum spacing itself.
Then, assuming a sum of independent unlikely events, $\G$ will be
approximately Poisson with mean
\begin{equation*}
  \mu = \int_I K_n(x,x) \frac{(\gamma K_n(x,x))^{\frac34}}{4}
    e^{ -\frac{(\gamma K_n(x,x))^2}{8} + 3\zeta'(-1) + \frac13\log(2) } dx
\end{equation*}
Then $\Pr(Z_n<\gamma) = \Pr(\G=0) \approx e^{-\mu}$ is the probability
that the maximum spacing is less than $\gamma$.
We specialize to the case of constant eigenvalue density $1$, so that 
\begin{equation*}
  \mu = |I| \frac{\gamma^{\frac74}}{16}
    e^{ -\frac{\gamma^2}{8} + 3 \zeta'(-1) + \frac13\log(2) }.
\end{equation*}

Our proposed strategy for recovering $\Pr(Z_n<\gamma)$ is less
direct than in the minimum spacing problem.  Let $\Ss$ be the set
\begin{eqnarray*}
  \Ss & = & \left\{ x \left|\begin{array}{c}
      \ [x,x+\gamma] \subset I \quad \mbox{ and} \\
      \ [x,x+\gamma] \quad \mbox{ contains\ \ no\ \ eigenvalues.}
    \end{array} \right. \right\} \\
  & = & \left\{ x \left|\begin{array}{c}
      \ x\in \tilde{I} \quad \mbox{ and} \\
      \ [x,x+\gamma] \quad \mbox{ contains\ \ no\ \ eigenvalues.}
    \end{array} \right. \right\} \\
\end{eqnarray*}
Let $\X = |\Ss|$ be the size of the set $\Ss$.  We wish to estimate
$\Pr(\X=0)$, which we bound using the following inequalities:
\begin{equation*}
  \Pr(X_{\gamma-\epsilon} \leq \epsilon)
    \leq \Pr(\X=0)
    \leq \Pr(\X\leq \epsilon).
\end{equation*}
It may be possible to estimate $\Pr(\X\leq \epsilon)$ and
$\Pr(X_{\gamma-\epsilon} \leq \epsilon)$ using
estimates for the moments of $\X$ and $X_{\gamma-\epsilon}$, respectively.

The $(k)$th moment of $\X$ is
\begin{equation*}
  \int_{\tilde{I}^k}
    F\left( [x_1,x_1+\gamma] \cup \dots \cup [x_k,x_k+\gamma] \right)
    dx_1 dx_2 \dots dx_k.
\end{equation*}
We expect that the dominant contributions to this integral occur
when the $x_j$ are ``clustered together'' according to some
partition of $\{1,2, \dots,k \}$.  For example, consider the
contribution to $E(\X^4)$ corresponding to the partition
$(13|2|4)$.  The region $\Omega_{(13|2|4)} \subset \tilde{I}^4$
corresponding to this partition is
\begin{equation*}
  \Omega_{(13|2|4)} \subset \tilde{I}^4
    = \left\{ (x_1,x_2,x_3,x_4)\subset\tilde{I}^4 \left| \begin{array}{c}
        |x_1-x_3| \quad \mbox{small} \\
        |x_1-x_2| \quad \mbox{large} \\
        |x_1-x_4| \quad \mbox{large}
      \end{array}  \right. \right\}
\end{equation*}

Since $|x_1-x_3|$ is small, $[x_1,x_1+\gamma]\cup[x_3,x_3+\gamma]$
is a single interval of length $\gamma+|x_1-x_3|$.  Thus
$F([x_1,x_1+\gamma]\cup[x_3,x_3+\gamma]) =
E_2(0;\gamma+|x_1-x_3|)$, which is easily calculated and compared
to $E_2(0;\gamma)$.

When intervals $I_1$ and $I_2$ are well separated, it should be
the case that
\begin{equation*}
  F(I_1\cup I_2) \approx F(I_1) \cdot F(I_2),
\end{equation*}
and similarly for more than two intervals.  We call this
{\it approximate splitting}.  Since $|x_1-x_2|$ and $|x_1-x_4|$ are
large, we have three well-separated intervals. Thus,
\begin{eqnarray*}
  & & F\left([x_1,x_1+\gamma]\cup[x_2,x_2+\gamma]
        \cup[x_3,x_3+\gamma]\cup[x_4,x_4+\gamma]\right)\\
  & & \quad \approx F\left([x_1,x_1+\gamma]\cup[x_3,x_3+\gamma]\right)
    \cdot F([x_2,x_2+\gamma]) \cdot F([x_4,x_4+\gamma]) \\
  & & \quad \approx E_2(0;\gamma+|x_1-x_3|) \times
    E_2(0;\gamma) \times E_2(0;\gamma).
\end{eqnarray*}
The integral over $\Omega_{(13|2|4)}$ can then be approximated by
a product of three more easily evaluated integrals.

We call the readers attention to~\cite{widom-fredholm-multiple}.  It seems that
these asymptotics are for unions of intervals like $[0,n]\cup[2n,3n]$.  That
is, the ratios of length and separation remain fixed and $x\to\infty$.  For approximate splitting,
however, we require asymptotics where the intervals are of fixed length and
they are moved far apart.

Before attempting to show that the remainder of $\tilde{I}^k$
makes a negligible contribution to $E(\X^k)$, one would first have
to quantify ``large'' and ``small'' in the definition of
$\Omega_{(13|2|4)}$ and the in the regions corresponding to other
partitions.  The definition of ``small'' would depend on how
rapidly $E_2(0;\gamma+\epsilon)$ decays.  The definition of
``large'' would depend on the details of the approximate splitting
conjecture.

Here is one approach to proving the approximate splitting
\begin{equation*}
  F(I_1\cup I_2) \approx F(I_1) \cdot F(I_2).
\end{equation*}
Let $\alpha_{12,i}$ be the eigenvalues of $A_{I_1\cup I_2}$,
$\alpha_{1,j}$ the eigenvalues of $A_{I_1}$, and $\alpha_{2,k}$
the eigenvalues of $A_{I_2}$.  We must show that
\begin{equation*}
  \prod_{i}(1-\alpha_{12,i}) \approx
    \prod_{j}(1-\alpha_{1,j})
    \prod_{k}(1-\alpha_{2,k}).
\end{equation*}
Let us take for granted that all of the $\alpha$ are positive and
less than $1$; perhaps this is something that could be proven.

We expect that for the large $\alpha_{1,j}$ and $\alpha_{2,k}$
there correspond $\alpha_{12,i}$.  If this were true for all
$\alpha$, it would prove approximate separation.  Regarding the
$1-1$ correspondence, there is a subtlety which we illustrate
in Figure~\ref{eigenvalueCorrespondence}.

\begin{figure}
  \begin{center}
    \includegraphics[scale=0.6]{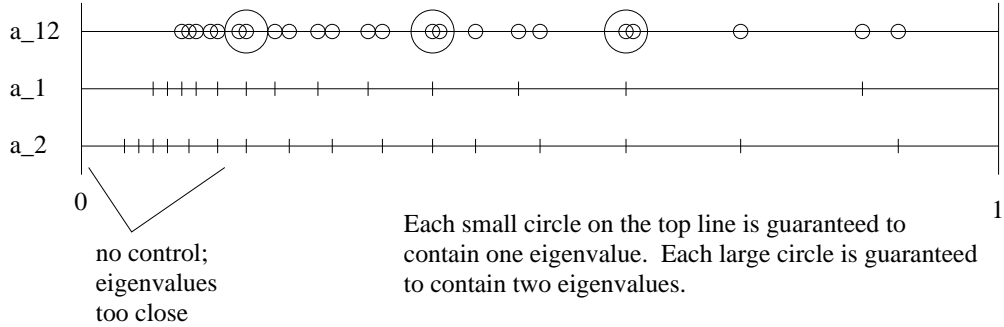}
    \caption{\label{eigenvalueCorrespondence}
        The correspondence between eigenvalues of $A_{I_1}$, $A_{I_2}$,
        and $A_{I_1\cup I_2}$.  When the circles are disjoint, there
        is no danger of confusion and we have a one-to-one correspondence
        of eigenvalues.}
  \end{center}
\end{figure}

It seems plausible that the large $\alpha_{1,j}$ are well
separated from other $\alpha_{1,j'}$, and that this could be
proven.  Similarly for the $\alpha_{2,k}$.  However, there is no
reason why one of the $\alpha_{1,j}$ cannot be close to one of the $\alpha_{2,k}$.
If we only have assurances that one eigenvalue $\alpha_{12,i}$ is nearby,
then we are missing one.
We therefore need a multiple eigenvalue version of Weyl's Criterion:

{\bf Weyl's Criterion:} Let $T$ be a self adjoint operator on a Hilbert space and $\psi$
and $\lambda$ such that
$|T\psi - \lambda\psi | \leq \epsilon |\psi|$,
where $\lambda$ is real.  Then the intersection of the spectrum of $T$ with
the interval $[\lambda-\epsilon,\lambda+\epsilon]$ is nonempty.

Embarrassingly, the author has been unable to find a reference for Weyl's
Criterion above, or the following statement.  We do not know whether the following
statement is true as stated or whether or not it is known.

{\bf Statement:} For $\epsilon>0$ sufficiently small, there is a constant
$C_{\epsilon}$ so that the following is true.
Suppose that $T$ is a self-adjoint operator on a Hilbert space $\HH$.
Suppose $\psi_1,\psi_2 \in \HH$ and $\lambda\in\RR$ are such that:
\begin{eqnarray*}
  & & |T\psi_1 - \lambda\psi_1 | \leq \epsilon |\psi_1 | \\
  & & |T\psi_2 - \lambda\psi_2 | \leq \epsilon |\psi_2 | \\
  & & |(\psi_1,\psi_2)| \leq \epsilon |\psi_1|\cdot |\psi_2|
\end{eqnarray*}
Then the interval $[\lambda - C_{\epsilon},\lambda + C_{\epsilon}]$
contains at least two eigenvalues of $T$.  The constants $C_{\epsilon}$
may be chosen so that $C_{\epsilon} \to 0$ as $\epsilon \to 0$.

To apply the multiple eigenvalue statement to the current situation, $T=A_{I_1\cup I_2}$,
and $\lambda$ be the average value of two nearby eigenvalues $\alpha_{1,j}$ and $\alpha_{2,k}$.
Let $\psi_1$ and $\psi_2$ be the corresponding eigenfunctions for $A_{I_1}$ and $A_{I_2}$.  Then
$\psi_1$ should be localized around $I_1$ and $\psi_2$ should be localized around $I_2$.  Hence
$\psi_1$ and $\psi_2$ are nearly orthogonal, and also $A_{I_2} \psi_1$ and $A_{I_1} \psi_2$ are
small, satisfying the hypotheses of the multiple eigenvalue statement.

Thus for large $\alpha$, there is a one to one correspondence between eigenvalues of
either $A_{I_1}$ or $A_{I_2}$ and eigenvalues of $A_{I_1\cup I_2}$.  For small $\alpha$,
however, this correspondence may no longer hold.  Fortunately the small $\alpha$ cannot
contribute much to any of the products $\prod(1-\alpha)$.  The reason is that
$A_{I_1\cup I_2}$ has trace $|I_1| + |I_2|$ and also has, we hope, all positive
eigenvalues.  The large eigenvalues $\alpha_{12,i}$ have accounted for all but a small
amount of this sum, and so the product $\prod (1-\alpha_{12,i})$ over the remaining $\alpha$ is
close to one.  Similarly for the other two products.

%***************************************************************
\bibliographystyle{alpha}
\bibliography{main}

\newcommand{\etalchar}[1]{$^{#1}$}
\begin{thebibliography}{DKM{\etalchar{+}}99b}

\bibitem[BO78]{bender}
Carl~M. Bender and Steven~A. Orszag.
\newblock {\em Advanced mathematical methods for scientists and engineers}.
\newblock McGraw-Hill Book Co., New York, 1978.
\newblock International Series in Pure and Applied Mathematics.

\bibitem[Dei99]{deift_book}
P.~A. Deift.
\newblock {\em Orthogonal polynomials and random matrices: a
  {R}iemann-{H}ilbert approach}.
\newblock New York University Courant Institute of Mathematical Sciences, New
  York, 1999.

\bibitem[DKM{\etalchar{+}}97]{deiftzhou4}
P.~Deift, T.~Kriecherbauer, K.~T-R McLaughlin, S.~Venakides, and X.~Zhou.
\newblock Asymptotics for polynomials orthogonal with respect to varying
  exponential weights.
\newblock {\em Internat. Math. Res. Notices}, (16):759--782, 1997.

\bibitem[DKM98]{deiftzhou10}
P.~Deift, T.~Kriecherbauer, and K.~T.-R. McLaughlin.
\newblock New results on the equilibrium measure for logarithmic potentials in
  the presence of an external field.
\newblock {\em J. Approx. Theory}, 95(3):388--475, 1998.

\bibitem[DKM{\etalchar{+}}99a]{deiftzhou1}
P.~Deift, T.~Kriecherbauer, K.~T-R McLaughlin, S.~Venakides, and X.~Zhou.
\newblock Strong asymptotics of orthogonal polynomials with respect to
  exponential weights.
\newblock {\em Comm. Pure Appl. Math.}, 52(12):1491--1552, 1999.

\bibitem[DKM{\etalchar{+}}99b]{deiftzhou2}
P.~Deift, T.~Kriecherbauer, K.~T.-R. McLaughlin, S.~Venakides, and X.~Zhou.
\newblock Uniform asymptotics for polynomials orthogonal with respect to
  varying exponential weights and applications to universality questions in
  random matrix theory.
\newblock {\em Comm. Pure Appl. Math.}, 52(11):1335--1425, 1999.

\bibitem[Dys76]{dyson-fredholm}
Freeman~J. Dyson.
\newblock Fredholm determinants and inverse scattering problems.
\newblock {\em Comm. Math. Phys.}, 47(2):171--183, 1976.

\bibitem[Ede92]{edelman-condition}
Alan Edelman.
\newblock On the distribution of a scaled condition number.
\newblock {\em Math. Comp.}, 58(197):185--190, 1992.

\bibitem[Meh91]{mehta}
Madan~Lal Mehta.
\newblock {\em Random matrices}.
\newblock Academic Press Inc., Boston, MA, second edition, 1991.

\bibitem[Sze75]{szego}
G{\'a}bor Szeg{\H{o}}.
\newblock {\em Orthogonal polynomials}.
\newblock American Mathematical Society, Providence, R.I., fourth edition,
  1975.
\newblock American Mathematical Society, Colloquium Publications, Vol. XXIII.

\bibitem[Wid94]{widom-orthopoly}
Harold Widom.
\newblock The asymptotics of a continuous analogue of orthogonal polynomials.
\newblock {\em J. Approx. Theory}, 77(1):51--64, 1994.

\bibitem[Wid95]{widom-fredholm-multiple}
Harold Widom.
\newblock Asymptotics for the {F}redholm determinant of the sine kernel on a
  union of intervals.
\newblock {\em Comm. Math. Phys.}, 171(1):159--180, 1995.

\end{thebibliography}
%***************************************************************

%***************************************************************
\end{document}